\documentclass[twoside,11pt]{article}

\usepackage[abbrvbib]{arXivCamera}
\usepackage{standalone}
\usepackage{amsmath,amssymb,amsfonts}
\usepackage{url}
\usepackage{enumerate}
\usepackage[subnum]{cases}
\usepackage{import}
\usepackage{enumitem} % For alphabet labeling of lists
\usepackage{tikz}
\usepackage{graphicx,epstopdf}
\usepackage[caption=false]{subfig} 
\usepackage[capitalise]{cleveref}
\usepackage{algorithmic}
\usepackage{algorithm}
\usepackage{color}
\usepackage{multirow}
\usepackage{rotating} % For sideways tables
\usepackage{amsopn}
\usepackage[pagebackref=true]{hyperref}
 \renewcommand*\backref[1]{\ifx#1\relax \else (cited on #1) \fi}

\newcommand{\qed}{\tag*{$\blacksquare$}}

\newcommand{\R}{{\rm I\!R}}

\newcommand{\prox}[1]{\hbox{prox}_{#1}}
\newcommand{\proj}[1]{\hbox{proj}_{#1}}
\newcommand{\argmin}[1]{\mathop{\hbox{arg\,min}}_{#1}}
\newcommand{\argmax}[1]{\mathop{\hbox{arg\,max}}_{#1}}

\newcommand{\interior}{\mathop{\hbox{int}}}

\newcommand{\B}{\mathcal{B}}

\def\half{\frac 1 2}
\newcommand{\dom}[1]{\hbox{{ dom\;}}{(#1)}}
\definecolor{blu}{rgb}{0,0,1}
\definecolor{blue}{rgb}{0,0,1}
\def\blu#1{{\color{black}#1}}
\def\blue#1{{\color{black}#1}}

\newtheorem{assumption}{Assumption}

\newcommand{\adaptive}{variable~}
\newcommand{\Adaptive}{Variable~}
\newcommand{\comm}[1]{\hfill \# #1}

\def\norm#1{\|#1\|}

\tikzset{
  fontscale/.style = {font=\relsize{#1}},
  level 1/.style={sibling distance=8mm,level distance=1cm},
  level 2/.style={sibling distance=8mm,level distance=1cm},
  level 3/.style={sibling distance=10mm,level distance=1cm},
  treenode/.style = {align=center, text centered,minimum size = 0.6cm},
  arn_n/.style = {treenode, font=\footnotesize,circle, black, draw=black,very thick},% arbre rouge noir, noeud noir
  arn_r/.style = {treenode, font=\footnotesize,circle, red, draw=red, very thick},% arbre rouge noir, noeud rouge
  arn_x/.style = {treenode, font=\footnotesize,circle, dashed, black, draw=black, very thick},% arbre rouge noir, noeud rouge
}

\usepackage{lastpage}
\jmlrheading{23}{2022}{1-\pageref{LastPage}}{1/18; Revised 9/20, 12/21, 2/22, 3/22, and
4/22}{4/22}{18-045}{Julie Nutini, Issam Laradji and Mark Schmidt}

\ShortHeadings{Let's Make Block Coordinate Descent Converge Faster}{Nutini, Laradji and Schmidt}
\firstpageno{1}

\begin{document}

\title{Let's Make Block Coordinate Descent \blue{Converge} Faster:\\Faster Greedy Rules, Message-Passing,\\ Active-Set Complexity, and Superlinear Convergence}

\author{\name Julie Nutini \email julie.nutini@gmail.com\\ 
\addr University of British Columbia\\
\name Issam Laradji \email issamou@cs.ubc.ca\\ 
\addr University of British Columbia, ServiceNow Research\\
\name Mark Schmidt \email schmidtm@cs.ubc.ca\\ 
\addr University of British Columbia, Canada CIFAR AI Chair (Amii)
}

\editor{Ryota Tomioka}

\maketitle

\begin{abstract}%   <- trailing '%' for backward compatibility of .sty file
Block coordinate descent (BCD) methods are widely used for large-scale numerical optimization because of their cheap iteration costs, low memory requirements, amenability to parallelization, and ability to exploit problem structure. Three main algorithmic choices influence the performance of BCD methods: the block partitioning strategy, the block selection rule, and the block update rule. In this paper we explore all three of these  building blocks and propose variations for each that can \blue{significantly improve the progress made by each BCD iteration}. We (i) propose new greedy block-selection strategies that guarantee more progress per iteration than the Gauss-Southwell rule; (ii) explore practical issues like how to implement the new rules when using ``variable'' blocks; (iii)  explore the use of message-passing to  compute matrix or Newton updates efficiently on huge blocks for problems with sparse dependencies between variables; and (iv) consider optimal active manifold identification, which leads to bounds on the ``active-set complexity'' of BCD methods and leads to superlinear convergence for certain problems with sparse solutions (and in some cases finite termination at an optimal solution). We support all of our findings with numerical results for the classic machine learning problems of least squares, logistic regression, multi-class logistic regression, label propagation, and L1-regularization.
\end{abstract}

\begin{keywords}
  Convex optimization, convergence rates, block coordinate descent, greedy algorithms, large-scale
\end{keywords}

%\documentclass{standalone}
%\begin{document}

\section{Introduction}
\label{sec:introduction}

Block coordinate descent (BCD) methods have become one of the key tools for solving some of the most important large-scale optimization problems. This is due to their typical ease of implementation, low memory requirements, cheap iteration costs, adaptability to distributed settings, ability to use problem structure, and numerical performance. Notably, they have been used for almost two decades in the context of L1-regularized least squares (LASSO)~\citep{fu1998penalized,sardy2000} and support vector machines (SVMs)~\citep{libsvm, svmlight}. Indeed, randomized BCD methods have recently been used  to solve instances of these widely-used models with a billion variables~\citep{richtarik2011}, while for ``kernelized'' SVMs \emph{greedy} BCD methods remain among the state of the art methods~\citep{you2016asynchronous}. Due to the wide use of these models, any improvements on the convergence of BCD methods will affect a myriad of applications.

\blue{In recent work, we have given a refined analysis of greedy coordinate descent in the special case where we are updating a single variable~\citep{nutini2015}. This work supports the use of greedy coordinate descent methods over randomized methods in the many scenarios where the iterations costs of these methods are similar. Other authors have subsequently shown a variety of interesting results related to greedy coordinate descent algorithms~\citep{locatello2018matching,song2017accelerated,stich2017approximate,lu2018accelerating,karimireddy2019efficient,fang2020greed}. These works also focus on the case of updating a single coordinate, and do not explore the large design space of possible choices we could make when updating multiple coordinates in a block.\footnote{Some recent works have considered ways to split the variables into blocks, but then apply the single-coordinate greedy update rules (in some cases in parallel) to the blocks~\citep{you2016asynchronous,fang2021efficient}. Our focus is on updating the variables as a block to obtain more progress at each iteration.} An exception to this is~\citet{csiba2017global}, which was contemporaneous with the first version of this work and whose ``greedy minibatch'' strategy is obtained as special case of a greedy BCD rule we present in Section~\ref{sec:gsl}.} 

\blue{The work described in this paper is motivated by the plethora of choices available when implementing greedy BCD. In this work we survey a variety of choices that are available in the literature, present new choices for several problem structures, and present both theoretical and empirical evidence that some choices are superior to others. By considering the conclusions from this paper and a particular problem structure that a reader is interested in, we expect that readers should be able to develop substantially-faster BCD algorithms for many applications. Further, some of the lessons we learned from this work can also be used to give faster algorithms in other settings like cyclic and  randomized BCD.
}

While there are a variety of ways to implement a BCD method, the three main building blocks that affect its performance are:
\begin{enumerate}
	\item {\bf Blocking strategy}. We need to define a set of possible ``blocks'' of problem variables that we might choose to update at a particular iteration. Two common strategies are to form a partition of the coordinates into disjoint sets (we call this \emph{fixed} blocks) or to consider any possible subset of coordinates as a ``block'' (we call this \emph{variable} blocks). Typical choices include using a set of fixed blocks related to the problem structure, or using \adaptive blocks by randomly sampling a fixed number of coordinates.
	\item {\bf Block selection rule}. Given a set of possible blocks, we need a rule to select a block of corresponding variables to update. Classic choices include cyclically going through a fixed ordering of blocks, choosing random blocks, choosing the block with the largest gradient (the \emph{Gauss-Southwell} rule), or choosing the block that leads to the largest improvement.	
%	r greedily choosing the ``best'' block to update based on gradients-norms or choo maximum-improvement rule.%as the one with the largest gradient norm (known as the \emph{Gauss-Southwell} rule).
	\item {\bf Block update rule}. Given the block we have selected, we need to decide how to update the block of corresponding variables. Typical choices include performing a gradient descent iteration, computing the Newton direction and performing a line-search, or computing the optimal update of the block by subspace minimization.
%	 When considering coordinate descent with a block size of 1, the question of what block update to use reduces to using a constant step size or using a line search method with a simple gradient step. With blocks of size greater than 1 there are various other update rules that can be used to exploit the structure of the subproblem defined by the selected block. The efficacy of these methods depends heavily on the form of the subproblem and if a suitable efficient solver exists.   \\
\end{enumerate}

In Section~\ref{sec:bcd} we introduce our notation, review the standard choices behind BCD algorithms, and discuss problem structures where BCD is suitable.
%discuss some important classes of problems where it makes sense to use BCD.
% the standard approaches used for each of the three building blocks listed above and the specific problem structures  that are suitable for using BCD methods. 
%Subsequently, the next sections propose four ways to speed up BCD by modifying the three building blocks above.
Subsequently, the following sections explore a wide variety of ways to speed up BCD by modifying the three building blocks above. \blue{In particular, we first consider the implementation issues associated with using greedy BCD methods for general unconstrained optimization:}
\begin{itemize}
	\item In \Cref{sec:gsl} \blue{we propose selection rules that generalize the Gauss-Southwell-Lipschitz rule proposed in~\citet{nutini2015}. The Gauss-Southwell-Lipschitz rule incorporates knowledge of Lipschitz constants into the classic Gauss-Southwell rule in order to give a better bound on the progress made at each iteration. The prior work considered this rule in the setting of updating a single coordinate, and in this work we propose several generalizations to the block setting.} We also show a result characterizing the convergence rate obtained using the Gauss-Southwell rule as well as the new greedy rules, under both the Polyak-\L{}ojasiewicz inequality and for general (potentially non-convex) functions.
	\item In \Cref{sec:practical} we discuss a variety of the practical implementation issues associated with implementing BCD methods. This includes how to approximate the new rules in the variable-block setting, how to estimate the Lipschitz constants, how to efficiently implement line-searches, how the blocking strategy interacts with greedy rules, and why we should prefer Newton updates over the ``matrix updates'' of recent works.	\blue{While we specifically discuss these in the context of greedy methods, several of the ideas in this section also apply when implementing cyclic or randomized BCD methods.}
%	. We further propose a block update  rule based on the cubic regularization framework, which incorporates second-order information without the extra cost and implementation effort of a line search. 
\end{itemize}
\blue{In the next sections we consider two of the most-common cases where BCD methods are used. In particular, we consider problems with a sparse dependency structure and problems that include a non-smooth but separable term. In this context we make the following contributions:}
\begin{itemize}
\item In \Cref{sec:messagepassing} we show how second-order updates, or the exact update for quadratic functions, can be computed in linear-time for problems with sparse dependencies when using ``forest-structured'' blocks. This allows us to use huge block sizes for problems with sparse dependencies, and uses a connection between sparse quadratic functions and Gaussian Markov random fields (GMRFs) by exploiting the ``message-passing'' algorithms developed for GMRFs. \blue{We discuss how to efficiently construct forest-structured blocks in both randomzied and greedy settings.}
	\item In \Cref{sec:exactupdates} \blue{we give bounds on the number of iterations required to reach the optimal manifold (``active-set complexity'') when applying greedy (or cyclic) BCD methods to problems with a separable non-smooth structure.} We also discuss how when using greedy rules with \adaptive blocks this leads to local superlinear convergence  for problems with sufficiently-sparse solutions (when we use updates incorporating second-order information). \blue{This leads to finite convergence for problems with sufficiently-sparse solutions and where exact updates are possible (such as SVMs and LASSO problems), when using greedy selection and sufficiently-large variable blocks. We further discuss how the old idea of two-metric projection can be used to obtain this finite convergence while maintaining the cost of BCD methods for smooth objectives.}
%	For sparse quadratic problems, we show how to use message-passing to implement BCD updates with tree-structured blocks, allowing us to update a huge number of variables in one iteration. \\ %\cred{We also show how this can be used to implement Newton-style updates for non-quadratic problems defined on graphs}. \\
\end{itemize}
We note that many related ideas have been explored by others in the context of BCD methods and we will go into detail about these related works in the relevant sections. 
In \Cref{sec:experiments} we use a variety of problems arising in machine learning to evaluate the effects of these choices on BCD implementations. These experiments indicate that in some cases the performance improvement obtained by using these enhanced methods can be dramatic. 

 \blue{Below we list 7 hypotheses that our work supports regarding the implementation of BCD methods, along with the relevant theoretical and/or empirical results in this work. When implementing BCD methods, we recommend to consider each of these issues and to follow the guideline provided that it does not substantially increase the iteration cost.}
\blue{
\begin{enumerate}
\item Similar to the the case of updating a single coordinate, BCD methods tend to converge faster when using greedy rules than random or cyclic rules. Thus, we should prefer greedy methods in cases where their iteration costs are similar to those of cyclic/random methods. We discuss these classic block selection rules in Section~\ref{subsec:select} and we discuss problem structures where random and greedy methods can be efficiently implemented in Sections~\ref{sec:problems}-\ref{subsec:examples} (with a detailed example in~\ref{app:sparsesoftmax}). We give a bound showing why greedy methods may converge faster in Section~\ref{subsec:boundsprogress}, while empirical evidence for the advantage of greedy methods is seen in Sections~\ref{sec:expGrad},~\ref{exp:message}-\ref{sec:experimentsProx}, and \ref{append:greedyGradient}.
\item All BCD methods tend to converge faster as the block size increases. Thus, it is beneficial to use larger blocks in cases where this does not increase the iteration cost. The main reasons why it might not increase the cost are the use of parallel computation, or the presence of special problem structures that make updating particular blocks of variables have the same cost as updating a single variable (we give an example in \ref{app:sparsesoftmax}). Section~\ref{subsec:boundsprogress} gives a bound showing why increasing the block size tends to improves the progress per iteration, while empirical evidence for faster convergence as we increase the block size is seen in Sections~\ref{sec:expGrad}-\ref{exp:message},~\ref{append:greedyGradient}, and \ref{append:greedyNewtonMatrix}.
%\item The performance differences between different ways to select the block decreased as we increase the block size. Thus, designing clever block selection strategies is most important when we use small blocks and is less important when we use large blocks. We see this empirically in Sections~\ref{sec:expGrad}-\ref{sec:expNewt},~\ref{append:greedyGradient}, and~\ref{append:greedyNewtonMatrix}. The only exception we saw to this was for non-smooth functions,  where in Section~\ref{sec:experimentsProx} we see the differences between some rules can sometimes increasing with the block size.
\item When using greedy selection rules, using variable blocks can substantially outperform using fixed blocks. When using randomized rules, we observed the opposite trend: we found that using fixed blocks tended to outperform variable blocks within randomized BCD methods. We discuss fixed and variable blocks in Section~\ref{subsec:variable}. Section~\ref{subsec:boundsprogress} discusses why variable blocks lead to a better progress bound on each iteration for greedy methods, and this is empirically supported by the experiments in Sections~\ref{sec:expGrad}-~\ref{sec:expNewt}, \ref{append:greedyGradient}, an \ref{append:greedyNewtonMatrix}
\item We found that using the new block versions of the Gauss-Southwell-Lipschitz rules can lead to substantially-faster convergence than the classic Gauss-Southwell rule. This gain was much larger than the modest improvements seen from the Gauss-Southwell-Lipschitz in prior work that only considered the single-coordinate case. On the other hand, we did not observe large differences between the various block generalizations of the Gauss-Southwell-Lipschitz rule, indicating that we can use the computationally-cheaper variants of the rule. These new Gauss-Southwell-Lipschitz rules are introduced in Sections~\ref{sec:BGSL}-\ref{subsec:fixed}, which discuss why these rules tend to make more progress per iteration than the classic GS rule. Formal convergence rates are given in Sections~\ref{subsec:convergence}-\ref{subsec:nonconvex}. We discuss how to efficiently implementing this type of rule in Section~\ref{subsec:adaptive} (while~\ref{append:blockLipschitz} discusses how to compute Lipschitz constants for some common problems). Empirical support for the advantages of the new rules is seen in Sections~\ref{sec:expGrad}-\ref{sec:expNewt}, \ref{append:greedyGradient}, and \ref{append:greedyNewtonMatrix}.\footnote{When using fixed blocks, we found that using Lipschitz constants to help partition the variables can lead to improved performance. In particular, we obtained the best performance by sorting the Lipschitz constants and constructing blocks based on the quantiles (so that the variables with the largest Lipschitz constants would be in the same group, and the variables with the smallest Lipschitz constants would in the same group). This often outperformed using a random assignment to the fixed groups, or choosing blocks that had small average Lipschitz constants. We discuss partition strategies in Section~\ref{sec:part}, and empirical support for this is seen in~\ref{append:greedyGradient} (though we do not have theoretical results explaining why this empirical difference was observed).}
\item We found that updates that scale the gradient by a matrix tend to outperform updates that use the gradient direction. Further, applying Newton's method to the block tended to outperform the fixed matrix scalings used in several recent works. We discuss matrix updates and Newton's method in Section~\ref{subsec:update}. Section~\ref{sec:BGSQ} shows why matrix updates tend to make more progress per iteration, with convergence rates given in Sections~\ref{subsec:convergence}-\ref{subsec:nonconvex}. The difference between using gradient updates and matrix updates can be observed by comparing the experiments in Section~\ref{sec:expGrad} to the experiments in Section~\ref{sec:expNewt}, and by comparing the experiments in \ref{append:greedyGradient} to the experiments in \ref{append:greedyNewtonMatrix}. 
Section~\ref{sec:newton} discusses why Newton updates should give a further improvement over matrix updates, and empirical evidence for the advantage of Newton updates is given in \ref{append:greedyNewtonMatrix}.
\item For problems with sparse dependencies, we found that using large forest-structured blocks can give a substantial peformance improvement over unstructured blocks (which must be much smaller to have a comparable cost). This applied whether we used random or greedy methods to construct the forest-structured blocks. Sections~\ref{sec:newtonCost}-\ref{subsec:solvingforestsBCD} discuss how to implement forest-structured updates in linear time, Sections~\ref{sec:redblack}-\ref{sec:treeApprox} and~\ref{app:forest} discuss how to efficiently construct forest-structured blocks, and Section~\ref{exp:message} shows experimentally that forest-structured blocks can lead to a huge performance gain.
\item For non-smooth problems with sparse solutions, the performance difference between random and greedy methods can be even larger than for smooth problems. This is because the greedy methods tend to focus on updating non-zero variables. Sections~\ref{subsec:mani}-\ref{sec:maniRate} show the speed at which greedy methods identify the set of non-zero variables, while Sections~\ref{sec:superlinear}-\ref{subsec:optQuad} discuss rules and updates that lead to superlinear convergence for certain non-smooth problems when using greedy rules with variable blocks. Section~\ref{sec:experimentsProx} shows experimentally the performance gains obtained by using variable blocks and rules/updates that give superlinar convergence, while~\ref{append:activeSetExp} shows the relative innefficiency of using randomized updates for this problem setting. Further, for obtaining superlinear convergence on non-smooth problems, the efficient two-metric projection strategies for updating the blocks has nearly the same convergence rate as the more-expensive projected-Newton updates. The two-metric projection method is discussed in Section~\ref{sec:tmp}, while Section~\ref{sec:experimentsProx} shows that the method has essentially-identical performance in practice to using more-expensive rules.
\end{enumerate}
We also found that using a line-search or local approximations of the Lipschitz constants led to faster convergence than using global upper bounds on these constants. This was particularly noticeable as we increased the block size. We discuss how to use local approximations of the Lipschitz constants in Section~\ref{sec:LA}, and how to efficiently implement line-searches for notable problem structures in Section~\ref{sec:line-search}. This is supported empirically in \ref{append:greedyGradient} and\ref{append:greedyNewtonMatrix}. However, when using approximations to Lipschitz constants we need to be careful about how it affects selection rules that depend on the Lipschitz constants (we observe a case where using approximated Lipshitz information within the selection hurts performance in \ref{append:greedyGradient}).
}

We summarize a set of recommendations that our experiments support in the table below:\\
{\scriptsize
\begin{center}
\begin{tabular}{|l|l|}
\hline
Premise & Recommendation\\
\hline
Random BCD has similar cost to gradient descent. & Do not use BCD.\\
Greedy BCD has similar cost random BCD. & Use greedy selection.\\
BCD with large blocks has similar cost to small blocks. & Use bigger blocks.\\
Using random BCD. & Use fixed blocks.\\
Using greedy BCD, variable blocks have similar cost to fixed blocks. & Use variable blocks.\\
Using greedy BCD, any GSL rule has similar cost to GS rule. & Use GSL.\\
Using greedy BCD wth fixed blocks, and have Lipschitz constants. & Partition blocks using sort strategy.\\
Matrix updates have similar cost to gradient updates. & Do not use gradient updates.\\
Newton updates have similar cost to matrix updates. & Use Newton updates.\\
Using matrix/Newton updates and dependencies are very sparse. & Use forest-structured blocks.\\
Using matrix/Newton updates but have non-differentiable regularizer. & Use two-metric projection.\\
Function is not quadratic. & Set the step-size with a line-search.\\
\hline
\end{tabular}
\end{center}
}

The source code and data files required to reproduce the experimental results of this paper can be downloaded from: \url{https://github.com/IssamLaradji/BlockCoordinateDescent}.

% COntributions:
% - bigger is better, variable is better, some Lipschitz is better, trees are awesome, support set size, two-metric projection

%\end{document}

%\documentclass{standalone}
%\begin{document}

\section{Block Coordinate Descent Algorithms}
\label{sec:bcd}

We first consider the problem of minimizing a differentiable multivariate function,
\begin{equation}
\label{eq:generalobjective}
	\argmin{x \in \R^n} f(x).
\end{equation}
At iteration $k$ of a BCD algorithm, we first select a block $b_k \subseteq \{ 1, 2, \dots, n\}$ and then update the subvector $x_{b_k} \in \R^{|b_k|}$ corresponding to these variables,
\[
x^{k+1}_{b_k} = x^k_{b_k} + d^k.
\]
Coordinates of $x^{k+1}$ that are not included in $b_k$ are simply kept at their previous value. The vector $d^k \in \R^{|b_k|}$ is typically selected as a descent direction in the reduced dimensional subproblem,
\begin{equation}
d^k \in \argmin{d\in\R^{|b_k|}} f(x^k + U_{b_k}d),
\label{eq:subProb}
\end{equation}
where we construct $U_{b_k} \in \{ 0,1 \}^{n \times |b_k|}$ from the columns of the identity matrix corresponding to the coordinates in $b_k$. Using this notation, we have
\[
x_{b_k} = U_{b_k}^Tx,
\]
%Specifically, we construct $U_b$ from the columns of the identity matrix corresponding to the coordinates in $b$.
%, so $U_b$ has a 1 in position $(i,j)$ iff variable $i$ of $x$ is element $j$ of the block $x_b$.
which allows us to write the BCD update of all $n$ variables in the form
\[
x^{k+1} = x^k + U_{b_k}d^k.
\]
There are many possible ways to define the block $b_k$ as well as the direction $d^k$. Typically we have a maximum block size $\tau$, which is chosen to be the largest number of variables that we can efficiently update at once. Given $\tau$ that divides evenly into $n$, consider a simple ordered fixed partition of the coordinates into a set $\B$ of $n/\tau$ blocks,
\[
	\B = \{ \{1, 2, \dots, \tau \}, \{ \tau + 1, \tau + 2, \dots, 2\tau \}, \dots, \{(n - \tau) + 1, (n - \tau) + 2, \dots, n \} \}.
\]
%Given $\tau$, a simple way to choose the blocks is to put variables $1$ up to $\tau$ in the first block, variables $(\tau+1)$ through $2\tau$ in the second block, and so on. 
To select the block in $\B$ to update at each iteration we could simply repeatedly cycle through this set in order.
%updating the first block on the first iteration, the second block on the second iteration, and so on until we have updated the final block and then repeat to update the first block on the next iteration. 
A simple choice of $d^k$ is the negative gradient corresponding to the coordinates in $b_k$, multiplied by a scalar step-size $\alpha_k$ that is sufficiently small to ensure that we decrease the objective function. This leads to a gradient update of the form
\begin{equation}
x^{k+1} = x^k - \alpha_k U_{b_k}\nabla_{b_k} f(x^k),
\label{eq:grad}
\end{equation}
where $\nabla_{b_k} f(x^k)$ are the elements of the gradient $\nabla f(x^k)$ corresponding to the coordinates in $b_k$.
While this gradient update and cyclic selection among an ordered fixed partition is simple, we can often drastically improve the performance using more clever choices. We highlight some common alternative choices in the next three subsections.

\subsection{Block Selection Rules}
\label{subsec:select}

Repeatedly going through a fixed partition of the coordinates is known as \emph{cyclic} selection~\citep{bertsekas1999}, and this is referred to as Gauss-Seidel when solving linear systems \citep{ortega1970}. The performance of cyclic selection may suffer if the order the blocks are cycled through is chosen poorly, but it has been shown that random permutations can fix a worst case for cyclic CD~\citep{lee2017}.
%, and we note that it's not obvious how we would implement cyclic selection for \adaptive blocks. 
A variation on cyclic selection is ``essentially'' cyclic selection where each block must be selected at least every $m$ iterations for some fixed $m$ that is larger than the number of blocks~\citep{bertsekas1999}. Alternately, several authors have explored \emph{randomized} block selection% and provided convergence rate estimates in this setting
~\citep{nesterov2012efficiency,richtarik2011}. The simplest randomized selection rule is to select one of the blocks uniformly at random. However, several recent works show dramatic performance improvements over this naive random sampling  by incorporating knowledge of the Lipschitz continuity properties of the gradients of the blocks~\citep{nesterov2012efficiency,richtarik2013nsync,qu2015} or more recently by trying to estimate the optimal sampling distribution online~\citep{namkoong2017adaptive}.

An alternative to cyclic and random block selection is \emph{greedy} selection.  Greedy methods solve an optimization problem to select the ``best'' block at each iteration. A classic example of greedy selection is the \textbf{block Gauss-Southwell} (GS) rule, which chooses the block whose gradient has the largest Euclidean norm,
\begin{equation}
	b_k \in \argmax{b \in \mathcal{B}} \| \nabla_b f(x^k) \|,
	\label{eq:GS}
\end{equation}
where we use $\mathcal{B}$ as the set of possible blocks. This rule tends to make more progress per iteration in theory and practice than randomized selection~\citep{dhillon2011nearest,nutini2015}. Unfortunately, for many problems it is more expensive than cyclic or randomized selection. However, several recent works show that certain problem structures allow efficient calculation of GS-style rules~\citep{meshi2012convergence,nutini2015,fountoulakis2016exploiting,lei2016coordinate}, allow efficient approximation of GS-style rules~\citep{dhillon2011nearest,thoppe2014,stich2017approximate}, or allow other rules that try to improve on the progress made at each iteration~\citep{glasmachers2013accelerated,csiba2015,perekrestenko2017faster}. 

 The ideal selection rule is the maximum improvement (MI) rule, which chooses the block that decreases the function value by the largest amount. Notable recent applications of this rule include leading eigenvector computation~\citep{li2015convergence}, polynomial optimization~\citep{chen2012maximum}, and fitting Gaussian processes~\citep{bo2012}. While recent works explore computing or approximating the MI rule for quadratic functions~\citep{bo2012,thoppe2014},
%  While Bo \& Sminchisescu~\citep{bo2012} propose a greedy algorithm to approximate the MI rule for quadratic functions,
 in general the MI rule is much more expensive than the GS rule.%and other rules that make more progress than the GS rule tend to be much more expensive.

%Other greedy rules have also been explored~\citep{tseng2009,bo2012,scherrer2012,thoppe2014} such as the maximum improvement (MI) rule which chooses the block that decreases the function value by the largest amount. However, the MI rule and other rules that make more progress than the GS rule tend be much more expensive.

\subsection{Fixed vs. {\Adaptive}Blocks}
\label{subsec:variable}

While the choice of the block to update has a significant effect on performance, how we define the set of \emph{possible} blocks also has a major impact. Although other variations are possible, we highlight below the two most common {blocking strategies}:
 
\begin{enumerate} 
\item {\bf Fixed blocks}. This method uses a partition of the coordinates into disjoint blocks, as in our simple example above. This partition is typically chosen prior to the first iteration of the BCD algorithm, and this set of blocks is then held fixed for all iterations of the algorithm. We often use blocks of roughly equal size, so if we use blocks of size $\tau$ this method might partition the $n$ coordinates into $n/\tau$ blocks.
 %, where at each iteration a block selection rule is used to select one block to update. 
Generic ways to partition the coordinates include ``in order" as we did above~\citep{bertsekas1999}, % beatson2000radialbasis, sardy2000, li2007leastsquares, liu2009, xu2013} 
or using a random partition~\citep{nesterov2012efficiency}.
%,richtarik2011,fountoulakis2015}. 
Alternatively, the partition may exploit some inherent substructure of the objective function such as block separability~\citep{meier2008group}, the ability to efficiently solve the corresponding sub-problem~\eqref{eq:subProb} with respect to the blocks~\citep{sardy2000}, or based on the Lipschitz constants of the resulting blocks~\citep{thoppe2014,csiba2016importance}.
%conditioning of the Hessian submatrices defined by each block~\citep{needell2012} which is related to the above-mentioned Lipschitz constants.
\item {\bf {\Adaptive}blocks}.  
Instead of restricting our blocks to a pre-defined partition of the coordinates, we could instead consider choosing any of the $2^n - 1$ possible sets of coordinates as our block. In the randomized setting, this is referred to as ``arbitrary'' sampling~\citep{richtarik2013nsync,qu2015}. We say that such strategies use \emph{\adaptive}blocks because we are not choosing from a partition of the coordinates that is fixed across the iterations. Due to computational considerations, when using {\adaptive}blocks we typically want to impose a restriction on the size of the blocks.  For example, we could construct a block of size $\tau$ by randomly sampling $\tau$ coordinates without replacement, which is known as $\tau$-nice sampling~\citep{qu2015,richtarik2016parallel}. Alternately, we could include each coordinate in the block $b_k$ with some probability like $\tau/n$ (so the block size may change across iterations, but we control its expected size).
%We call such strategies \emph{\adaptive} because the blocks are not fixed across the iterations.
A version of the greedy Gauss-Southwell rule~\eqref{eq:GS} with \adaptive blocks would select the $\tau$ coordinates corresponding to the elements of the gradient with largest magnitudes~\citep{TsengYun2009}. This can be viewed as a greedy variant of $\tau$-nice sampling. While we can find these $\tau$ coordinates for the Gauss-Southwell rule easily given the gradient vector, computing the MI rule with \adaptive blocks is much more difficult. Indeed, while methods exist to compute the MI rule for quadratics with fixed blocks~\citep{thoppe2014}, with \adaptive blocks it is NP-hard to compute the MI rule and existing works resort to approximations~\citep{bo2012}.
%argue that computing the MI rule is NP-hard to calculate the MI block and subsequently they resort to building the block incrementally as an approximation to the NP-hard problem. 
%In this work, we propose new greedy rules that use Lipchitz continuity information to achieve better bounds than the Gauss-Southwell rule but that are not as computationally difficult to compute as the maximum improvement rule (\ref{sec:gslrules}).
%At each iteration, the BCD algorithm uses a block selection rule to {\em define} a set of coordinates to update. Unlike in the fixed block case where the block selection rule selects a block from the predefined set of blocks, the \adaptive block selection rule selects a block over {\em all possible blocks} satisfying some given condition(s).\\
\end{enumerate}

\subsection{Block Update Rules}
\label{subsec:update}

The selection of the update vector $d^k$ can significantly affect performance of the BCD method. For example, in the gradient update~\eqref{eq:grad} the method can be sensitive to the choice of the step-size $\alpha_k$. Classic ways to set $\alpha_k$ include using a fixed step-size (with each block possibly having its own fixed step-size), using an approximate line-search, or using the optimal step-size (which has a closed-form solution for quadratic objectives)~\citep{bertsekas1999}.

\blu{A generalization of the gradient update, where $d^k = -\alpha_k \nabla_{b_k} f(x^k)$, is what we will call a \textbf{matrix update} where
\begin{equation}\label{eq:matrix}
	d^k = - \alpha_k \left ( H_{b_k} \right )^{-1} \nabla_{b_k} f(x^k).
  \end{equation}
  Here, $H_{b_k}$ can be any positive-definite matrix. A notable special case is the \textbf{Newton update},
    \begin{equation}\label{eq:newton}
 	d^k = - \alpha_k\left ( \nabla^2_{b_kb_k} f(x^k) \right )^{-1} \nabla_{b_k} f(x^k),
 \end{equation}
where  $H_{b_k}$ is chosen as the instantaneous Hessian of the block $\nabla^2_{b_kb_k} f(x^k)$ (or a positive-definite approximation to this quantity)}.
    In this context $\alpha_k$ is again a step-size that can be chosen using similar strategies to those mentioned above. Several recent works analyze various forms of matrix updates and show that they can substantially improve the convergence rate~\citep{tappenden2016,qu2015sdna,fountoulakis2015flexible,rodomanov2020randomized,mutny2020convergence}.
 For special problem classes, another possible type of update is what we will call the \textbf{optimal update}. This update chooses $d^k$ to solve~\eqref{eq:subProb}. In other words, it updates the block $b_k$ to maximally decrease the objective function. 

\subsection{Problems of Interest}
\label{sec:problems}

\blu{BCD methods are not suitable for all problems, since for some problems computing the gradient with respect to a block of variables has the same cost as computing the full gradient vector. It is thus important to consider the sets of problem for which they are appropriate. BCD methods tend to be suitable when the cost of updating a block of size $|b_k|$ is $\tilde{O}((|b_k|/n)\kappa)$, where $\kappa$ is the cost of computing the full gradient.}
Two common classes of objective functions where single-coordinate descent methods $(|b_k|=1)$ tend to be suitable are
\begin{align*}
h_1(x)  := \sum_{i=1}^n g_i(x_i) + f(Ax),\quad \text{or} \quad
h_2(x)  := \sum_{i \in V}g_i(x_i) + \sum_{(i,j)\in E}f_{ij}(x_i,x_j),
\end{align*}
where $f$ is smooth and cheap, the $f_{ij}$ are smooth, $G=\{V,E\}$ is a
graph, and $A$ is a matrix. Examples of problems leading to functions of the form $h_1$ include least squares, logistic regression, LASSO, and SVMs.\footnote{Coordinate descent remains suitable for multi-linear generalizations of problem $h_1$ like functions of the form $f(XY)$ where $X$ and $Y$ are both matrix variables.} The most important example of problem $h_2$ is quadratic functions, which are crucial to many aspects of scientific computing.\footnote{Problem $h_2$ can be generalized to allow functions between more than 2 variables, and coordinate descent remains suitable as long as the expected number of functions in which each variable appears is $n$-times smaller than the total number of functions (assuming each function has a constant cost).}

Problems $h_1$ and $h_2$ are also suitable for BCD methods, as they tend to admit efficient block update strategies. In general, if single-coordinate descent is efficient for a problem, then BCD methods with gradient updates~\eqref{eq:grad} are also efficient for that problem and this applies whether we use fixed blocks or \adaptive blocks.\footnote{\blu{This tends to also be true with matrix and Newton updates, provided the blocks are small enough that dealing with the matrices does not significantly change the cost.}}
\blue{In \citet[][Appendix A]{nutini2015} we discuss additional assumptions on the sparsity of $A$ in $h_1$ and the graph structure in $h_2$ that allow greedy rules to be implemented efficiently.}
Other scenarios where coordinate descent and BCD methods have proven useful include matrix and tensor factorization methods~\citep{yu2012scalable,xu2013block}, problems involving log-determinants~\citep{scheinberg2009sinco,hsieh2013big}, and problems involving convex extensions of sub-modular functions~\citep{jegelka2013reflection,ene2015random}.
%For example, if we divide the variables into disjoint blocks of size $n/\tau$ then we tend to be able to update all blocks at a similar cost to computing the full gradient. The efficiency of BCD for problem classes $h_1$ and $h_2$ applies whether we use fixed blocks or \adaptive blocks.

An important point to note is that there are special problem classes where BCD with fixed blocks is reasonable even though using \adaptive blocks (or single-coordinate updates) would not be suitable. For example, consider a variant of problem $h_1$ where we use \emph{group} L1-regularization~\citep{bakin1999adaptive},
\begin{equation}
\label{eq:groupL1}
h_3(x) := \sum_{b \in \mathcal{B}} \norm{x_b} + f(Ax),
\end{equation}
where $\mathcal{B}$ is a partition of the coordinates. We cannot apply single-coordinate updates to this problem due to the non-smooth norms, but we can take advantage of the group-separable structure in the sum of norms and apply BCD using the blocks in $\mathcal{B}$~\citep{meier2008group,qin2013}.~\citet{sardy2000} in their early work on solving LASSO problems consider problem $h_1$ where the columns of $A$ are the union of a set of orthogonal matrices. By choosing the fixed blocks to correspond to the orthogonal matrices, it is very efficient to apply BCD. In~\ref{app:sparsesoftmax}, we outline how fixed blocks lead to an efficient greedy BCD method for the widely-used multi-class logistic regression problem when the data has a certain sparsity level.\footnote{\blu{Another structure where blocks have an advantage is in solving the SVM dual problem when an unregularized bias variable is used. In this case, there is an equality constraint in the dual which prevents any progress from being made by updating single coordinates. We do not consider equality constraints in this work, but we note that greedy BCD steps are used by the libSVM software~\citep{libsvm} which has been among the best solvers for this problem for around two decades~\citep{horn2018comparative}}.}

\subsection{Iteration Cost for Example Problems}
\label{subsec:examples}

\blue{The total time complexity of a BCD algorithm is given by the number of iterations multiplied by the cost per iteration. In this work we discuss a variety of ways to reduce the number of iterations. However, it is difficult to give generic advice on the use of BCD methods like statements of the form ``greedy BCD methods outperform cyclic BCD methods''. This is because the cost per iteration can vary widely with the problem structure. Indeed, for some problem structures the iteration cost of BCD methods means we should not use them at all.  }

\blue{Since we cannot know the specific problem(s) that the reader has in mind, in this section we illustrate how to use the iteration cost  and progress-per-iteration to choose between 3 potential algorithms in 3 different scenarios. The specific algorithms we consider are gradient descent, cyclic coordinate descent, and greedy coordinate descent with the GS rule (updating a single coordinate on each iteration, and assuming a constant step-size). Gradient descent will usually have the higest iteration cost among these 3 methods but tends to make the most progress per iteration, while cyclic coordinate descent will usually have the smallest iteration cost while making the least progress per iteration.}
\blue{
\begin{itemize}
\item \textbf{Dense quadratic functions}. Quadratic functions are a special case of problem $h_2$, and the use of a dense matrix leads to a complete graph $E$. In this setting, the iteration cost of gradient descent is $O(n^2)$ while the iteration cost of cyclic coordinate descent and greedy coordinate descent is $O(n)$~\citep[see][Appendix A]{nutini2015}. In this setting gradient descent is the worst strategy to use among the 3 methods, due to its high iteration cost, while greedy coordinate descent is the best strategy to use due to its low iteration cost and faster progress per iteration than cyclic coordinate descent.
\item \textbf{Sparse logistic regression}. Logistic regression is a special case of problem $h_1$, and by ``sparse'' we mean that $A$ has many values set to exactly zero. Using $z$ as the total number of non-zero values in $A$, the iteration cost of gradient descent is $O(z)$ while the average amortized iteration cost of cyclic coordinate is $O(z/n)$. For greedy coordinate descent, the worst-case iteration cost can range from $O(z/n)$ up to $O(z)$ depending on the sparsity pattern in $A$~\citep[see][Appendix A]{nutini2015}. If $A$ has an unfavourable sparsity pattern for greedy methods then cyclic coordinate descent is the best strategy due to its low iteration cost. Further, for unfavourable sparsity structures greedy coordinate descent is the worst strategy due to its high iteration cost and the fact that it typically makes less progress per iteration than gradient descent.
\item \textbf{Neural network with multiple hidden layers}. Neural networks do not naturally fit into problem class $h_1$ or $h_2$. Indeed, due to the structure of the gradient calculation, in the worst case computing a single partial derivative has the same cost as computing the entire gradient. Thus, the worst-case cost of cyclic and greedy coordinate descent methods is the same as the cost of gradient descent methods. Thus, in this setting gradient descent is the best among the three options. In contrast, cyclic coordinate descent would be the worst choice in this setting (it has the same iteration cost as gradient descent or greedy coordinate descent but makes less progress per iteration).
\end{itemize}
}
\blue{
As these scenarios illustrate, we should consider gradient descent over coordinate descent methods for problems where computing partial derivatives has the same cost as computing the entire gradient. We should consider cyclic coordinate descent methods when these updates are approximately $O(n)$-times cheaper than gradient descent steps but where greedy rules remain expensive. And finally we should consider greedy coordinate descent in settings where these updates are approximately $O(n)$-times cheaper than gradient descent updates.\footnote{Random coordinate descent with uniform sampling tends to have the same amortized iteration cost as cyclic coordinate descent. Random coordinate descent with Lipschitz sampling can have the same cost as greedy coordinate descent in the worst case.} 
%In this work we give various ways to improve the convergence speed, and we recommend using these methods in scenarios where it does not significantly change the iteration cost.
In this work we give various ways to improve the convergence speed of BCD methods. We recommend using these methods in scenarios where the problem structure supports the use of BCD methods (as discussed above) and where the modifications that lead to faster convergence rates do not significantly change the  iteration cost.
}

%\end{document}

\section{Improved Greedy Rules}
\label{sec:gsl}

\blue{If we want to speed up the convergence of BCD methods, it is natural to consider the GS rule since it tends to make more progress per iteration than cyclic and randomized rules.}
Previous works have identified that the greedy GS rule can lead to suboptimal progress, and have proposed rules that are closer to the MI rule for the special case of quadratic functions~\citep{bo2012,thoppe2014}. However, for non-quadratic functions it is not obvious how we should approximate the MI rule.
%The ideal scenario would be to use the MI rule, but typically this is not computationally tractable. 
As an intermediate between the GS rule and the MI rule for general functions, a new rule known as the Gauss-Southwell-Lipschitz (GSL) rule has recently been introduced~\citep{nutini2015}. The GSL rule was proposed in the case of single-coordinate updates, and is a variant of the GS rule that incorporates Lipschitz information to guarantee more progress per iteration. The GSL rule is equivalent to the MI rule in the special case of quadratic functions, so either rule can be used in that setting. However, the MI rule involves optimizing over a subspace which will typically be expensive for non-quadratic functions. After reviewing the classic block GS rule, in this section we consider several possible block extensions of the GSL rule that give a better approximation to the MI rule without requiring subspace optimization.

%showed that we can do better than the classic greedy selection rule formally known as the Gauss-Southwell (GS) selection rule by using the Lipschitz constants of the gradient of the objective function. This section explores extending their Gauss-Southwell-Lipschitz (GSL) rule to the block case. First we present some general notation that will be used in the subsequent sections of the paper and then define the classic GS rule, as well as the naive extension of this selection rule to the block case. We note that other greedy selection strategies have been discussed by Tseng and Yun~\citep{tseng2009} but they weren't prescriptive. 

\subsection{Block Gauss-Southwell}
\label{subsec:boundsprogress}

When analyzing BCD methods we typically assume that the gradient of each block $b$ is $L_b$-Lipschitz continuous, meaning that for all $x \in \R^n$ and $d \in \R^{|b|}$
\begin{equation}
\norm{\nabla_b f(x+U_bd) - \nabla_b f(x)} \leq L_b\norm{d},
%\quad\mbox{$\forall x\in\mathbb{R}^n$ and $\alpha\in\mathbb{R}$},
\label{eq:blockLipschitz}
\end{equation}
for some constant $L_b > 0$. This is a standard assumption, and in~\ref{append:blockLipschitz} we give bounds on $L_b$ for the common data-fitting models of least squares and logistic regression.
If we apply the descent lemma~\citep{bertsekas1999} to the reduced sub-problem~\eqref{eq:subProb} associated with some block $b_k$ selected at iteration $k$, then we obtain the following upper bound on the function value progress,
\begin{align}
	f(x^{k+1}) 
	&\le f(x^k) + \langle \nabla f(x^k), x^{k+1} - x^k \rangle + \frac{L_{b_k}}{2} \| x^{k+1} - x^k \|^2 \label{eq:upperbound}\\
	&= f(x^k) + \langle \nabla_{b_k} f(x^k),d^k \rangle + \frac{L_{b_k}}{2} \| d^k \|^2 \nonumber.
\end{align}
%For a first-order update rule, we define $d_k^{(b)}$ as the vector that makes this bound as tight as possible, i.e.,
%\[
 %	d_k^{(b)} = \argmin{d^{(b)} \in \R^{n_b}} ~\left \{ \langle \nabla_b f(x_k), d^{(b)} \rangle + \frac{L}{2} \| d^{(b)} \|^2_2  \right \}.
%\]
The right side is minimized in terms of $d^k$ under the choice
%Differentiating with respect to $d^{(b)}$ we obtain a first order update,
 \begin{equation}
\label{eq:1/L}
 	d^k = - \frac{1}{L_{b_k}} \nabla_{b_k} f(x^k),
 \end{equation}
which is simply a gradient update with a step-size of $\alpha_k = 1/L_{b_k}$. Substituting this into our upper bound, we obtain
  \begin{align}
 	f(x^{k+1}) 
%	&\le f(x_k) -\frac{1}{L} \langle \nabla_b f(x_k), \nabla_b f(x_k) \rangle + \frac{1}{2L} \| \nabla_b f(x_k) \|^2_2 \\
	&\le f(x^k) - \frac{1}{2L_{b_k}} \| \nabla_{b_k} f(x^k) \|_2^2.
	\label{eq:quadBound}
 \end{align}
% Similar bounds hold for other standard ways to choose the step-size.
  We can use the bound~\eqref{eq:quadBound} to compare the progress made by different selection rules. \blue{For example, if we assume the $L_{b_k}$ are constant then ~\eqref{eq:quadBound} indicates that we expect more progress per iteration as the block size increases, since this will increase the norm of $\nabla_{b_k} f(x^k)$.  If we fix the $L_{b_k}$ and the block size, the GS rule is obtained by minimizing the bound~\eqref{eq:quadBound} with to respect to $b_k$. Thus, under~\eqref{eq:quadBound} the GS rule is guaranteed to make more progress than using cyclic or randomized selection.}
 
The bound~\eqref{eq:quadBound} also indicates that the GS rule can make more progress  with \adaptive blocks than with fixed blocks (under the usual setting where the fixed blocks are a subset of the possible \adaptive blocks). In particular, consider the case where we have partitioned the coordinates into blocks of size $\tau$ and we are comparing this to using variable blocks of size $\tau$. The case where there is no advantage for variable blocks is when the indices corresponding to the $\tau$-largest $|\nabla_i f(x^k)|$ values are in one of the fixed partitions; in this (unlikely) case the GS rule with fixed blocks and variable blocks will choose the same variables to update. The case where we see the largest advantage of using variable blocks is when each of the indices corresponding to the $\tau$-largest $|\nabla_i f(x^k)|$ values are in different blocks of the fixed partition; in this case the last term in~\eqref{eq:quadBound} can be improved by a factor as large as $\tau^2$ when using variable blocks instead of fixed blocks. Thus, with larger blocks there is more of an advantage to using variable blocks over fixed blocks.
 
\subsection{Block Gauss-Southwell-Lipschitz}
\label{sec:BGSL}
The GS rule is not the optimal block selection rule in terms of the bound~\eqref{eq:quadBound} if we know the block-Lipschitz constants $L_b$.
Instead of choosing the block with largest norm, consider minimizing~\eqref{eq:quadBound} in terms of $b_k$,
\begin{equation}
b_k \in \argmax{b \in \mathcal{B}}\left\{\frac{\norm{\nabla_b f(x^k)}^2}{L_{b}}\right\}.
\label{eq:GSL}
\end{equation}
We call this the \textbf{block Gauss-Southwell-Lipschitz} (GSL) rule. If all $L_b$ are the same, then the GSL rule is equivalent to the classic GS rule. However, in the typical case where the $L_b$ differ, the GSL rule guarantees more progress than the GS rule since it incorporates the gradient information as well as the Lipschitz constants $L_b$. For example, it reflects that if the gradients of two blocks are similar, but their Lipschitz constants are very different, then we can guarantee more progress by updating the block with the smaller Lipschitz constant. In the extreme case, for both fixed and variable blocks the GSL rule improves the bound~\eqref{eq:quadBound} over the GS rule by a factor as large as $(\max_{b\in\mathcal{B}}L_b)/(\min_{b\in\mathcal{B}}L_b$).
   
The block GSL rule in \eqref{eq:GSL} is a simple  generalization of the single-coordinate GSL rule to blocks of any size. However, it loses a key feature of the single-coordinate GSL rule: the block GSL rule is \emph{not} equivalent to the MI rule for quadratic functions. Unlike the single-coordinate case, where $\nabla_{ii}^2 f(x^k) = L_i$ so that \eqref{eq:upperbound} holds with equality, for the block case we only have $\nabla_{bb}^2 f(x^k) \preceq L_bI$ so~\eqref{eq:upperbound} may underestimate the progress that is possible in certain directions.\footnote{We say that a matrix $A$ ``upper bounds'' a matrix $B$, written $A \succeq B$, if for all $x$ we have $x^TAx \geq x^TBx$.} In the next section we give a second generalization of the GSL rule that is equivalent to the MI rule for quadratics.

%However, two complicating factors arise: (i) it's possible to obtain tighter quadratic bounds (see the next section), and (ii)
 
 \subsection{Block Gauss-Southwell-Quadratic}
 \label{sec:BGSQ}

For single-coordinate updates, the bound in~\eqref{eq:quadBound} is the tightest quadratic bound on progress we can expect given only the assumption of block Lipschitz-continuity (it holds with equality for quadratic functions). However, for block updates of more than one variable we can obtain a tighter quadratic bound using general quadratic norms of the form $\norm{\cdot}_H = \sqrt{\langle H \cdot, \cdot \rangle}$ for some positive-definite matrix $H$. In particular, assume that each block has a Lipschitz-continuous gradient with $L_b=1$ for a particular positive-definite matrix $H_b \in \R^{|b| \times |b|}$, meaning that
\[
\norm{\nabla_b f(x+U_bd) - \nabla_b f(x)}_{H_b^{-1}} \leq \norm{d}_{H_b},
\]
for all $x \in \R^n$ and $d \in \R^{|b|}$. \blu{In the case of twice-differentiable functions, this is equivalent to requiring that $\nabla_{bb}^2 f(x) \preceq H_b$.}
Due to the equivalence between norms, this merely changes how we measure the continuity of the gradient and is not imposing any new assumptions. Indeed, the block Lipschitz-continuity assumption in~\eqref{eq:blockLipschitz} is just a special case of the above with $H_b = L_bI$, where $I$ is the $|b| \times |b|$ identity matrix.
%Note that this isn't imposing any new assumptions beyond block Lipschitz-continuity of the gradient; 
 Although this characterization of Lipschitz continuity appears more complex, for some functions it is actually computationally cheaper to construct matrices $H_b$ than to find valid bounds $L_b$. We show this in~\ref{append:blockLipschitz} for the cases of least squares and logistic regression.

Under this alternative characterization of the Lipschitz assumption, at each iteration $k$ we have
\begin{equation}\label{eq:altLip}
	f(x^{k+1}) \le f(x^k) + \langle \nabla_{b_k} f(x^k),d^k \rangle + \frac{1}{2} \| d^k \|^2_{H_{b_k}}.
\end{equation}
%where $H_k^{(b)} \in \R^{n_b \times n_b}$ is an invertible matrix that approximates the sub-Hessian matrix for block $b$ at iteration $k$.
%Then we define $d_k^{(b)}$ such that this bound is as tight as possible, i.e., 
% \[
% 	d_k^{(b)} = \argmin{d^{(b)} \in \R^{n_b}} ~\left \{ \langle \nabla_b f(x_k), d^{(b)} \rangle + \frac{1}{2} \left (d^{(b)} \right )^T H_k^{(b)} d^{(b)} \right \},
 %\]
The right-hand side of~\eqref{eq:altLip} is minimized when
 \begin{equation}\label{eq:soupdate}
 	d^k = - \left ( H_{b_k} \right )^{-1} \nabla_{b_k} f(x^k),
 \end{equation}
\blu{ which corresponds to using the matrix update~\eqref{eq:matrix} with a step size of $\alpha_k=1$. While this update is equivalent to Newton's method for quadratic functions, it is important to distinguish this matrix update from Newton's method: Newton's method is based on the instantaneous Hessian $\nabla_{bb}^2 f(x^k)$ and may require $\alpha_k < 1$ to decrease the objective, while the matrix update~\eqref{eq:soupdate} is based on a matrix $H_{b}$ that upper bounds the Hessian for all $x$ (in the twice-differentiable case) and is thus guaranteed to decrease the objective with $\alpha_k=1$. We will discuss Newton updates further in Section~\ref{sec:newton}.}

Substituting the matrix update~\eqref{eq:soupdate} into the upper bound  yields
\begin{equation}\label{eq:uppersecondorder}
    f(x^{k+1}) \le f(x^k) - \frac{1}{2}  \| \nabla_{b_k} f(x^k) \|_{H_{b_k}^{-1}}^2.
\end{equation}
Consider a simple quadratic function $f(x) = x^TAx$ for a positive-definite matrix $A$. In this case we can take $H_b$ to be the sub-matrix $A_{bb}$ while in our previous bound we would require $L_b$ to be the maximum eigenvalue of this sub-matrix. Thus, in the worst case (where $\nabla_{b_k}f(x^k)$ is in the span of the principal eigenvectors of $A_{bb}$) the new bound is at least as good as~\eqref{eq:quadBound}. However, if the eigenvalues of $A_{bb}$ are spread out then this bound shows that the matrix update will typically guarantee substantially more progress; in this case the quadratic bound~\eqref{eq:uppersecondorder} can improve on the bound in~\eqref{eq:quadBound} by a factor as large as the condition number of $A_{bb}$ when updating block $b$.
The update \eqref{eq:soupdate} was analyzed for randomized BCD methods in several recent works~\citep{tappenden2016,qu2015sdna,fountoulakis2015flexible,rodomanov2020randomized,mutny2020convergence}, which considered random selection of the blocks. They show that this update provably reduces the number of iterations required, and in some cases dramatically. For the special case of quadratic functions where~\eqref{eq:uppersecondorder} holds with equality, greedy rules based on minimizing this bound have been explored for both fixed~\citep{thoppe2014} and variable~\citep{bo2012} blocks.
%Another interesting recent work considers~\eqref{eq:soupdate} with randomized selection and \adaptive blocks, exploring many of the issues that arise in this setting~\citep{csiba2015}. 

%where $H_k$ is a full dimensional invertible approximate Hessian matrix and $H_k^{(b)}$ guarantees that $d_k^{(b)}$ is a descent direction for the selected block $b$ (i.e., $H_k^{(b)}$ is a positive definite matrix). We note that the matrix norm $\| \cdot \|_{H_k}$ should be (?) related to the eigenvalues of the original problem (for $\mathcal{C}^2$ functions). We again assume that block $b$ is selected according to the dual norm, where in this case the dual norm is a mixed-norm with a matrix 2-norm. 
%Instead of randomized methods, our focus is on defining a better greedy rule than~\eqref{eq:GSL} for not-necessarily-quadratic functions with Lipschitz-continuous gradients.
Rather than focusing on the special case of quadratic functions, we want to define a better greedy rule than~\eqref{eq:GSL} for functions with Lipschitz-continuous gradients. 
By optimizing~\eqref{eq:uppersecondorder} in terms of $b_k$ we obtain a second generalization of the GSL rule,
\begin{equation}
b_k \in \argmax{b \in \mathcal{B}}\left\{\norm{\nabla_b f(x^k)}_{H_b^{-1}}\right\} \equiv \argmax{b \in \mathcal{B}} \left\{\nabla_b f(x^k)^TH_b^{-1}\nabla_b f(x^k)\right\},
\label{eq:GSF}
\end{equation}
which we call the \textbf{block Gauss-Southwell quadratic} (GSQ) rule.\footnote{While preparing this work for submission, we were made aware of a work that independently proposed a related rule under the name ``greedy mini-batch'' rule~\citep{csiba2017global}. \blu{The greedy mini-batch rule is a special case of the more-restricted GSQ rule that we discuss in Section~\ref{sec:aGSQ}.}}
Since~\eqref{eq:uppersecondorder} holds with equality for quadratics this new rule is equivalent to the MI rule in that case. This rule also applies to non-quadratic functions where it guarantees a better bound on progress than the GS rule (and the GSL rule).
%Two of the update methods proposed in~\citep{qu2015sdna} correspond to two of our selection rules in the following fixed blocks section.

\subsection{Block Gauss-Southwell-Diagonal}
\label{subsec:fixed}

While the GSQ rule has appealing theoretical properties, for many problems it may be difficult to find full matrices $H_b$ and their storage may also be an issue. Previous  related works~\citep{TsengYun2009,qu2015sdna,csiba2017global} address this issue by restricting the matrices $H_b$ to be diagonal matrices $D_b$. Under this choice we obtain a rule of the form
\begin{equation}
b_k \in \argmax{b \in \mathcal{B}}\left\{\norm{\nabla_b f(x^k)}_{D_b^{-1}}\right\} \equiv \argmax{b \in \mathcal{B}}\left\{\sum_{i \in b} \frac{ |\nabla_i f(x^k)|^2 }{D_{b,i} }\right\},
\label{eq:GSD}
\end{equation}
where we are using $D_{b,i}$ to refer to the diagonal element corresponding to coordinate $i$ in block $b$. We call this the \textbf{block Gauss-Southwell diagonal} (GSD) rule. This bound arises if we consider a gradient update, where coordinate $i$ has a constant step-size  of $D_{b,i}^{-1}$ when updated as part of block $b$. This rule gives an intermediate approach that can guarantee more progress per iteration than the GSL rule, but that may be easier to implement than the GSQ rule.

% --------------------------------------------------------------------
\subsection{Convergence Rate under Polyak-\L{}ojasiewicz}
\label{subsec:convergence}

Our discussion above focuses on the progress we can guarantee at each iteration, assuming only that the function has a Lipschitz-continuous gradient. Under additional assumptions, it is possible to use these progress bounds to derive convergence rates on the overall BCD method. For example, we say that a function $f$ satisfies the Polyak-\L{}ojasiewicz (PL) inequality \citep{polyak1963} if for all $x$ we have for some $\mu > 0$ that
\begin{equation}
\label{eq:PL}
	\frac{1}{2} \left ( \| \nabla f(x) \|_* \right )^2 \ge \mu \left ( f(x) - f^* \right ),
\end{equation}
where $\| \cdot \|_*$ can be any norm and $f^*$ is the optimal function value. The function class satisfying this inequality includes all strongly-convex functions, but also includes a variety of other important problems like least squares~\citep{karimi2016}. This inequality leads to a simple proof of the linear convergence of any algorithm which has a progress bound of the form
\begin{equation}
\label{eq:gradBound}
f(x^{k+1}) \leq f(x^k) - \frac{1}{2}\norm{\nabla f(x^k)}_*^2,
\end{equation}
such as gradient descent or coordinate descent with the GS rule~\citep{karimi2016}. 

\begin{theorem}
\label{thm:linearconvergence}
Assume $f$ satisfies the PL-inequality~\eqref{eq:PL} for some $\mu > 0$ and norm $\| \cdot \|_*$. Any algorithm that satisfies a progress bound of the form \eqref{eq:gradBound} with respect to the same norm $\| \cdot \|_*$ obtains the following linear convergence rate, 
\begin{equation}
\label{thmeq:linearrate}
	f(x^{k+1}) - f^* \le (1 - \mu)^k \left [ f(x^0) - f^* \right ].
\end{equation}
\end{theorem}

\begin{proof}
By subtracting $f^*$ from both sides of \eqref{eq:gradBound} and applying \eqref{eq:PL} directly, we obtain our result by recursion.
\end{proof}
%In particular, by subtracting $f^*$ from both sides and using the PL-inequality we obtain
%\[
%	f(x^{k+1}) - f^* \le \left ( 1 - \mu \right ) \left [ f(x^{k}) - f^* \right ].
%\]
%If we apply this recursively we obtain a linear convergence rate of the form
%\[
%f(x^k) - f(x^*) \le \left ( 1 - \mu \right )^k \left [ f(x^0) - f^* \right ].
%\]

Thus, if we can describe the progress obtained using a particular block selection rule and block update rule in the form of~\eqref{eq:gradBound}, then  we have a linear rate of convergence for BCD on this class of functions. It is straightforward to do this using an appropriately defined norm, as shown in the following corollary.

\begin{corollary}
Assume $\nabla f$ is Lipschitz continuous~\eqref{eq:altLip} and that $f$ satisfies the PL-inequality \eqref{eq:PL} in the norm defined by
\begin{equation}
\label{eq:mixedNormDef}
\norm{v}_{\B} = \max_{b \in \B} \norm{v_b}_{H_{b^{-1}}},
\end{equation}
for some $\mu > 0$ and matrix $H_b \in \R^{|b| \times |b|}$. Then the BCD method using either the GSQ, GSL, GSD or GS selection rule achieves a linear convergence rate of the form~\eqref{thmeq:linearrate}.
\end{corollary}

\begin{proof}
Using the definition of the GSQ rule~\eqref{eq:GSF} in the progress bound resulting from the Lipschitz continuity of $\nabla f$ and the matrix update~\eqref{eq:uppersecondorder}, we have %that using the matrix update after using the GSQ selection rule satisfies
\begin{equation}
\label{eq:mixedNorm}
\begin{aligned}
	f(x^{k+1}) & \leq f(x^k) - \frac{1}{2}\max_{b \in \B}\left\{\norm{\nabla_{b}f(x^k)}_{H_{b}^{-1}}^2\right\}\\
	& = f(x^k) - \frac{1}{2}\norm{\nabla f(x^k)}_{\B}^2.
\end{aligned}
\end{equation}
%This progress bound has the same form as \eqref{eq:gradBound} and holds for the same mixed norm under which we assume the PL-inequality. Thus, b
By Theorem~\ref{thm:linearconvergence} and the observation that the GSL, GSD and GS rules are all special cases of the GSQ rule corresponding to specific choices of $H_b$, we have our result.
%By the argument above, this gives us a linear convergence rate for the GSQ rule in terms of the PL inequality constant for the $\norm{\cdot}_\B$-norm. It also implies rates for the GSL and GSD rules (as well as the classic GS rule) since they correspond to specific choices of $H_b$.
\end{proof}
We refer the reader to the work of~\citet{csiba2017global} for alternate analyses of a variety of BCD methods, \blu{but we note that this rate will typically be faster than their rate for greedy methods (they show a general result that also includes randomized selection rules, but they obtain the same rate for random and greedy variants while this rate will be faster than random).}

%For the particular case of the GSQ rule, we can use this to prove a convergence rate by defining  ``mixed'' norms of the form
%\begin{equation}
%\label{eq:mixedNormDef}
%\norm{v}_{\B} = \max_{b \in \B} \norm{v_b}_{H_{b^{-1}}},
%\end{equation}
%to measure the length of a vector $v$.
%This mixed norm can be view as the $\infty$-norm of the quadratic norms $\| \cdot \|_{H_b^{-1}}$ of the sub-vectors $v_b$ for each $b\in \B$, and it is straightforward to verify that it defines a valid norm (provided that all coordinates are included in at least one block $b$). Using the definition of the GSQ rule~\eqref{eq:GSF} within the bound on progress made by the matrix update~\eqref{eq:uppersecondorder}, we have that using the matrix update after selecting a block according to the GSQ rule satisfies
%\begin{equation}
%\label{eq:mixedNorm}
%\begin{aligned}
%	f(x^{k+1}) & \leq f(x^k) - \frac{1}{2}\max_{b \in \B}\left\{\norm{\nabla_{b}f(x^k)}_{H_{b}^{-1}}^2\right\}\\
%& = f(x^k) - \frac{1}{2}\norm{\nabla f(x^k)}_{\B}^2.
%\end{aligned}
%\end{equation}
%By the argument above, this gives us a linear convergence rate for the GSQ rule in terms of the PL inequality constant for the $\norm{\cdot}_\B$-norm. It also implies rates for the GSL and GSD rules (as well as the classic GS rule) since they correspond to specific choices of $H_b$. 

\subsection{Convergence Rate with General Functions}
\label{subsec:nonconvex}

The PL inequality is satisfied for many problems of practical interest, and is even satisfied for some non-convex functions. However, general non-convex functions do not satisfy the PL inequality and thus the analysis of the previous section does not apply. Without a condition like the PL inequality, it is difficult to show convergence to the optimal function value $f^*$ (since we should not expect a local optimization method to be able to find the global solution of functions that may be NP-hard to minimize). However, the bound~\eqref{eq:mixedNorm} still implies a weaker type of convergence rate even for general non-convex problems. The following result is a generalization of a standard argument for the convergence rate of gradient descent~\citep{Nes04b}, and gives us an idea of how fast the BCD method is able to find a point resembling a stationary point even in the general non-convex setting.

\begin{theorem}
\label{thm:2}
Assume $\nabla f$ is Lipschitz continuous~\eqref{eq:altLip} and that  $f$ is bounded below by some $f^*$. Then the BCD method using either the GSQ, GSL, GSD or GS selection rule achieves the following convergence rate of the minimum gradient norm,
\[
\min_{t=0,1,\dots,k-1}\norm{\nabla f(x^t)}_\B^2 \leq \frac{2(f(x^0) - f^*)}{k}.
\]
\end{theorem}

\noindent {\bf Proof}~
By rearranging~\eqref{eq:mixedNorm}, we have
\[
\frac{1}{2}\norm{\nabla f(x^k)}_\B^2 \leq f(x^k) - f(x^{k+1}).
\]
Summing this inequality over iterations $t = 0$ up to $(k-1)$ yields
\[
\frac{1}{2}\sum_{t=0}^{k-1}\norm{\nabla f(x^t)}_\B^2 \leq f(x^0) - f(x^{k+1}).
\]
Using that all $k$ elements in the sum are lower bounded by their minimum and that $f(x^{k+1}) \geq f^*$, we get 
\[
\frac{k}{2}\left(\min_{t=0,1,\dots,k-1}\norm{\nabla f(x^t)}_\B^2\right) \leq f(x^0) - f^*. \qed
\]
%\end{proof}
Due to the potential non-convexity we cannot say anything about the gradient norm of the final iteration, but this shows that the minimum gradient norm converges to zero with an error at iteration $k$ of $O(1/k)$. This is a global sublinear result, but note that if the algorithm eventually arrives and stays in a region satisfying the PL inequality around a set of  local optima, then the local convergence rate to this set of optima will increase to be linear.

%In particular, if we only assume that the function $f$ is bounded below by some $f^*$ then we can still bound the rate at which the gradient norm approaches zero. This gives us an idea of how fast the algorithm is able to find a point resembling a stationary point even in general non-convex settings.
%
%To analyze the general non-convex setting, we generalize a standard argument for the convergence rate of gradient descent~\citep{Nes04b}. In particular, we first re-arrange~\eqref{eq:mixedNorm} to give
%\[
%\frac{1}{2}\norm{\nabla f(x^k)}_\B^2 \leq f(x^k) - f(x^{k+1}).
%\]
%Summing this inequality over iterations $t = 0$ up to $(k-1)$ yields
%\[
%\frac{1}{2}\sum_{t=0}^{k-1}\norm{\nabla f(x^t)}_\B^2 \leq f(x^0) - f(x^{k+1}).
%\]
%Now using that all $k$ elements in the sum are lower bounded by their minimum and that $f(x^{k+1}) \geq f^*$, we get
%\[
%\frac{k}{2}\left(\min_{t=0,1,\dots,k-1}\norm{\nabla f(x^t)}_\B^2\right) \leq f(x^0) - f^*.
%\]
%By re-arranging we obtain
%\[
%\min_{t=0,1,\dots,k-1}\norm{\nabla f(x^t)}_\B^2 \leq \frac{2(f(x^0) - f^*)}{k},
%\]
%Due to the potential non-convexity we cannot say anything about the gradient norm of the final iteration, but this shows that the minimum gradient norm converges to zero with an error at iteration $k$ of $O(1/k)$. This is a global sublinear result, but note that if the algorithm eventually arrives and stays in a region satisfying the PL inequality around local optima then the local convergence rate to this set of optima will increase to be linear.

\section{Practical Issues}
\label{sec:practical}

The previous section defines new greedy rules that yield a simple analysis. But in practice there are several issues that remain to be addressed. For example, it seems intractable to compute any of the new rules in the case of variable blocks. Furthermore, we may not know the Lipschitz constants for our problem. For fixed blocks we also need to consider how to partition the coordinates into blocks. Another issue is that the $d^k$ choices used above do not incorporate the local Hessian information.
%, while yet another issue is how to take advantage of sparse dependencies between variables. 
Although how we address these issues will depend on our particular application, in this section we discuss several issues associated with these practical considerations.
\blu{Several parts of this section discuss the use of coordinate-wise Lipschitz constants $L_i$ or having a matrix upper bound $M$ on the full Hessian $\nabla^2 f(x)$ (for all $x$). In~\ref{append:blockLipschitz} we discuss how to obtain such bounds for least squares, binary logistic regressionm, and multi-class logistic regression. For other problems it may be more difficult to obtain such quantities.}

\subsection{Tractable GSD for Variable Blocks}
\label{subsec:adaptive}

The problem with using any of the new selection rules above in the case of \adaptive blocks is that they seem intractable to compute for any non-trivial block size. In particular, to compute the GSL rule using \adaptive blocks requires the calculation of $L_b$ for each possible block, which seems intractable for any problem of reasonable size. Since the GSL rule is a special case of the GSD and GSQ rules, these rules also seem intractable in general. 
%We found this somewhat surprising since computing the standard GS rule simply involves finding the $\tau$-largest elements of the gradient. 
In this section we show how to restrict the GSD matrices so that this rule has the same complexity as the classic GS rule. 

Consider a variant of the GSD rule where each $D_{b,i}$ can depend on $i$ but does not depend on $b$, so we have $D_{b,i} = d_i$ for some value $d_i \in \R_+$ for all blocks $b$. This gives a rule of the form
\begin{equation}
b_k \in \argmax{b \in \mathcal{B}}\left\{\sum_{i \in b} \frac{ |\nabla_i f(x^k)|^2 }{d_{i} }\right\}.
\label{eq:GSC}
\end{equation}
%which we call the {\bf block Gauss-Southwell-Coordinate} (GSC) rule. 
%This rule arises by considering a gradient update of the block where each coordinate $i$ has a constant step-size as in the GSD rule, but this step-size is the same across all blocks $b$ where $i$ might be included.
Unlike the general GSD rule, this rule has essentially the same complexity as the classic GS rule since it simply involves finding the largest values of the ratio $|\nabla_i f(x^k)|^2/d_i$. 
% and given the gradient values this can be computed in linear time using a fast selection algorithm~\citep{cormen2009introduction}. 

A natural choice of the $d_i$ values would seem to be $d_i = L_i$, since in this case we recover the GSL rule if the blocks have a size of 1 (here we are using $L_i$ to refer to the coordinate-wise Lipschitz constant of coordinate $i$). Unfortunately, this does not lead to a bound of the form~\eqref{eq:gradBound} as needed in Theorems~\ref{thm:linearconvergence} and~\ref{thm:2} because coordinate-wise $L_i$-Lipschitz continuity with respect to the Euclidean norm does not imply 1-Lipschitz continuity with respect to the norm $\| \cdot \|_{D_b^{-1}}$ when the block size is larger than 1. Subsequently, the steps under this choice may increase the objective function. %Hence, we could not use this $D_i$ in our update of the block and this $D_i$ does not allow us to apply the theory of the previous section.
A similar restriction on the $D_b$ matrices in \eqref{eq:GSC} is used in the implementation of Tseng and Yun based on the Hessian diagonals~\citep{TsengYun2009}, but their approach similarly does not give an upper bound and thus they employ a line-search in their block update. 

%While the natural choice of using $D_{b,i} = L_i$ in the selection rule and block update does not lead to convergence of the algorithm in general, 
It is possible to avoid needing a line-search by setting $D_{b,i} = L_i\tau$, where $\tau$ is the maximum block size in $\mathcal{B}$. This still generalizes the single-coordinate GSL rule, and in~\ref{append:GSd} we show that this leads to a bound of the form~\eqref{eq:gradBound} for twice-differentiable convex functions (thus  Theorems~\ref{thm:linearconvergence} and~\ref{thm:2} hold). If all blocks have the same size then this approach selects the same block as using $D_{b,i} = L_i$, but the matching block update uses a much-smaller step-size that guarantees descent. We do not advocate using this smaller step, but note that the bound we derive also holds for alternate updates like taking a gradient update with $\alpha_k = 1/L_{b_k}$ or using a matrix update based on $H_{b_k}$. 

The choice of $D_{b,i} = L_i\tau$ leads to a fairly pessimistic  bound, but it is not obvious even for simple problems how to choose an optimal set of $D_{i}$ values. 
Choosing these values is related  to the problem of finding an expected separable over-approximation (ESO), which arises in the context of randomized coordinate descent methods~\citep{richtarik2016parallel}. Qu and Richt\'arik give an extensive discussion of how we might bound such quantities for certain problem structures~\citep{qu2016coordinate}.
In our experiments we also explored another simple choice that is inspired by the ``simultaneous iterative reconstruction technique'' (SIRT) from tomographic image reconstruction~\citep{gregor2015comparison}. In this approach, we use a matrix upper bound $M$ on the full Hessian $\nabla^2 f(x)$ (for all $x$) and set\footnote{It follow that $D - M \succeq 0$ because it is symmetric and diagonally-dominant with non-negative diagonals.}
\begin{equation}
\label{eq:GSDsum}
D_{b,i} = \sum_{j=1}^n|M_{i,j}|.
\end{equation}
\blu{Unfortunately, for many problems this approximation might be too costly since it requires forming the full-matrix upper-bound $M$.\footnote{Costing $O(n^2)$ to compute all possible $n$ values, plus the cost of forming the upper bound $M$ which is problem-dependent, but may be larger than $O(n^2)$.}}
In our experiments we did find that this choice worked better when using gradient updates, but using the simpler $L_i\tau$ is less expensive and was more effective when doing matrix updates.

By using the relationship $L_b \leq \sum_{i\in b}L_i \leq |b|\max_{i\in b}L_i$, two other ways we might consider defining a more-tractable rule could be 
\[
b_k \in \argmax{b \in \mathcal{B}}\left\{\frac{\sum_{i\in b}|\nabla_i f(x^k)|^2}{|b|\max_{i\in b}L_i}\right\}, \quad \text{or} \quad b_k \in \argmax{b \in \mathcal{B}}\left\{\frac{\sum_{i\in b}|\nabla_i f(x^k)|^2}{\sum_{i\in b}L_i}\right\}.
\]
The rule on the left can be computed using dynamic programming while the rule on the right can be computed using an algorithm of~\citet{megiddo1979combinatorial}. However, when using a step-size of $1/L_b$ we found both rules performed similarly or worse to using the GSD rule with $D_{i,b} = {L_i}$ (when paired with gradient or matrix updates).\footnote{On the other hand, the rule on the right worked better if we forced the algorithms to use a step-size of $1/(\sum_{i \in b}L_i)$. However, this lead to worse performance overall than using the larger $1/L_b$ step-size.}

\subsection{Tractable GSQ for Variable Blocks}
\label{sec:aGSQ}

%Above we show that the GSD rule, which is intractable for variable blocks, can be made tractable under the restriction that $D_{b,i} = D_i$. 
In order to make the GSQ rule tractable with variable blocks, we could similarly require that the entries of $H_b$ depend solely on the coordinates $i \in b$, so that $H_b = M_{b,b}$ where $M$ is a fixed matrix (as above) and $M_{b,b}$ refers to the sub-matrix corresponding to the coordinates in $b$. 
%Here we would assume that the full gradient $\nabla f(x)$ is $1$-Lipschitz in the quadratic norm defined by $M$. 
Our restriction on the GSD rule in the previous section corresponds to the case where $M$ is diagonal. In the full-matrix case, the block selected according to this rule is given by the coordinates corresponding to the non-zero variables of an L0-constrained quadratic minimization,
\begin{equation}
\label{eq:L0min}
\argmin{\norm{d}_0 \leq \tau} \left\{f(x^k) + \langle \nabla f(x^k), d\rangle + \frac{1}{2}d^TMd\right\},
\end{equation}
where $\norm{\cdot}_0$ is the number of non-zeroes.
This selection rule is discussed in~\citet{TsengYun2009}, but in their implementation they use a diagonal $M$. \blu{\citet{csiba2017global} discuss and analyze the convergence rate of this rule under the name ``greedy mini-batch'' rule.}
Although this problem is NP-hard with a non-diagonal $M$, there is a recent wealth of literature on algorithms for finding approximate solutions. For example, one of the simplest local optimization methods for this problem is the iterative hard-thresholding (IHT) method~\citep{blumensath2009iterative}. Another popular method for approximately solving this problem is the orthogonal matching pursuit (OMP) method from signal processing which is also known as forward selection in statistics~\citep{pati1993orthogonal,hocking1976biometrics}.
Computing $d$ via~\eqref{eq:L0min} is also equivalent to computing the MI rule for a quadratic function, \blu{and thus we could alternately use the analogue of OMP designed by~\citet{bo2012} for this problem. }

Although it appears quite general, note that the exact GSQ rule under this restriction on $H_b$ does not guarantee as much progress as the more-general GSQ rule (if computed exactly) that we proposed in the previous section. For some problems we can obtain tighter matrix bounds over blocks of variables than are obtained by taking the sub-matrix of a fixed matrix-bound over all variables. We show this for the case of multi-class logistic regression in~\ref{append:blockLipschitz}. As a consequence of this result we conclude that there does not appear to be a reason to use this restriction in the case of fixed blocks.

Although using the GSQ rule with variable blocks forces us to use an approximation, these approximations might still select a block that makes more progress than methods based on diagonal approximations (which ignore the strengths of dependencies between variables). On the other hand, it is possible that approximating the GSQ rule does not necessarily lead to a bound of the form~\eqref{eq:gradBound} as there may be no fixed norm for which this inequality holds. \blu{If we are using an iterative solution method to solve~\eqref{eq:L0min}, one way to guarantee that Theorems~\ref{thm:linearconvergence} and~\ref{thm:2} hold for some fixed norm is to intialize the iterative solution method with the result from an update rule that already satisfies these bounds (such as the GS or GSD rule). In other words, if you have an iterative solver for approximating the GSQ rule then you could initialize it with something like the GS rule. With this initialization it makes at least as much progress as the GS rule on each iteration and thus has a convergence rate that is at least as fast as the GS rule.\footnote{This is assuming that the GSQ rule is chosen to provide a tighter bound, and assuming that the iterative solver returns a value making at least as much progress as its initialization in solving~\eqref{eq:gradBound}.} }
%Further, although using an approximation to the GSQ rule does not necessarily lead to a bound of the form~\eqref{eq:gradBound}, we can obtain such bounds if we initialize the algorithm with a rule that does achieve such a bound (so that we do at least as well as this reference rule). 

The main disadvantages of this approach for large-scale problems are the need to deal with the full matrix $M$ (which does not arise when using a diagonal approximation or using fixed blocks) and \blu{ the cost of computing an approximate solution to~\eqref{eq:L0min}. For example, assuming no additional problem structure, the cost of the OMP method in this setting is $O(n\tau^2)$. This per-iteration cost becomes expensive compared to rules based on diagonal updates for larger $\tau$, and we found that it made the approach less effective than diagonal approximations in terms of runtime}. In large-scale settings we would need to consider matrices $M$ with special structures like the sum of a diagonal matrix with a sparse and/or a low-rank matrix.

\subsection{Lipschitz Estimates for Fixed Blocks}
\label{sec:LA}

%We've motivated the GSC rule in the context of \adaptive blocks,
Using the GSD rule with the choice of $D_{i,b} = L_i$ %(and smaller step in the update) 
may also be useful in the case of fixed blocks. In particular, if it is easier to compute the single-coordinate $L_i$ values than the block $L_b$ values then we might prefer to use the GSD rule with this choice. On the other hand, 
%for the GSL rule an alternative in the case where even computing all $L_i$ is expensive 
an appealing alternative in the case of fixed blocks 
is to use an estimate of $L_b$ for each block as in Nesterov's work~\citep{nesterov2012efficiency}. In particular, for each $L_b$ we could start with some small estimate (like $L_b = 1$) and then double it whenever the inequality~\eqref{eq:quadBound} is not satisfied (since this indicates $L_b$ is too small). Given some $b$, the bound obtained under this strategy is at most a factor of 2 slower than using the optimal values of $L_b$. Further, if our estimate of $L_b$ is much smaller than the global value, then this strategy can actually guarantee much more progress than using the ``correct'' $L_b$ value.\footnote{While it might be tempting to also apply such estimates in the case of \adaptive blocks, a practical issue is that we would need a step-size for all of the exponential number of possible \adaptive blocks.}

In the case of matrix updates, we can use~\eqref{eq:uppersecondorder} to verify that an $H_b$ matrix is valid~\citep{fountoulakis2015flexible}. Recall that~\eqref{eq:uppersecondorder} is derived by plugging the update~\eqref{eq:soupdate} into the Lipschitz progress bound~\eqref{eq:altLip}. %It is possible that the $H_b$ in the matrix update differs from the one used in the quadratic norm measure of Lipschitz continuity and it is not obvious how to update the matrix $H_b$ to deal with this.
Unfortunately, it is not obvious how to update a matrix $H_b$ if we find that it is not a valid upper bound.
 One simple possibility is to multiply the elements of our estimate $H_b$ by 2. This is equivalent to using a matrix update but with a scalar step-size $\alpha_k$,
 \begin{equation}
 \label{eq:matrixStepSize}
 d^k = -\alpha_k(H_b)^{-1}\nabla_{b_k}f(x^k),
 \end{equation}
 similar to the step-size in the Newton update~\eqref{eq:newton}. 
 %Through this technique, we are approximating the Lipschitz constant (which is not necessarily 1) for our update matrix $H_b$ to ensure the bound~\eqref{eq:uppersecondorder} is satisfied.

\subsection{Efficient Line-Searches}
\label{sec:line-search}

The Lipschitz approximation procedures of the previous section do not seem practical when using variable blocks, since there are an exponential number of possible blocks. To use variable blocks for problems where we do not know $L_b$ or $H_b$, a reasonable approach is to use a line-search. For example, we can choose $\alpha_k$ in~\eqref{eq:matrixStepSize} using a standard line-search like those that use the Armijo condition or Wolfe conditions~\citep{wright1999numerical}. When using large block sizes with gradient updates, line-searches based on the Wolfe conditions are likely to make more progress than using the \emph{true} $L_b$ values (since for large block sizes the line-search would tend to choose values of $\alpha_k$ that are much larger than $\alpha_k = 1/L_{b_k}$).

Further, the problem structures that lend themselves to efficient coordinate descent algorithms tend to lend themselves to efficient line-search algorithms. For example, if our objective has the form $f(Ax)$ then a line-search would try to minimize the $f(Ax^k + \alpha_kAU_{b_k}d^k)$ in terms of $\alpha_k$. Notice that the algorithm would already have access to $Ax^k$ and that we can efficiently compute $AU_{b_k}d^k$ since it only depends on the columns of $A$ that are in $b_k$. Thus, after (efficiently) computing $AU_{b_k}d^k$ once, the line-search simply involves trying to minimize $f(v_1 + \alpha_kv_2)$ in terms of $\alpha_k$ (for particular vectors $v_1$ and $v_2$). \blu{The cost of line-search backtracking in this setting would thus simply be $O(\tau)$ to compute $v_1 + \alpha_k v_2$ plus the cost of evaluating $f$ given a vector (which is assumed to be much cheaper than evaluating products with $A$), the same cost as performing the BCD update}. Thus, the cost of this approximate minimization does not add a large cost to the overall algorithm. \blu{Line-searches can similarly be efficitently implemented for many of the problem structures highlighted in Section~\ref{sec:problems}}.

\subsection{Block Partitioning with Fixed Blocks}
\label{sec:part}

Several prior works note that for fixed blocks the partitioning of coordinates into blocks can play a significant role in the performance of BCD methods.~\citet{thoppe2014} suggest trying to find a block-diagonally dominant partition of the coordinates, and experimented with a heuristic for quadratic functions where the coordinates corresponding to the rows with the largest values were placed in the same block.
%Thoppe et al. \cite{thoppe2014} present a BCD method for unconstrained quadratic programs and select blocks using a heuristic method that minimizes the difference between a diagonally dominant approximation and the true block submatrix of the subproblem, enforcing the condition that the minimum eigenvalue of the approximation is large.
 In the context of parallel BCD,~\citet{scherrer2012} consider a feature clustering approach in the context of problem $h_1$ that tries to minimize the spectral norm between columns of $A$ from different blocks.~\citet{csiba2016importance} discuss strategies for partitioning the coordinates when using randomized selection. \blu{In the case of least squares, finding a good partitioning into blocks is related to the problem of matrix paving~\citep{needell2012}.}

Based on the discussion in the previous sections, for greedy BCD methods it is clear that we guarantee the most progress if we can make the mixed norm $\norm{\nabla f(x^k)}_\B$  as large as possible \emph{across iterations} (assuming that the $H_b$ give a valid bound). This supports strategies where we try to minimize the maximum Lipschitz constant across iterations. One way to do this is to try to ensure that the average Lipschitz constant across the blocks is small. For example, we could place the largest $L_i$ value with the smallest $L_i$ value, the second-largest $L_i$ value with the second-smallest $L_i$ value, and so on.
While intuitive, this may be sub-optimal; it ignores that if we cleverly partition the coordinates we may force the algorithm to often choose blocks with very-small Lipschitz constants (which lead to much more progress in decreasing the objective function). In our experiments, similar to the method of Thoppe et.\ al.\ for quadratics, we explore the simple strategy of \emph{sorting the $L_i$  values and partitioning this list into equal-sized blocks}. Although in the worst case this leads to iterations that are not very productive since they update all of the largest $L_i$ values, it also guarantees some very productive iterations that update none of the largest $L_i$ values and leads to better overall performance in our experiments.

\subsection{Newton Updates}
\label{sec:newton}

%Above we discussed various issues associated with choosing the block $b_k$. Our remaining contributions focus on the problem of choosing the  vector $d^k$ that we use to update the block $x_{b_k}$. 
Choosing the  vector $d^k$ that we use to update the block $x_{b_k}$ would
 seem to be straightforward since in the previous section we derived the block selection rules in the context of specific block updates; the GSL rule is derived assuming a gradient update~\eqref{eq:1/L}, the GSQ rule is derived assuming a matrix update~\eqref{eq:soupdate}, and so on. However, using the update $d^k$ that leads to the selection rule can be highly sub-optimal. For example, we might make substantially more progress using the matrix update~\eqref{eq:soupdate} even if we choose the block $b_k$ based on the GSL rule. 
Indeed, given $b_k$ the matrix update makes the optimal amount of progress for quadratic functions, so in this case we should prefer the matrix update for all selection rules (including random and cyclic rules).

However, the matrix update in~\eqref{eq:soupdate} can itself be highly sub-optimal for non-quadratic functions as it employs an upper-bound $H_{b_k}$ on the sub-Hessian $\nabla_{b_kb_k}^2 f(x)$ that must hold for {\em  all} parameters $x$. For twice-differentiable non-quadratic functions, we could potentially make more progress by using classic Newton updates where we use the instantaneous Hessian $\nabla_{b_kb_k}^2 f(x^k)$ with respect to the block. 
Indeed, considering the extreme case where we have one block containing all the coordinates, Newton updates can lead to superlinear convergence~\citep{dennis1974characterization} while matrix updates destroy this property. That being said, we should not expect superlinear convergence of BCD methods with Newton or even optimal updates\footnote{Consider a 2-variable quadratic objective where we use single-coordinate updates. The optimal update (which is equivalent to the matrix/Newton update) is easy to compute, but if the quadratic is non-separable then the convergence rate of this approach is only linear.}. Nevertheless, in Section~\ref{sec:exactupdates} we show that for certain common problem structures it is possible to achieve superlinear convergence with Newton-style updates.

Fountoulakis \& Tappenden recently highlight this difference between using matrix updates and using Newton updates~\citep{fountoulakis2015flexible}, and propose a BCD method based on Newton updates. To guarantee progress when far from the solution classic Newton updates require safeguards like a line-search or trust-region~\citep{TsengYun2009,fountoulakis2015flexible}, but as we have discussed in this section line-searches tend not to add a significant cost to BCD methods. Thus, if we want to maximize the progress we make at each iteration we recommend to use one of the greedy rules to select the block to update, but then update the block using the Newton direction and a line-search. In our implementation, we used a backtracking line-search starting with $\alpha_k = 1$ and backtracking for the first time using quadratic Hermite polynomial interpolation and using cubic Hermite polynomial interpolation if we backtracked more than once (which rarely happened since $\alpha_k=1$ or the first backtrack were typically accepted)~\citep{wright1999numerical}.\footnote{We also explored a variant based on cubic regularization of Newton's method~\citep{nesterov2006} as in the work of~\citet{amaral2020complexity}, but were not able to obtain a significant performance gain with this approach.}

\section{Message-Passing for Huge-Block Updates}
\label{sec:messagepassing}

\citet{qu2015sdna} discuss how in some settings increasing the block size with matrix updates does not necessarily lead to a performance gain due to the higher iteration cost. In the case of Newton updates the additional cost of computing the sub-Hessian $\nabla_{bb}^2 f(x^k)$ may also be non-trivial. Thus, whether matrix and Newton updates will be beneficial over gradient updates will depend on the particular problem and the chosen block size. However, in this section we argue that in some cases matrix updates and Newton updates can be computed efficiently using huge blocks. \blu{We first discuss the cost of using Newton updates and standard approaches to reduce this cost}. \blu{We then show how, for problems with sparse dependencies between variables, we can in some cases choose the structure of the blocks in order to guarantee that the matrix/Newton update can be computed in linear time.}

\subsection{Cost of Computing Newton Updates}
\label{sec:newtonCost}

The cost of using Newton updates with the BCD method depends on two factors: (i) the cost of calculating the sub-Hessian $\nabla^2_{b_kb_k} f(x^k)$ and (ii) the cost of solving the corresponding linear system.  The cost of computing the sub-Hessian depends on the particular objective function we are minimizing. For the problems where coordinate descent is efficient (see Section~\ref{sec:problems}), it is typically substantially cheaper to compute the sub-Hessian for a block than to compute the full Hessian. Indeed, for many cases where we apply BCD, computing the sub-Hessian for a block is cheap due to the sparsity of the Hessian. For example, in the graph-structured problems $h_2$ the edges in the graph correspond to the non-zeroes in the Hessian. 

Although this sparsity and reduced problem size would seem to make BCD methods with exact Newton updates ideal, in the worst case the iteration cost would still be $O(|b_k|^3)$ using standard matrix factorization methods. A similar cost is needed using the matrix updates with fixed Hessian upper-bounds $H_b$ and for performing an optimal update in the special case of quadratic functions. In some settings we can reduce this to $O(|b_k|^2)$ by storing matrix factorizations, but this cost is still prohibitive if we want to use large blocks (we can use $|b_k|$ in the thousands, but not the millions). 

An alternative to computing the exact Newton update is to use an approximation to the Newton update that has a runtime dominated by the sparsity level of the sub-Hessian. For example, we could use conjugate gradient methods or use randomized Hessian approximations~\citep{dembo1982inexact,pilanci2015newton}. However, these approximations require setting an approximation accuracy and may be inaccurate if the sub-Hessian is not well-conditioned. 

\blu{In this section we consider an alternative to using conjugate gradient or randomized algorithms to exploit sparsity in the Hessian}: choosing blocks with a sparsity pattern that guarantees we can solve the resulting linear systems involving the sub-Hessian (or its approximation) in $O(|b_k|)$ using a ``message-passing'' algorithm.
If the sparsity pattern is favourable, this allows us to update huge blocks at each iteration using exact matrix updates or Newton updates (which are the optimal updates for quadratic problems). \blu{This idea has previously been explored in the context of linear programing relaxations of inference in graphical models~\citep{sontag2009tree}, but here we consider it for computing matrix/Newton updates and consider constructing the trees greedily.} In Section~\ref{subsec:solvingforests}, we first discuss how message passing can be used to solve forest-structured linear systems, and then in Section~\ref{subsec:solvingforestsBCD} we show how this can be used within BCD methods.

\subsection{Solving Forest-Structured Linear Systems}
\label{subsec:solvingforests}

\blu{Consider the classic problem of finding a solution $x\in\R^n$ to a square system of linear equations, $Ax = c$,
where $A \in \R^{n\times n}$ and $c \in \R^n$.
We can define a pairwise undirected graph $G = (V,E)$, where we have $n$ vertices in $V$ and the edges $E$ are the non-zero off-diagonal elements of $A$. Thus, if $A$ is diagonal then $G$ has no edges, if $A$ is dense then there are edges between all nodes ($G$ is fully-connected), if $A$ is tridiagonal then edges connect adjacent nodes ($G$ is a chain-
structured graph where $(1)-(2)-(3)-(4)-\dots$), and so on.}

\blu{We are interested in the special case where the graph $G$ forms a \emph{forest}, meaning that it has no cycles.\footnote{An undirected cycle is a sequence of adjacent nodes in $V$ starting and ending at the same node, where there are no repetitions of nodes or edges other than the final node.} In the special case of forest-structured graphs, we can solve the linear system $Ax=c$ in linear time using message passing~\citep{shental2008GaBP}. This is as oppposed to the cubic worst-case time required by typical matrix factorization implementations. In the forest-structured case, the message passing algorithm is equivalent to Gaussian elimination with a particular permutation  that guarantees that the amount of ``fill-in'' is  linear~\citep[Prop. 3.4.1]{bickson2009}. This idea of exploiting tree structures within Gaussian elimination dates back over 50 years~\citep{parter1961use}.}

\begin{algorithm}[!ht]
\caption{Message Passing for a Tree Graph} %\cite{shental2008GaBP}}
\label{alg:GaBP}
\begin{algorithmic}
\STATE{ 1. {\bf Initialize}: \\
	$\quad$ Input: vector $c$, forest-structured matrix $A$, and levels $L\{1\}, L\{2\}, \dots, L\{T\}$.\\
	$\quad$ \algorithmicfor{ $i = 1, 2, \dots, n$} \\
	$\quad \quad$ Set $P_{ii} \leftarrow A_{ii}$, $C_{i} \leftarrow c_i$.  \comm{$P, C$ track row operations} \bigskip \\
}
\STATE{ 2. {\bf Gaussian Elimination}: \\
	$\quad$ \algorithmicfor{ $t = T,T-1,\dots, 1$} \comm{start furthest from root}\smallskip \\
	$\quad \quad$ \algorithmicfor{ $i \in L\{t\}$} \\
	$\quad \quad \quad$ \algorithmicif{ $t > 1$}  \\
	$\quad \quad \quad \quad$ $J \leftarrow N\{i\} \backslash L\{1:t-1\}$ \comm{neighbours that are not parent node} \smallskip \\
	$\quad \quad \quad \quad$ \algorithmicif{ $J = \emptyset$} \comm{$i$ corresponds to a leaf node} \smallskip \\
	$\quad \quad \quad \quad \quad$ {\bf continue} \comm{no updates} \smallskip \\
	$\quad \quad \quad$ \algorithmicelse{} \\
	$\quad \quad \quad \quad$ $J \leftarrow N\{i\}$ \comm{root node has no parent node}  \smallskip \\
	$\quad \quad \quad P_{Ji} \leftarrow A_{Ji}$ \comm{initialize off-diagonal elements}\\
	$\quad \quad \quad P_{ii} \leftarrow P_{ii} - \displaystyle \sum_{j \in J} \frac{P_{ji}^2}{P_{jj}}$ \comm{update diagonal elements of $P$ in $L\{t\}$}\\
	$\quad \quad \quad C_{i} \leftarrow C_{i} - \displaystyle \sum_{j \in J} \frac{P_{ji}}{P_{jj}} \cdot C_j$ \\
}
\STATE{ 3. {\bf Backward Solve}: \\
	$\quad$ \algorithmicfor{ $t = 1, 2, \dots, T$} \comm{start with root node}\smallskip \\
	$\quad \quad$ \algorithmicfor{ $i \in L\{t\}$} \smallskip \\
	$\quad \quad \quad$ \algorithmicif{ $t < T$}  \\
	$\quad \quad \quad \quad$ $p \leftarrow N\{i\} \backslash L\{t+1: T\}$ \comm{parent node of $i$ (empty for $t = 1$)} \smallskip \\
	$\quad \quad \quad$ \algorithmicelse{} \\
	$\quad \quad \quad \quad$ $p \leftarrow N\{i\}$ \comm{only neighbour of leaf node is parent}  \smallskip \\ 
	$\quad \quad \quad$ $x_{i} \leftarrow \displaystyle \frac{C_i - A_{i p}\cdot x_{p}}{P_{ii}}$ \comm{solution to $Ax = c$}
} 
\end{algorithmic}
\end{algorithm}

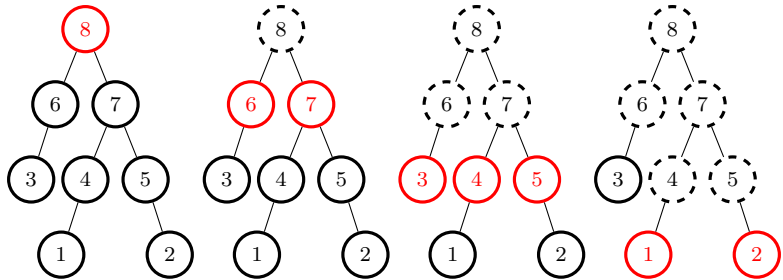
\begin{figure}[!ht]
\centering
\begin{tikzpicture}
\node [arn_r] {8}
		child{ node [arn_n] {6}
			child[left]{ node [arn_n] {3}}
		}
    		child{ node [arn_n] {7} 
            		child{ node [arn_n] {4} 
            			child[left]{ node [arn_n] {1}} %for a named pointer
            		}
            		child{ node [arn_n] {5}
				child[right]{ node [arn_n] {2}}
			}                      
    		}
; 
\end{tikzpicture}
\begin{tikzpicture}
\node [arn_x] {8}
		child{ node [arn_r] {6}
			child[left]{ node [arn_n] {3}}
		}
    		child{ node [arn_r] {7} 
            		child{ node [arn_n] {4} 
            			child[left]{ node [arn_n] {1}} %for a named pointer
            		}
            		child{ node [arn_n] {5}
				child[right]{ node [arn_n] {2}}
			}                      
    		}
; 
\end{tikzpicture}
\begin{tikzpicture}
\node [arn_x] {8}
		child{ node [arn_x] {6}
			child[left]{ node [arn_r] {3}}
		}
    		child{ node [arn_x] {7} 
            		child{ node [arn_r] {4} 
            			child[left]{ node [arn_n] {1}} %for a named pointer
            		}
            		child{ node [arn_r] {5}
				child[right]{ node [arn_n] {2}}
			}                      
    		}
; 
\end{tikzpicture}
\begin{tikzpicture}
\node [arn_x] {8}
		child{ node [arn_x] {6}
			child[left]{ node [arn_n] {3}}
		}
    		child{ node [arn_x] {7} 
            		child{ node [arn_x] {4} 
            			child[left]{ node [arn_r] {1}} %for a named pointer
            		}
            		child{ node [arn_x] {5}
				child[right]{ node [arn_r] {2}}
			}                      
    		}
;
\end{tikzpicture}
\caption{Process of partitioning nodes into level sets. In the above graph, we have the following {\color{blue}level} sets: $L\{1\} = \{ 8 \}, L\{2\} = \{6,7\}, L\{3\} = \{3,4,5\}$, $L\{4\} = \{1,2\}$.}
\label{fig:levelsets}
\end{figure}

\begin{figure}[!ht]
\begin{center}
\includegraphics[scale=0.39]{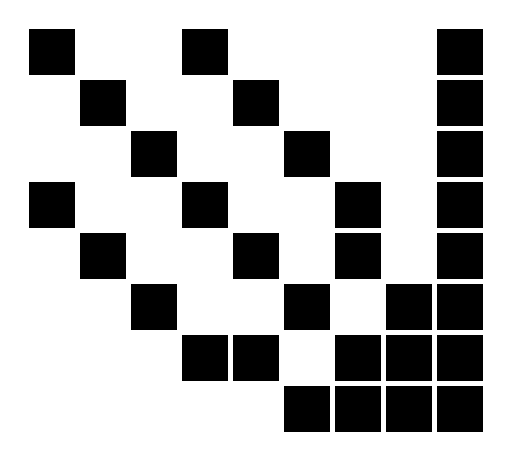} \hspace{0.8em}
\includegraphics[scale=0.39]{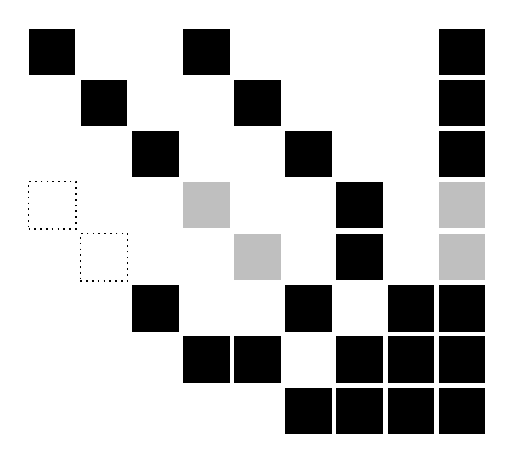} \hspace{0.8em}
\includegraphics[scale=0.39]{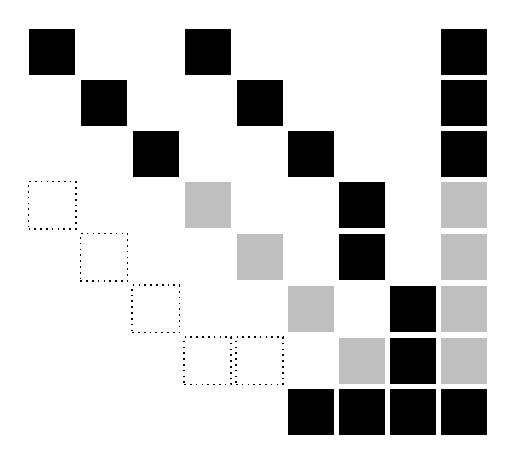} \hspace{0.8em}
\includegraphics[scale=0.39]{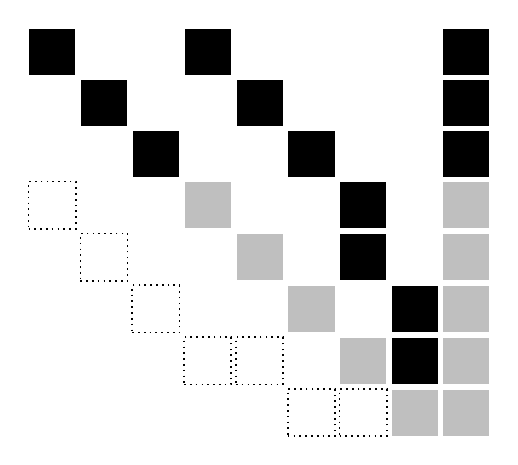}
\end{center}
\caption{Illustration of Step 2 (row-reduction process) of Algorithm~\ref{alg:GaBP} for the tree in Figure \ref{fig:levelsets}. The matrix represents $[A \; | \; c]$. The black squares represent unchanged non-zero values of $A$ and the grey squares represent non-zero values that are updated at some iteration in Step 2. In the final matrix (far right), the values in the last column are the values assigned to the vector $C$ in Steps 1 and 2 above, while the remaining columns that form an upper triangular matrix are the values corresponding to the constructed $P$ matrix. The backward solve of Step 3 solves the linear system.}
\label{fig:rowReduce}
\end{figure}

To illustrate the message-passing algorithm in the terminology of Gaussian elimination, 
%assume that the rows of the sub-matrix $\tilde{A} = A_{bb}$ and entries of $\tilde{c} = c_b - A_{b\bar{b}}x_{\bar{b}}$ correspond to nodes in descending order of degree so that $\tilde{A}_{11}$ corresponds to the root of the tree (node with maximal degree in $G_b$) and $\diag(\tilde{A}_{(|b| - l:|b|)(|b| - l:|b|)})$ correspond to the $l$ leaves of the tree (nodes that are only connected to one other node in $G_b$). By this assumption, $\tilde{A}$ has the specific form that each column has at most one non-zero entry above the diagonal. 
we first need to divide the nodes $\{1, 2, \dots, n\}$ in the forest into sets $L\{1\}$, $L\{2\}, \dots, L\{T\}$, where $L\{1\}$ is an arbitrary node in the graph $G$ selected to be the ``root'' node, $L\{2\}$ is the set of all neighbours of the ``root'' node, $L\{3\}$ is the set of all neighbours of the nodes in $L\{2 \}$ excluding parent nodes (nodes in $L\{1\!:\!2\}$), and so on until all nodes are assigned to a set (if the forest is made of disconnected trees, we will need to do this for each tree). An example of this process is depicted in Figure~\ref{fig:levelsets}.
%Observe that by eliminating the nodes in $L\{1\}$ in $G_b$, the nodes in $L\{2\}$ will become the leaf nodes in a new tree graph. We can then eliminate the nodes in $L\{2\}$ and continue eliminating until we reach the root, as shown in Figure~\ref{fig:levelsets}. 
%This procedure can reduce the matrix to a lower triangular form (that has a linear number of non-zeroes) and 
Once these sets are initialized, we start with the nodes furthest from the root node(s) $L\{T\}$, and carry out the row operations of Gaussian elimination moving towards the root. Then we use backward substitution to solve the system $Ax = c$. We outline the full procedure in Algorithm~\ref{alg:GaBP} and illustrate the algorithm on a particular example in Figure~\ref{fig:rowReduce}.

\subsection{Forest-Structured BCD Updates}
\label{subsec:solvingforestsBCD}

%Indeed, for some special cases like a quadratic function with a lattice-structured dependency, this allows us to update half the nodes with a cost
%Performing a Newton-like update with a fixed $H_{b_k}$ or performing an optimal update
%Given the amazing properties that we get using the greedy BCD method for solving LASSO and dual SVM problems, it is natural to question if efficient exact updates exist for other problems. The obvious choice is quadratic problems, as we can do exact updates on our reduced subproblem for a cost of $O(|b|^3)$. This allows us to update blocks that are thousands of variables, but if we narrow our consideration to {\em sparse} quadratics, we can do exact updates for much larger block sizes. In this section we show how a message passing algorithm from Gaussian Markov random fields can be used to exactly update a block for an $O(|b|)$ cost for sparse quadratic functions. We note that this approach could be used to implement the Newton step for non-quadratic objectives but we do not explore that in this work.

%Recall the previously defined problem of interest $h_2$,
%\[
%	h_2(x)  := \sum_{i \in V(G)}g_i(x_i) + \sum_{(i,j)\in E(G)}f_{ij}(x_i,x_j),
%\]
%where $G$ is a graph with a set of vertices $V(G)$ and edges $E(G)$. 

\blu{To illustrate how being able to solve forest-structured linear systems in linear time can be used within BCD methods, consider  the basic quadratic minimization problem
\[
	\argmin{x \in \R^n} \frac{1}{2} x^T A x - c^T x,
\]
where we assume the matrix $A \in \R^{n \times n}$ is positive-definite and sparse. By excluding terms not depending on the coordinates in the block, the optimal update for block $b$ is given by the solution to the linear system
\begin{equation}
\label{eq:messageProb}
A_{bb}x_b = \tilde{c},
%\argmin{x_{b} \in \R^{|b|}} \frac{1}{2}x_{b}^TA_{bb}x_b - \tilde{c}^Tx_b,
\end{equation}
where $A_{bb} \in \R^{|b| \times |b|}$ is the submatrix of $A$ corresponding to block $b$, and $\tilde{c} = c_b - A_{b\bar{b}}x_{\bar{b}}$ is a vector with $\bar{b}$ defined as the complement of $b$ and $A_{b\bar{b}} \in \R^{|b| \times |\bar{b}|}$ is the submatrix of $A$ with rows from $b$ and columns from $\bar{b}$. We note that in practice efficient BCD methods will typically track $Ax$ (to implement the iterations efficiently) so computing $\tilde{c}$ is efficient. }
%Observe that~\eqref{eq:messageProb} is a simple convex quadratic and we simply need to solve the corresponding linear system involving $A_{bb}$. 

% The belief propagation update is given by
%\[
%	f_b(x) = (x^{(b)})^T A_{bb} x^{(b)} - \tilde{c}^T x^{(b)},
%\]
%where $A_{bb} \in \R^{n_b \times n_b}$ is the submatrix of $A$ corresponding to block $b$ and $\tilde{c} = c^{(b)} - A_{bb} x^{(\sim b)} \in \R^{n_b}$. 
%The time complexity of computing $\tilde{c}$ is $O(n_b \cdot (n - n_b))$, where $n_b$ is the size of block $b$. This might be costly to compute if the total number of coordinates $n$ is extremely large. We can reduce the time complexity by keeping track of the changed values of $\tilde{c}$ at every iteration (do we actually do this in practice?).

%In the special case of $h_2$, the edges are given by the pairs of variables $(i,j)$ that appear together in a function $f_{ij}$ (but we may also have a sparse graph for problem $h_1$). 

\blu{The graph obtained from the sub-matrix $A_{bb}$ is called the \emph{induced subgraph} $G_b$ of the original graph $G$. Specifically, the nodes $V_b \in G_b$ are the coordinates in the set $b$, while the edges $E_b \in G_b$ are all edges $(i,j) \in E$ where $i,j \in V_b$ (edges between nodes in $b$). Even if the original problem is not a forest-structured graph, we can \emph{choose the block $b$ so that the induced subgraph $G_b$ is forest-structured}. Using blocks constructed in this way, we can compute the optimal update~\eqref{eq:messageProb}
 in linear time using message passing
 as described in the previous seciton. }
 
\blu{  For non-quadratic problems, the optimal update is not the solution of a linear system. Nevertheless, we can choose blocks so that the induced subgraph is forest-structured and  use message-passing to efficiently compute matrix updates (which leads to a linear system involving $H_b$) or Newton updates (which leads to a linear system involving the sub-Hessian). We note that similar ideas have recently been explored by~\citet{srinivasan2015} for Newton methods. Their \emph{graphical Newton} algorithm can solve the Newton system in $O(t^3)$ times the size of $G$, where $t$ is the ``treewidth'' of the graph ($t=1$ and the graph size is at most linear for forests). However, the tree-width of $G$ is usually large while it is more reasonable to assume that we can find low-treewidth induced subgraphs $G_b$.\footnote{\blu{\citet{srinivasan2015} actually  show a more general result depending on the width of the Hessian's computation graph rather than simply its sparsity structure. However, this quantity is also typically non-trivial for the Hessian with respect to all variables, and it is not immediately apparent in general how one would choose blocks to yield a computation graph with a low treewidth.}}}
 
\begin{figure}[!ht]
\centering
\subfloat[Red-black, $|b| = n/2$.]{\label{fig:RB}\includegraphics[width=0.333\textwidth]{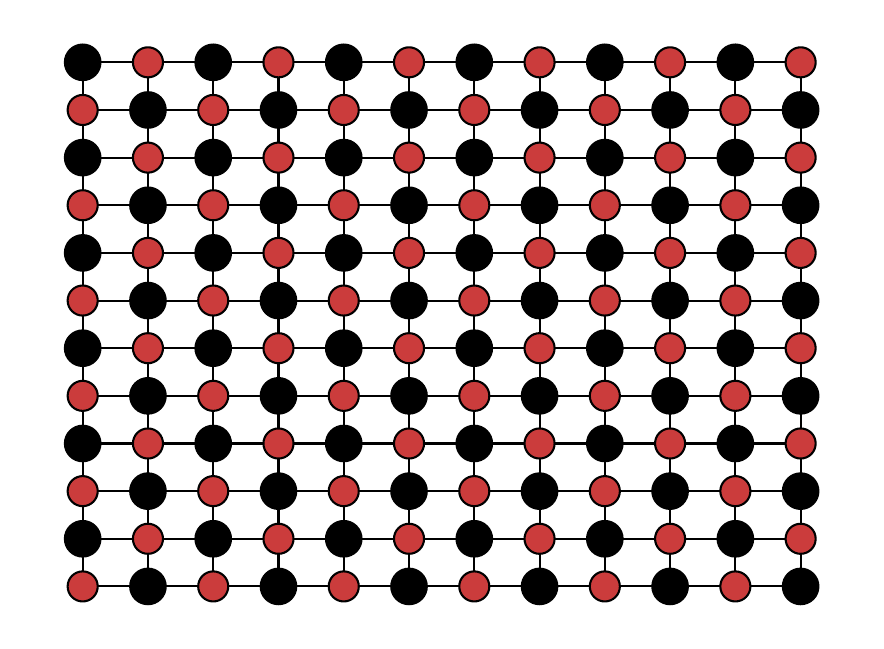}}
\subfloat[Fixed block, $|b| = n/2$.]{\label{fig:FT}\includegraphics[width=0.333\textwidth]{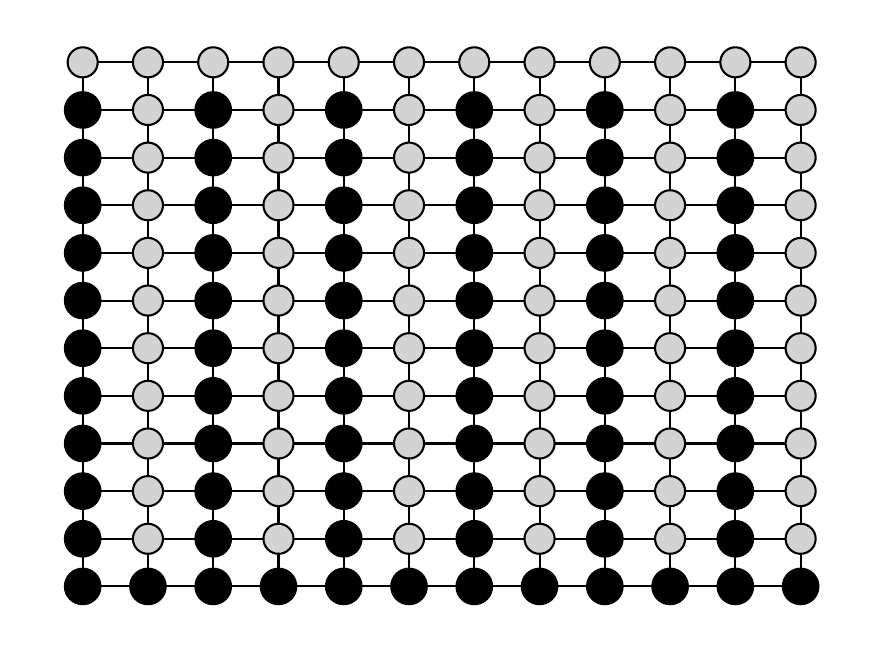}}
\subfloat[Variable block, $|b| \approx 2n/3$.]{\label{fig:VT}\includegraphics[width=0.333\textwidth]{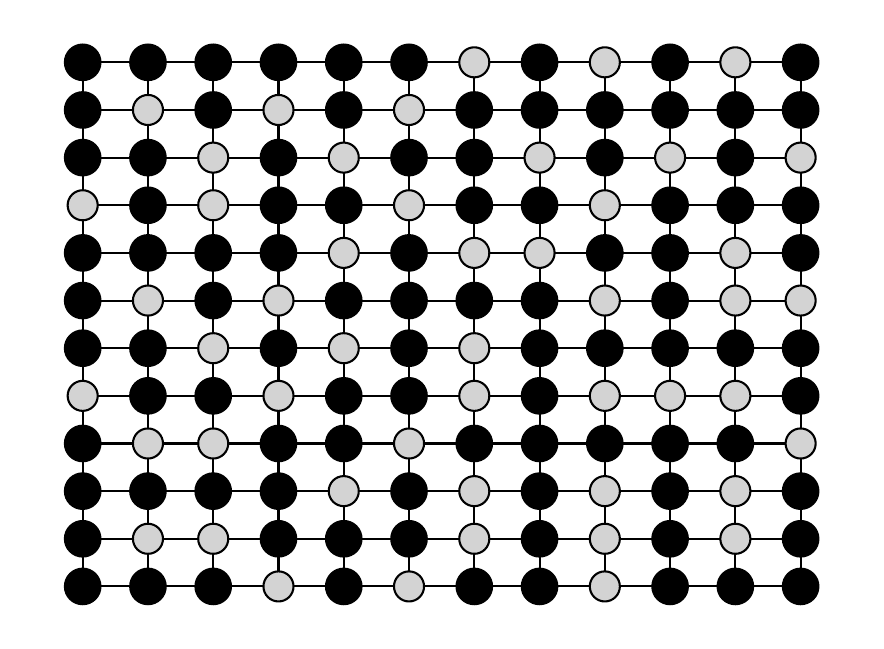}}
\caption{Partitioning strategies for defining forest-structured blocks.}
\label{fig:lattice}
\end{figure}

\blu{Whether or not message-passing is useful will depend on the sparsity pattern of $A$. For example, if $A$ is dense then we may only be able to choose forest-structured induced-subgraphs of size 2. On the other hand, if $A$ is diagonal then this would correspond to a disconnected graph, which is clearly a forest and would allow an optimal BCD update with a block size of $n$ (so we would solve a quadratic problem in 1 iteration). As an example between these extremes, consider a quadratic function with a lattice-structured non-zero pattern as in Figure~\ref{fig:lattice}. This graph is a bipartite graph (``two-colourable''), and a classic fixed partitioning strategy for problems with this common structure is to use a ``red-black ordering'' (see Figure~\ref{fig:RB}). Choosing this colouring makes the matrix $A_{bb}$ diagonal when we update the red nodes, allowing us to solve~\eqref{eq:messageProb} in linear time  (and similarly for the black nodes). So this colouring scheme allows an $O(n)$-time optimal BCD update with a block size of $n/2$, rather than the $O(n^3)$-time required for the optimal update of an arbitrary block of size $n/2$.}

Message passing allows us to go beyond the red-black colouring, and instead update \emph{any} forest-structured block of size $\tau$ in linear time. For example, the partition given in Figure~\ref{fig:FT} also has blocks of size $n/2$, but these blocks include dependencies. Our experiments indicate that blocks that maintain dependencies, such as Figure~\ref{fig:FT}, make substantially more progress than using the red-black blocks or using small non-forest-structured blocks \blu{(though red-black blocks are very well-suited for parallelization)}.

Alternatively, we can consider variable blocks with forest-structured induced subgraphs, where the block size may vary at each iteration, but restricting to forests still leads to a linear-time update. As seen in Figure~\ref{fig:VT}, by allowing variable block sizes we can select a forest-structured block of size $|b| \approx 2n/3$ in one iteration (black nodes) while still maintaining a linear-time update. If we further sample different random forests or blocks at each iteration, then the convergence rate under this strategy is covered by the arbitrary sampling theory~\citep{zheng2014quartz}. Also note that the maximum of the gradient norms over all forests defines a valid norm, so our analysis of Gauss-Southwell can be applied to this case.

\subsection{\blu{From Red-Black to Multi-Colourings}}
\label{sec:redblack}

\blu{ The adjacency matrix in the red-black example is called ``consistently ordered'' in the numerical linear algebra literature, meaning that the nodes can be partitioned into sets such that any two adjacent nodes belong to different sets~\citep[Def. 5.3.2]{young1971}. A matrix being consistently-ordered is equivalent to the graph being  ``two-colourable'' in the language of graph theory: each node can be assigned one of two colours, such that no adjacent nodes receive the same colour.}
% In the case of red-black ordering, all odd numbered sets would make up one block and all even numbered sets would make up another block. 
We can generalize the red-black approach to arbitrary graphs by defining our blocks such that no two neighbours are in the same block. While for lattice-structured graphs we only need two blocks to do this, for general graphs we may need a larger number of blocks. \blu{This is called a multi-colouring in the numerical linear algebra literature~\citep[\S 12.4]{harrar1993orderings,saad2003iterative}, and may lead to blocks of different sizes.
 If a graph is $\nu$-colourable, then we can arrange the adjacency matrix such that it has $\nu$ diagonal blocks along its diagonal (the off-diagonal blocks may be dense). In relation to BCD methods, this means that $A_{bb}$ is diagonal for each block and we can update each of the $\nu$ blocks in linear time in the size of the block.}  %Further, we can extend these ideas to general graphs, where we could use any forest-structured partition of the nodes. 

Finding the minimum number of ``colours'' (number of blocks) for a given graph is exactly the graph colouring problem, which is NP-hard~\citep{garey1979}. However, there are various heuristics that quickly give a (potentially non-minimal) valid colouring of the nodes~(\blu{for a survey of heuristic, meta-heuristic, and hybrid methods for graph colouring, see~\citet{baghel2013}}).  For example, in our experiments we used the following classic greedy algorithm~\citep{welsh1967upper}:
\begin{enumerate}
\item Proceed through the vertices of the graph in some order $i=1,2,\dots,n$.
\item For each vertex $i$, assign it the smallest positive integer (``colour'') such that it does not have the same colour as any of its neighbours among the vertices $\{1,2,\dots,i-1\}$.
\end{enumerate}
We can use all vertices assigned to the same integer as our blocks in the algorithm, and if we apply this algorithm to a lattice-structured graph (using row- or column-ordering of the nodes) then we obtain the classic red-black colouring of the graph.

\subsection{Partitioning into Forest-Structured Blocks}
\label{sec:treePart}

Instead of disconnected blocks, in this work we instead consider forest-structured blocks. 
The size of the largest possible forest is also related to the graph colouring problem~\citep{esperet2015}, and we can consider a slight variation on the second step of the greedy colouring algorithm to find a set of forest-structured blocks:
\begin{enumerate}
\item Proceed through the vertices of the graph in some order $i=1,2,\dots,n$.
\item For each vertex $i$, assign it the smallest positive integer (``forest'') such that the nodes assigned to that integer among the set $\{1,2,\dots,i\}$ form a forest.
\end{enumerate}
If we apply this to a lattice-structured graph (in column-ordering), this generates a partition into two forest-structured graphs similar to the one in Figure~\ref{fig:FT} (only the bottom row is different). This procedure requires us to be able to  test whether adding a node to a forest maintains the forest property, and we show how to do this efficiently in~\ref{app:forest}.

In the case of lattice-structured graphs there is a natural ordering of the vertices, but for many graphs there is no natural ordering. In such cases we might simply consider a random ordering. Alternately, if we know the individual Lipschitz constants $L_i$, we could order by these values (with the largest $L_i$ going first so that they are likely assigned to the same block if possible). In our experiments we found that this ordering improved performance for an unstructured dataset, and performed similarly to using the natural ordering in a lattice-structured dataset.

\subsection{Approximate Greedy Rules with Forest-Structured Blocks}
\label{sec:treeApprox}

Similar to the problems of the previous section, computing the Gauss-Southwell rule over forest-structured variable blocks is NP-hard, as we can reduce the 3-satisfiability problem to the problem of finding a maximum-weight forest~\citep{garey1979}. However, we use a similar greedy method to approximate the greedy Gauss-Southwell rule over the set of trees:
\begin{enumerate}
\item Initialize $b_k$ with the node $i$ corresponding to the largest gradient, $|\nabla_i f(x^k)|$.
\item Search for the node $i$ with the largest gradient that is not part of $b_k$ and that maintains that $b_k$ is a forest.
\item If such a node is found, add it to $b_k$ and go back to step 2. Otherwise, stop.
\end{enumerate}
Although this procedure does not yield the exact solution in general, it is appealing since (i) the procedure is efficient as it is easy to test whether adding a node maintains the forest property (see \ref{app:forest}), (ii) it outputs a forest so that the subsequent update is linear-time, (iii) we are guaranteed that the coordinate corresponding to the variable with the largest gradient is included in $b_k$ (guaranteeing convergence), % (and it will often include many more of the largest-gradient nodes), 
and (iv) we cannot add any additional node to the final forest and still maintain the forest property. A similar heuristic can be used to approximate the GSD rule under the restriction from Section~\ref{subsec:adaptive} or to generate a forest randomly.

\section{Manifold Identification}
\label{sec:exactupdates}

\blu{
The previous section advocated the use of Newton updates over matrix updates, where feasible. We usually associate Newton updates  with superlinear convergence as opposed to linear convergence. However, updating blocks of variables destroys the possibility for superlinear convergence even in settings where Newton's method would achieve superlinear convergence. In this section, we consider non-smooth problems where in some settings an appropriate block selection strategy combined with Newton updates of the blocks does lead to superlinear convergence.}

In paritcular, we consider optimization problems of the form
\begin{equation}
\label{eq:compositeproblem}
	\argmin{x \in \R^n} f(x) + \sum_{i=1}^n g_i (x_i),
\end{equation}
where $\nabla f$ is Lipschitz-continuous and % $g$ is a separable function.
%\begin{equation}\label{eq:separable}
%	g(x) = \sum_{i=1}^n g_i (x_i),
%\end{equation}
each $g_i$ only needs to be convex and  lower semi-continuous (it may be non-smooth or infinite at some $x_i$).
%where $f \in \mathcal{C}^{2}$, i.e., $\nabla f$ is Lipchitz continuous and satisfies the condition that every critical point is a global minimum (i.e., $f$ is strongly-convex or the PL-inequality is satisfied). We assume $P$ is a regularization function that has the following separable structure,
%where each $P_b$ is a closed, proper (extended-valued) convex function. We use $\B$ to denote a partitioning of coordinates $S:=\{ 1, 2, \dots, n \}$ into some finite $T$ blocks $b^{(i)}$ such that
%\begin{equation}\label{eq:gaussseidel}
%	\B = \{ b^{(1)}, b^{(2)}, \dots, b^{(T)} \} \quad \text{and} \quad b^{(1)} ~\cup~ b^{(2)} ~\cup~ \dots ~\cup~ b^{(T)} = S.
%\end{equation}
A classic example of a problem in this framework is optimization subject to non-negative constraints,
\begin{equation}\label{eq:bounds}
\argmin{x \geq 0} f(x),
\end{equation}
where in this case $g_i$ is the indicator function on the non-negative orthant,
\[
	g_i(x_i)%_{\cdot \ge 0}(x_i) 
	= \begin{cases} 0 & \text{if }x_i \ge 0, \\ \infty & \text{if }x_i < 0. \end{cases}
\]
 %(a function that is 0 if $x_i \geq 0$ and $\infty$ otherwise). 
Another example that has received significant recent attention is the case of an L1-regularizer,
\begin{equation}
\label{eq:L1}
\argmin{x \in \R^n} f(x) + \lambda \norm{x}_1,
\end{equation}
where in this case $g_i = \lambda|x_i|$. Here, the L1-norm regularizer is used to encourage sparsity in the solution.
A related problem is the group L1-regularization problem~\eqref{eq:groupL1}, where instead of being separable, $g$ is block-separable.
%Examples of such functions include regularized least squares problems using an $\ell_1$, group-$\ell_2$ or a group-$\ell_\infty$ regularizer, and regularized logistic regression problems with either an $\ell_1$ or a group-$\ell_2$ regularizer. 

Proximal gradient methods have become one of the default strategies for solving problem~\eqref{eq:compositeproblem}, and a BCD variant of these methods has an update of the form
\begin{equation}\label{eq:update1}
x^{k+1} = \prox{\alpha_k g_{b_k}}\left[x^k - \alpha_k U_{b_k}\nabla_{b_k} f(x^k)\right],
\end{equation}
where for any block $b$ and step-size $\alpha$ the proximal operator is defined by the separable optimization
\[
\prox{\alpha g_b}[x] = \argmin{y\in\mathbb{R}^n} \frac{1}{2\alpha} \|y-x \|^2 +  \sum_{i \in b}g_i(y_i).
\]
We note that all variables not in block $b$ stay at their existing values in the optimal solution to this problem. In the special case of non-negative constraints like~\eqref{eq:bounds} the update in \eqref{eq:update1} is given by
\[
x_i^{k+1} = \left[x_i - \frac{1}{L}\nabla_{i} f(x^k)\right]^+,
\]
where $[\beta]^+ = \max\{0,\beta\}$ projects onto the non-negative orthant. 
For L1-regularization problems~\eqref{eq:L1} the update reduces to an element-wise soft-thresholding step, 
\begin{equation}
\label{update:proxL1}
\begin{aligned}
x_{i}^{k+\half} & = x^k - \frac{1}{L}\nabla_{i}f(x^k),\\
x_{i}^{k+1} & = \blu{\text{sgn}(x_{i}^{k+\half})}\left[\left|x_i^{k+\half}\right| - \frac{\lambda}{L}\right]^+,
\end{aligned}
\end{equation}
which we have written as a gradient update followed by the soft-threshold operator. \blu{Here, the signum function sgn$(x_i)$ returns $1$ if $x_i$ is positive, returns -1 if it is negative, and can return any number between $-1$ and $1$ if $x_i$ is zero.}

Most of the issues discussed in the previous sections for smooth BCD methods carry over in a straightforward way to this proximal setting; we can still consider fixed or variable blocks, there exist matrix and Newton updates, and we can still consider cyclic, random, or greedy selection rules. One subtle issue is that there are many generalizations of the Gauss-Southwell rule to the proximal setting~\citep{nutini2015}. However, the GS-$q$ rule defined by~\citet{tseng2009} seems to be the generalization of GS with the best theoretical properties. A GSL variant of this rule would take the form
\begin{equation}\label{eq:GSq}
b_k \in \argmin{b \in \B}\left\{\min_d\left\{ \langle \nabla_{b} f(x^k), d\rangle + \frac{L_{b}}{2}\|d\|^2  + \sum_{i \in b}g_i(x_i + d_i) - \sum_{i\in b}g_i(x_i)\right\}\right\},
\end{equation}
where we assume that the gradient of $f$ is $L_b$-Lipschitz continuous with respect to block $b$. A generalization of the GS rule is obtained if we assume that the $L_b$ are equal across all blocks. %There also exist Newton-like proximal updates, but they present additional challenges so we'll delay discussing them until Section~\ref{sec:superlinear}
\blu{We can similarly define a GSD-$q$ and GSQ-$q$ rule by using
\begin{equation}\label{eq:GSQq}
b_k \in \argmin{b \in \B}\left\{\min_d\left\{ \langle \nabla_{b} f(x^k), d\rangle + \frac{1}{2}\|d\|_{H_b}^2  + \sum_{i \in b}g_i(x_i + d_i) - \sum_{i\in b}g_i(x_i)\right\}\right\},
\end{equation}
where the gradient of block $b$ is $1-$Lipschitz continuous with respect to the quadratic norm defined by the matrix $H_b$ (GSD-$q$ would correspond to the special case where the $H_b$ are diagonal). Convergence of this method under standard assumptions follows existing arguments~\citep{TsengYun2009,tseng2009,csiba2017global}.}

It has been established that coordinate descent and BCD methods based on the update~\eqref{eq:update1} for problem~\eqref{eq:compositeproblem} obtain similar convergence rates to the case where we do not have a non-smooth term $g$~\citep{nesterov2012efficiency,richtarik2011,nutini2015}. %However, our focus in this section is not on showing that proximal methods maintain the same convergence rate as their smooth variant  despite the addition of a non-smooth term $g$. Instead,
The focus of this section is to show that \emph{the non-smoothness of $g$ can actually lead to a faster convergence rate}. 

This idea dates back at least 40 years to the work of~\citet{bertsekas1976goldstein}.\footnote{A similar property was shown for proximal point methods in a more general setting around the same time,~\citep{rockafellar1976}.} For the case of non-negative constraints, he shows that the sparsity pattern of $x^k$ generated by the projected-gradient method matches the sparsity pattern of the solution $x^*$ for all sufficiently large $k$. Thus, after a finite number of iterations the projected-gradient method will ``identify'' the final set of non-zero variables. Once these values are identified,
 Bertsekas suggests that we can fix the zero-valued variables and apply an unconstrained Newton update to the set of non-zero variables to obtain superlinear convergence. Even without switching to a superlinearly-convergent method, the convergence rate of the projected-gradient method can be faster once the set of non-zeroes is identified since it is effectively optimizing in the lower-dimensional space corresponding to the non-zero variables.
 
This idea of identifying a smooth ``manifold'' containing the solution $x^*$ has been generalized to allow polyhedral constraints~\citep{burke1988identification}, general convex constraints~\citep{wright1993identifiable}, and even non-convex constraints~\citep{hare2004}. Similar results exist in the proximal gradient setting. For example, it has been shown that the proximal gradient method identifies the sparsity pattern in the solution of L1-regularized problems after a finite number of iterations~\citep{hare2011}. The active-set identification property has also been shown for other algorithms like certain coordinate descent and stochastic gradient methods~\citep{mifflin2002,wright2012,lee2012manifold}. Specifically, Wright shows that BCD also has this manifold identification property for separable $g$~\citep{wright2012}, provided that the coordinates are chosen in an essentially-cyclic way (or provided that we can simultaneously choose to update all variables that do not lie on the manifold). Wright also shows that superlinear convergence is possible if we use a Newton update on the manifold, assuming the Newton update does not leave the manifold.

In the next subsection, we present a manifold identification result for proximal coordinate descent for general separable $g$. We follow a similar argument to~\citet{bertsekas1976goldstein}, which yields a simple proof that holds for many possible selection rules including greedy rules (which may not be essentially-cyclic). Further, our argument leads to bounds on the {\em active-set complexity} of the method, which is the number of iterations required to reach the manifold~\citep{liang2017activity,nutini2017manifold}. As examples, we consider problems~\eqref{eq:bounds} and~\eqref{eq:L1}, and show explicit bounds for the active-set complexity. We subsequently show how to generalize this argument to cases like block updates and Newton updates. The latter naturally leads to superlinear convergence of greedy BCD methods when using variable blocks of size larger than the dimension of the manifold.

\subsection{Manifold Identification for Separable $g$}% for General $g$}
\label{subsec:mani}

Assume the subdifferential of $g$ is nonempty for all $x \in\! \dom g$. By our separability assumption on $g$, the subdifferential of $g$ can be expressed as the concatenation of the individual subdifferential of each $g_i$, where the subdifferential of $g_i$ at any $x_i \in \R$ is defined by
\[
	\partial g_i(x_i) = \{ v \in \R : g_i(y) \ge g_i(x_i) + v \cdot (y-x_i), \text{ for all $y \in \!\dom{g_i}$}\}.
\]
This implies that the subdifferential of each $g_i$ is just an interval on the real line. In particular, the interior of the subdifferential of each $g_i$ at a non-differentiable point $x_i$ can be written as an open interval,
\begin{equation}
\label{eq:subinterior}
	 \interior \partial g_i(x_i) \equiv (l_i, u_i ),
\end{equation}
where $l_i \in \R \cup \{-\infty\}$ and $u_i \in \R \cup \{\infty\}$ (the $\infty$ values occur if $x_i$ is at its lower or upper bound, respectively). The {\em active-set} at a solution $x^*$ for a separable $g$ is then defined by
\[
	\mathcal{Z} = \{ i : \partial g_i(x_i^*) \text{ is not a singleton} \}.
\]
By \eqref{eq:subinterior}, the set $\mathcal{Z}$ includes indices $i$ where $x_i^*$ is equal to the lower bound on $x_i$, is equal to the upper bound on $x_i$, or occurs at a non-smooth value of $g_i$. In our examples of non-negative constraints or L1-regularization, $\mathcal{Z}$ is the set of coordinates that are zero at the solution $x^*$.
With this definition, we can formally define the manifold identification property.
\begin{definition}
The manifold identification property for problem~\eqref{eq:compositeproblem} is satisfied if for all sufficiently large $k$, we have that $x_i^k = x_i^*$ for some solution $x^*$ for all $i \in \mathcal{Z}$.
\end{definition}

In order to prove the manifold identification property for the proximal coordinate descent method, in addition to assuming that $\nabla f$ is $L$-Lipschitz continuous, we require two assumptions. Our first assumption is that the iterates of the algorithm converge to a solution $x^*$.
\begin{assumption}
\label{assump2}
The iterates of the proximal coordinate descent method converge to an optimal solution $x^*$ of problem~\eqref{eq:compositeproblem}, that is $x^k \to x^*$ as $k \to \infty$.
\end{assumption}
This assumption holds if $f$ is strongly-convex and if we use cyclic or greedy selection~(see~\ref{append:strong}), but will also hold under a variety of other scenarios. Our second assumption is a {\em nondegeneracy} condition on the solution $x^*$ that the algorithm converges to. Below we write the standard nondegeneracy condition from the literature for the special case of~\eqref{eq:compositeproblem}.
\begin{assumption} 
\label{assump1} 
We say that $x^*$ is a nondegenerate solution for problem~ \eqref{eq:compositeproblem} if it holds that
\[
\begin{cases}
	-\nabla_i f(x^*) =  \nabla_i g(x^*_i) &\text{if $\partial g_i(x_i^*)$ is a singleton ($g_i$ is smooth at $x_i^*$)} \\
	-\nabla_i f(x^*) \in  \interior \partial g_i(x^*_i) &\text{if $\partial g_i(x_i^*)$ is not a singleton ($g_i$ is non-smooth at $x_i^*$)}.
\end{cases}
\]
\end{assumption}
This condition states that $-\nabla f(x^*)$ must be in the ``relative interior"~\citep[see][Section 2.1.3]{boyd2004convex} of the subdifferential of $g$ at the solution $x^*$. In the case of the non-negative bound constrained problem~\eqref{eq:bounds}, this requires that $\nabla_i f(x^*) > 0$ for all variables $i$ that are zero at the solution ($x_i^* = 0$). For the L1-regularization problem~\eqref{eq:L1}, this requires that $|\nabla_i f(x^*)| < \lambda$ for all variables $i$ that are zero at the solution.\footnote{Note that $|\nabla_i f(x^*)| \leq \lambda$ for all $i$ with $x_i^* = 0$ follows from the optimality conditions, so this assumption simply rules out the case where $|\nabla_i f(x_i^*)|=\lambda$. We note that in this case the nondegeneracy condition is a strict complementarity condition~\citep{santis2015}.}

There are three results that we require in order to prove the manifold identification property. The first result follows directly from Assumption~\ref{assump2} and establishes that for any $\beta > 0$ there exists a finite iteration $\bar{k}$ such that the distance from the iterate $x^{k}$ to the solution $x^*$ for all iterations $k \ge \bar{k}$ is bounded above by $\beta$. 

\begin{lemma}
\label{lem:manifold1}
Let Assumption~\ref{assump2} hold. For any $\beta$, there exists some minimum finite $\bar{k}$ such that $\| x^k - x^* \| \le \beta$ for all $k \geq \bar{k}$.
\end{lemma}

%\begin{proof}
%This follows directly from Assumption~\ref{assump2}.
%\end{proof}

The second result we require is that for any $i \in \mathcal{Z}$ such that $x_i^{k} \not = x_i^*$, eventually coordinate $i$ is selected at some finite iteration.

\begin{lemma}
\label{lem:manifold2}
Let Assumption~\ref{assump2} hold. If $x_i^k \not = x^*_i$ for some $i \in \mathcal{Z}$, then coordinate $i$ will be selected by the proximal coordinate descent method after a finite number of iterations.
\end{lemma}

\begin{proof}
For eventual contradiction, suppose we did not select such an $i$ after iteration $k'$. Then for all $k \geq k'$ we have that
\begin{equation}
\label{eq:convToZero}
|x_i^{k'} - x_i^*| = |x_i^k - x_i^*| \leq \norm{x^k - x^*}.
\end{equation}
By Assumption~\ref{assump2} the right-hand side is converging to 0, so it will eventually be less than $|x_i^{k'} - x^*_i|$ for some $k \ge k'$, contradicting the inequality. Thus after a finite number of iterations we must have that $x_i^k \neq x_i^{k'}$, which can only be achieved by selecting $i$.
\end{proof}

The third result we require is that once Lemma~\ref{lem:manifold1} is satisfied for some finite $\bar{k}$ and a particular $\beta > 0$, then for any coordinate $i \in \mathcal{Z}$ selected at some iteration $k' \ge \bar{k}$ by the proximal coordinate descent method, we have $x^{k'}_i = x^*_i$. We prove that this happens for a value $\beta$ depending on a quantity $\delta$ defined by
\begin{equation}
\label{eq:Delta}
	 \delta =  \min_{i\in\mathcal{Z}}\left\{ \min\{ -\nabla_i f(x^*) - l_i, u_i + \nabla_i f(x^*) \}\right\},
\end{equation}
which is the minimum distance to the nearest boundary of the subdifferential \eqref{eq:subinterior} among indices $i \in \mathcal{Z}$.

\begin{lemma}
\label{lem:manifold3}
Consider problem~\eqref{eq:compositeproblem}, where $f$ is convex with $L$-Lipschitz continuous gradient and the $g_i$ are proper convex functions (not necessarily smooth). Let Assumptions~\ref{assump2} be satisfied and Assumption~\ref{assump1} be satisfied for the particular $x^*$ that the algorithm converges to. Then for the \blu{coordinate-wise proximal-gradient iteration~\eqref{eq:update1}} with a step-size of $1/L$, if $\norm{x^k - x^*} \leq \delta/2L$ holds and $i \in \mathcal{Z}$ is selected at iteration $k$, then $x_i^{k+1} = x^*_i$.
\end{lemma}
\begin{proof}
\blu{We obtain this result by applying~\citep[Lemma~1]{nutini2017manifold}, but restricting the function to the coordinate being updated.}
%The proof is identical to the case of proximal-gradient updates under the same step-size~\citep[Lemma~1]{nutini2017manifold}, but restricting to the update of the single coordinate. 
%See~\ref{append:manifold}.
\end{proof}
 With the above results we next have the manifold identification property.
\begin{theorem}
\label{thm:manifoldgeneral}
Consider problem ~\eqref{eq:compositeproblem}, where $f$ is convex with $L$-Lipschitz continuous gradient and the $g_i$ are proper convex functions. Let Assumptions~\ref{assump2} be satisfied and Assumption~\ref{assump1} be satisfied for the particular $x^*$ that the algorithm converges to. Then for the \blu{coordinate-wise proximal-gradient iteration~\eqref{eq:update1}} a step-size of $1/L$, there exists a finite $k$ such that $x_i^k = x_i^*$ for all $i \in \mathcal{Z}$.
\end{theorem}

\begin{proof}
Lemma~\ref{lem:manifold1} implies that the assumptions of Lemma~\ref{lem:manifold3} are eventually satisfied, and combining this with  Lemma~\ref{lem:manifold2} we have the result.
\end{proof}

Both problems~\eqref{eq:bounds} and~\eqref{eq:L1} satisfy the manifold identification result. By the definition of $\delta$ in~\eqref{eq:Delta}, we have that $\delta = \min_{i \in \mathcal{Z}} \{ \nabla_i f(x^*) \}$ for problem~\eqref{eq:bounds}. We note that if $\delta = 0$, then we may approach the manifold through the interior of the domain and the manifold may never be identified (this is the purpose of the nondegeneracy condition). For problem~\eqref{eq:L1}, we have that $\delta = \lambda - \max_{i \in \mathcal{Z}} \{ | \nabla_i f(x^*) | \}$. From these results, we are able to define explicit bounds on the number of iterations required to reach the manifold, a new result that we explore in the next subsection.

Instead of using a step-size of $1/L$, it is more common to use a bigger step-size of $1/L_{i}$ within coordinate descent methods, where $L_{i}$ is the coordinate-wise Lipschitz constant. In this case, the results of Lemma~\ref{lem:manifold3} hold for $\beta = \delta/(L+L_i)$. This is a larger region since $L_i \leq L$, so with this standard step-size the iterates can move onto the manifold from further away and we expect to identify the manifold earlier. The argument can also be modified to use other step-size selection methods, provided that we can write the algorithm in terms of a step-size $\alpha_k$ that is guaranteed to be bounded from below. 
While the above result considers single-coordinate updates, it can trivially be modified to show that BCD with larger block sizes (and gradient updates) have the manifold identification property. The only change is that once $\norm{x^k - x^*} \leq \delta/2L$, we have that $x_i^{k+1} = x_i^*$ for all $i \in b_k \cap \mathcal{Z}$. Thus, BCD methods can simultaneously move many variables onto the optimal manifold.

\subsection{Active-Set Complexity}
\label{sec:maniRate}

The manifold identification property of the previous section could also be shown using the more sophisticated tools used in related works~\citep{burke1988identification, hare2004}. However, an appealing aspect of the simple argument above is that it can be combined with non-asymptotic convergence rates of the iterates to bound the \emph{number of iterations required to reach the manifold}. \blu{This was first explored by~\citet{liang2017activity} and we call this the ``active-set complexity'' of the method. We have recently given tighter bounds than Liang et al. on the active-set complexity of proximal-gradient methods for strongly-convex objectives~\citep{nutini2017manifold}}. Here we use a similar analysis to bound the active-set complexity of BCD methods, which is complicated by the fact that not all coordinates are updated on each iteration.

We will consider the active-set complexity in the case where $f$ is strongly-convex and we use cyclic or greedy selection. This guarantees that
\begin{equation}\label{eq:iterBound}
\norm{x^k - x^*} \leq \left (1-\frac{1}{\kappa} \right )^k \gamma,
\end{equation}
for some $\gamma \geq 0$ and some $\kappa \geq 1$ (see~\ref{append:strong}, and we note that this type of rate also holds for a variety of other types of selection rules). 
By using that $(1 - 1/\kappa)^k \leq \exp(- k/\kappa)$, the linear convergence rate~\eqref{eq:iterBound} implies the following result on how many iterations it will take to identify the active-set, and thus reach the manifold.
\begin{theorem}
\label{thm:activeset}
Consider any method that achieves an iterate bound~\eqref{eq:iterBound}. 
For $\delta$ as defined in~\eqref{eq:Delta}, we have $\norm{x^{\bar{k}} - x^*} \leq \delta/2L$ after at most $\kappa \log(2L\gamma/\delta)$ iterations. Further, we will identify the active-set after an additional $t$ iterations, where $t$ is the number of additional iterations required to select all suboptimal $x_i$ with $i \in \mathcal{Z}$ as part of some block.
\end{theorem}
\begin{proof}
\blu{See Corollary~1 in~\citet{nutini2017manifold} for bounding the time until we are ``close'' to the solution, in the sense that $\norm{x^{\bar{k}} - x^*} \leq \delta/2L$. It then follows from applying~\citet[Lemma~1]{nutini2017manifold} that each time a suboptimal $x_i$ with $i \in \mathcal{Z}$ is selected that it will set $x_i^k=x_i^*$. Thus, by how $t$ is defined, we will identify the active-set $t$ iterations after we become ``close'' to the solution.}
\end{proof}
 The value of $t$ depends on the selection rule we use. If we use cyclic selection we will require at most $t = n$ additional iterations to select all suboptimal coordinates $i \in \mathcal{Z}$ and thus to reach the optimal manifold. To bound the active-set complexity for general rules like greedy rules, we cannot guarantee that all coordinates will be selected after $n$ iterations once we are close to the solution. In the case of non-negative constraints, the number of additional iterations depends on a quantity we will call $\epsilon$, which is the smallest non-zero variable $x_i^{\bar{k}}$ for $i \in \mathcal{Z}$ and $\bar{k}$ satisfying the first part of Theorem~\ref{thm:activeset}. It follows from~\eqref{eq:iterBound} that we require at most $\kappa \log(\gamma/\epsilon)$ iterations beyond $\bar{k}$ to select all non-zero $i \in \mathcal{Z}$. Thus, the active-set complexity for greedy rules for problem~\eqref{eq:bounds} is bounded above by 
 %$\kappa(\log(2L\gamma/\delta) + [\log(\gamma/\epsilon) - \log(2L\gamma/\delta)]^+)$. 
 $\kappa(\log(2L\gamma/\delta) + \log(\gamma/\epsilon))$. 
 Based on this bound, greedy rules (which yield a smaller $\kappa$) may identify the manifold more quickly than cyclic rules in cases where $\epsilon$ is large. However, if $\epsilon$ is very small then greedy rules may take a larger number of iterations to reach the manifold.\footnote{If this is a concern, the implementer could consider a safeguard ensuring that the method is essentially-cyclic. Alternately, we could consider rules that prefer to include variables that are near the manifold and have the appropriate gradient sign.} 
%It is interesting to note that this bound only depends logarithmically on $1/\delta$, so for larger values of $\delta$ this condition will be satisfied in fewer iterations. 

\subsection{Proximal-Newton Updates and Superlinear Convergence}
\label{sec:superlinear}

Once we have identified the optimal manifold, we can think about switching from using the proximal BCD method to using an unconstrained optimizer on the coordinates $i \not\in \mathcal{Z}$. The unconstrained optimizer can be a Newton update, and thus under the appropriate conditions can achieve superlinear convergence. However, a problem with such ``2-phase'' approaches is that we do not know the exact time at which we reach the optimal manifold. This can make the approach inefficient: if we start the second phase too early, then we sub-optimize over the wrong manifold, while if we start the second phase too late, then we waste iterations performing first-order updates when we should be using second-order updates. Wright proposes an interesting alternative where at each iteration we consider replacing the gradient proximal-BCD update with a Newton update on the current manifold~\citep{wright2012}. This has the advantage that the manifold can continue to be updated, and that Newton updates are possible as soon as the optimal manifold has been identified.
However, note that the dimension of the current manifold might be larger than the block size and the dimension of the optimal manifold, so this approach can significantly increase the iteration cost for some problems. 

Rather than ``switching'' to an unconstrained Newton update, we can alternately take advantage of the superlinear converge of proximal-Newton updates~\citep{lee2012proximal}. For example, in this section we consider Newton proximal-BCD updates as in several recent works~\citep{qu2015sdna,tappenden2016,fountoulakis2015flexible}. For a block $b$ these updates have the form
\begin{equation}
\label{eq:proxNewton}
x_b^{k+1} \in \argmin{y\in\mathbb{R}^{|b|}}\left\{\langle \nabla_b f(x_b^k),y-x_b^k\rangle + \frac{1}{2\alpha_k}\norm{y-x_b^k}_{H_b^k}^2 + \sum_{i\in b}g_i(y_i)\right\},
\end{equation}
where $H_b^k$ is the matrix corresponding to block $b$ at iteration $k$ (which can be the sub-Hessian $\nabla_{bb}^2 f(x^k)$). \blu{This is a generalization of the matrix update~\eqref{eq:matrix}, and that update is obtained in the special case where the $g_i$ are constant.}
As before if we set $H_b^k = H_b$ for some fixed matrix $H_b$, then we can take $\alpha_k = 1$ if block $b$ of $f$ is 1-Lipschitz continuous in the $H_b$-norm. \blu{Alternately, if we set $H_b^k = \nabla_{b_k,b_k}^2 f(x^k)$ then we obtain a generalization of Newton's method.}

%We note that while the arguments above are for the case of single-coordinate updates, our results imply the results of Nutini et. al.\ ~\cite{nutini2017manifold}, which show that proximal gradient methods also have the manifold identification property. For the full proximal gradient method, once $\| x^k - x^* \|$ is small enough we do not have to ensure that all $i \in \mathcal{Z}$ are selected. Rather, the method will identify the manifold in one additional iteration as it carries out an update on the full iterate.

% apply the proximal operator to the Newton-like BCD step,
%\[
%x^{k+1} = \prox{\alpha_k g_{b_k}}\left[x^k - \alpha_k U_{b_k}H_{b_k}^{-1}\nabla_{b_k} f(x^k)\right],
%\]
%but where the $H_b$-norm is used in the proximal operator,
%\[
%\prox{\alpha g_b}[y] = \argmin{x\in\mathbb{R}^n} \frac{1}{2\alpha} \|x-y \|_{H_b}^2 +  \sum_{i \in b}g_i(x_i).
%\]
%As before, we can set $\alpha_k = 1$ if 
%\[
%x^{k+1} = x^k + \beta(\bar{x}_\alpha^k - x^k),
%\]
%where $\beta$ is a second step-size and we've defined 
%\[
%\bar{x}_{\alpha}^k \in \argmin{x\in\mathbb{R}^n} \frac{1}{2\alpha}\norm{x-y}_{H_b}^2 + \sum_{i \in b}g_i(x_i).
%\]
%\[
%b_k \in \argmin{b \in \B}\left\{\min_d\left\{ \langle \nabla_{b} f(x^k), d\rangle + \frac{L_{b}}{2}\|d\|^2  + \sum_{i \in b}g_i(x_i + d_i) - \sum_{i\in b}g_i(x_i)\right\}\right\},
%\]

In the next section, we give a practical variant on proximal-Newton updates that also has the manifold identification property under standard assumptions\footnote{A common variation of the proximal-Newton method solves~\eqref{eq:proxNewton} with $\alpha_k = 1$ and then sets $x^{k+1}$ based on a search along the line segment between $x^k$ and this solution~\citep{schmidt2010,fountoulakis2015flexible}. This variation does \emph{not} have the manifold identification property; only when the line-search is on $\alpha_k$ do we have this property.}.
An advantage of this approach is that the block size typically restricts the computational complexity of the Newton update (which we discuss further in the next sections). Further, superlinear convergence is possible in the scenario where the coordinates $i \not\in \mathcal{Z}$ are chosen as part of the block $b_k$ for all sufficiently large $k$. However, note that this superlinear scenario only occurs in the special case where we use a \emph{greedy rule with variable blocks} and where the \emph{size of the blocks is at least as large as the dimension of the optimal manifold}. 
 With variable blocks, the GS-$q$ and GSL-$q$ rules~\eqref{eq:GSq} or GSD-$q$ and GSQ-$q$ rules~\eqref{eq:GSQq} will no longer select coordinates $i \in \mathcal{Z}$ since their optimal $d_i$ value is zero when close to the solution and on the manifold. Thus, these rules will only select $i \not\in\mathcal{Z}$ once the manifold has been identified.\footnote{A subtle issue is the case where $d_i = 0$ in~\eqref{eq:GSq}, but $i \not\in\mathcal{Z}$. In such cases we can break ties by preferring coordinates $i$ where $g_i$ is differentiable so that the $i \not\in\mathcal{Z}$ are included.} In contrast, we would not expect superlinear convergence for fixed blocks unless all $i \not\in\mathcal{Z}$ happen to be in the same partition. While we could show superlinear convergence of subsequences for random selection with variable blocks, the number of iterations between elements of the subsequence may be prohibitively large.\footnote{\blu{Since the first version of this work was released,~\citet{lopes2019accelerating} consider decreasing the probability of sampling the variables at zero over time. This may lead to cheaper iterations than using greedy rules while still having the potential to achieve superlinear convergence if implemented carefully.}}

\subsection{Practical Proximal-Newton Methods}
\label{sec:tmp}

A challenge with using the update~\eqref{eq:proxNewton} in general is that the optimization is non-quadratic (due to the $g_i$ terms) and non-separable (due to the $H_b^k$-norm). If we make the $H_b^k$ diagonal, then the objective is separable, but this  destroys the potential for superlinear convergence. %\footnote{We could generalize the GSD and GSQ rule to this setting based on this update, but the cost of solving~\eqref{eq:proxNewton} suggests that we should to restrict to diagonal matrices in the selection rule.} 
Fortunately, a variety of strategies exist in the literature to allow non-diagonal $H_b^k$. 

For example, for bound constrained problems we can apply two-metric projection (TMP) methods, which use a modified $H_b^k$ and allow the computation of a (cheap) projection under the Euclidean norm~\citep{gafni1982}. This method splits the coordinates into an ``active-'' set and a ``working-'' set, where the active-set $\mathcal{A}$ for non-negative constraints would be
\[
	\mathcal{A} = \{ i ~|~ x_i < \epsilon , \nabla_i f(x) > 0\},
\]
for some small $\epsilon$ while the working-set $\mathcal{W}$ is the compliment of this set. So the active-set contains the coordinates corresponding to the variables that we expect to be zero while the working-set contains the coordinates corresponding to the variables that we expect to be unconstrained. The TMP method can subsequently use the update
\begin{align*}
	x_{\mathcal{W}} &\leftarrow \proj{C} \left ( x_{\mathcal{W}} - \alpha H_{\mathcal{W}}^{-1} \nabla_{\mathcal{W}} f(x) \right ) \\
	x_{\mathcal{A}} &\leftarrow \proj{C} \left ( x_{\mathcal{A}} - \alpha \nabla_{\mathcal{A}} f(x) \right ).
\end{align*}	
This method performs a gradient update on the active-set and a Newton update on the working-set.~\citet{gafni1982} show that this preserves many of the essential properties of projected-Newton methods like giving a descent direction, converging to a stationary point, and superlinear convergence if we identify the correct set of non-zero variables. 
Also note that for indices $i \in \mathcal{Z}$, this eventually only takes gradient steps so our analysis of the previous section applies (it identifies the manifold in a finite number of iterations).
As opposed to solving the block-wise proximal-Newton update in~\eqref{eq:proxNewton}, in our experiments we explored simply using the TMP update applied to the block and found that it gave nearly identical performance.
 
 TMP methods have also been generalized to settings like L1-regularization~\citep{schmidt2010} and they can essentially be used for any separable $g$ function. Another widely-used strategy is to inexactly solve~\eqref{eq:proxNewton}~\citep{schmidt2010,lee2012proximal,fountoulakis2015flexible}. This has the advantage that it can still be used in the group L1-regularization setting or other group-separable settings.

\subsection{Optimal Updates for Quadratic $f$ and Piecewise-Linear $g$}
\label{subsec:optQuad}

Two of the most well-studied optimization problems in machine learning are the SVM and LASSO problems. The LASSO problem is given by an L1-regularized quadratic objective
\[
	\argmin{x} \frac{1}{2} \| Ax - b \|^2 + \lambda \| x \|_1,
\]
while the dual of the (non-smooth) SVM problem has the form of a bound-constrained quadratic objective
\[
	\argmin{0 \leq x \leq \gamma} \frac{1}{2} x^TAx,
\]
for some matrix $A$ (positive semi-definite in the SVM case), and regularization constants $\lambda$ and $\gamma$. In both cases we typically expect the solution to be sparse, and identifying the optimal manifold has been shown to improve practical performance of BCD methods~\citep{svmlight,santis2015}. 

Both problems have a set of $g_i$ that are piecewise-linear over their domain, implying that the they can be written as a maximum over univariate linear functions on the domain of each variable. Although we can still consider TMP or inexact proximal-Newton updates for these problems, this special structure actually allows us to compute the exact minimum with respect to a block (which is efficient when considering medium-sized blocks). Indeed, for SVM problems the idea of using exact updates in BCD methods dates back to the sequential minimal optimization (SMO) method~\citep{platt1998}, which uses exact updates for blocks of size 2. In this section we consider methods that work for blocks of arbitrary size.\footnote{The methods discussed in this section can also be used to compute exact Newton-like updates in the case of a non-quadratic $f$, but where the $g_i$ are still piecewise-linear.}

While we could write the optimal update as a quadratic program, the special structure of the LASSO and SVM problems lends well to exact homotopy methods. These methods date back to~\citet{osborne2000, osborne2011} who proposed an exact homotopy method that solves the LASSO problem for all values of $\lambda$. This type of approach was later popularized under the name ``least angle regression'' (LARS)~\citep{lars2004}. Since the solution path is piecewise-linear, given the output of a homotopy algorithm we can extract the exact solution for our given value of $\lambda$.~\citet{hastie2004} derive an analogous homotopy method for SVMs, while~\citet{rosset2007piecewise} derive a generic homotopy method for the case of piecewise-linear $g_i$ functions.

The cost of each iteration of a homotopy method on a problem with $|b|$ variables is $O(|b|^2)$. It is known that the worst-case runtime of these homotopy methods can be exponential~\citep{mairal2012complexity}. However, the problems where this arises are somewhat unstable, and in practice the solution is typically obtained after a linear number of iterations. This gives a runtime in practice of $O(|b|^3)$, which does not allow enormous blocks, but does allow us to efficiently use block sizes in the hundreds or thousands. That being said, since these methods compute the exact block update, in the scenario where we previously had superlinear convergence, we now obtain \emph{finite} convergence. That is, the algorithm will stop in a finite number of iterations with the exact solution {\em provided that} it has identified the optimal manifold, uses a greedy rule with variable blocks, and the block size is larger than the dimension of the manifold. %Indeed, in this setting our strategy for determining the active-set complexity above also gives a bound on the total number of iterations to {\em exactly} solve SVM and LASSO problems. 
This finite termination is also guaranteed under similar assumptions for TMP methods, and although TMP methods may make less progress per-iteration than exact updates, they may be a cheaper alternative to homotopy methods as the cost is explicitly restricted to $O(|b|^3)$.

\section{Numerical Experiments}
\label{sec:experiments}
We performed an extensive variety of experiments to evaluate the effects of the choices listed in the previous sections. In this section we include several of these results that highlight some key trends we observed, and in each subsection below we explicitly list the insights we obtained from the experiment. We considered five datasets that evaluate the methods in a variety of scenarios:
\begin{enumerate}[label=\Alph*]
\item Least-squares with a sparse data matrix \blu{(10,000 variables)}.
\item Binary logistic regression with a sparse data matrix \blu{(10,000 variables)}.
\item 50-class logistic regression problem with a sparse data matrix \blu{(50,000 variables)}.
\item Lattice-structured quadratic objective as in Section~\ref{sec:messagepassing} \blu{(2,400 variables)}.
\item Binary label propagation problem (sparse but unstructured quadratic) \blu{(1,900 variables)}.
\end{enumerate}
%For interested readers, we give our choices for 
We give the full details of these datasets in~\ref{append:experiments} where we have also included our full set of experiment results.
\blu{Due to the sparsity present in these datasets, it is possible to implement the basic GS rule for a similar cost to random selection (by tracking the gradient element in a max-heap structure).} \blu{We used the bounds in~\ref{append:blockLipschitz} to set the various Lipschitz constants ($L_i$, $L_b$, and $H_b$).}

In our experiments we use the number of iterations as our measure of performance. This measure is far from perfect, especially when considering greedy methods, since it ignores the computational cost of each iteration. However, this measure of performance provides an implementation- and problem-independent measure of performance. We seek an implementation-independent measure of performance since the actual runtimes of different methods will vary wildly across applications. However, it is typically easy to estimate the per-iteration runtime when considering a new problem. Thus, we hope that our quantification of what can be gained from faster-converging methods gives guidance to readers about whether the faster-converging methods will lead to a substantial performance gain on their applications. We are careful to qualify all of our claims with warnings in cases where the iteration costs differ.

\subsection{Greedy Rules with Gradient Updates}
\label{sec:expGrad}

Our first experiment considers gradient updates with a step-size of $1/L_b$, and seeks to quantify the effect of using fixed blocks compared to variable blocks (Section~\ref{subsec:boundsprogress}) as well as the effect of using the new \blu{block} GSL rule (Section~\ref{sec:BGSL}). In particular, we compare selecting the block using Cyclic, Random, Lipschitz (sampling the elements of the block proportional to $L_i$), GS, and GSL rules. For each of these rules we implemented a fixed block (FB) and variable block (VB) variant.
For VB using Cyclic selection, we split a random permutation of the coordinates into equal-sized blocks and updated these blocks in order (followed by using another random permutation). \blu{We also considered Perm with FB, where we use a fixed set of blocks and we go through a sequence of random permutations of these blocks.} 
 To approximate the seemingly-intractable GSL rule with VB, we used the GSD rule (Section~\ref{subsec:fixed}) using the SIRT-style approximation~\eqref{eq:GSDsum} from Section~\ref{subsec:adaptive}.
 To construct the partition of the coordinates needed in the FB method, we sorted the coordinates according to their $L_i$ values then placed the largest $L_i$ values into the first block, the next set of largest in the second block, and so on. \blu{Since quite a few methods are being compared, we list the full set of methods used in this experiment (as well as the experiment of the next section) in Table~\ref{tab4}.}

\begin{sidewaystable}
    \footnotesize
    \centering
    \begin{tabular}{@{}ll||*{3}{l}}
    &   &\multicolumn{2}{c}{{\bf Blocking Strategy}} \\ [1em]
    &   &\multicolumn{1}{c}{Fixed (FB)}  &\multicolumn{1}{c}{Variable (VB)}   \\ 
 \hline \hline
    \multirow{18}*{\rotatebox{90}{{\bf Block Selection Rule}}}  
   &\begin{tabular}{@{}l@{}}Cyclic \\ $b = 1, 2, \dots, n/\tau, 1, 2, \dots$ \end{tabular}   
   &\begin{tabular}{@{}l@{}}Cyclic-FB: \\ \hspace{1em} sort $L_i$, group into equal sized \\ \hspace{1em} blocks, cycle through blocks in order \end{tabular}
   &\begin{tabular}{@{}l@{}}Cyclic-VB: \\ \hspace{1em} randomly permute $i$, group into equal \\ \hspace{1em} sized blocks, cycle through blocks in order, \\ \hspace{1em} repeat \end{tabular}
       \\ [2em] \cline{2-4} 
   &\begin{tabular}{@{}l@{}}Random \\  $b \in \{1, 2, \dots, n/\tau\}$ \end{tabular} 
   &\begin{tabular}{@{}l@{}}Random-FB: \\ \hspace{1em} sort $L_i$, group into equal sized blocks,  \\ \hspace{1em} select blocks uniformly at random \end{tabular}
   &\begin{tabular}{@{}l@{}}Random-VB: \\ \hspace{1em} select $\tau$ coordinates uniformly at random at \\ \hspace{1em} each iteration \end{tabular}  
	\\ \cline{2-4}
   &\begin{tabular}{@{}l@{}}Lipschitz Sampling \\ $p(b_k = b) = \frac{L_b}{\sum_{\hat{b} \in \mathcal{B}} L_{\hat{b}}}$ \end{tabular} 
   &\begin{tabular}{@{}l@{}}Lipschitz-FB: \\ \hspace{1em} sort $L_i$, group into equal sized blocks,  \\ \hspace{1em} select blocks randomly according to \\ \hspace{1em} non-uniform distribution \end{tabular}
   &\begin{tabular}{@{}l@{}}Lipschitz-VB: \\ \hspace{1em} select $\tau$ coordinates randomly according to \\ \hspace{1em} non-uniform distribution, $p_i = L_i/ \sum_{i = 1}^n L_{i}$ \end{tabular}  
	\\ \cline{2-4}
   &\begin{tabular}{@{}l@{}}Gauss-Southwell (GS) \\ $\displaystyle b_k \in \argmax{b \in \mathcal{B}} \| \nabla_b f(x^k) \|$ \end{tabular} 
   &\begin{tabular}{@{}l@{}}GS-FB: \\ \hspace{1em} sort $L_i$, group into equal sized blocks, \\ \hspace{1em} select blocks using GS rule \end{tabular}  
   &\begin{tabular}{@{}l@{}}GS-VB: \\ \hspace{1em} select the $\tau$ maximal coordinates according \\ \hspace{1em} to the single-coordinate GS rule \end{tabular}  
	\\  \cline{2-4}
   &\begin{tabular}{@{}l@{}}Gauss-Southwell-Lipschitz (GSL) \\ $\displaystyle b_k \in \argmax{b \in \mathcal{B}}\left\{\frac{\|\nabla_b f(x^k)\|^2}{L_{b}}\right\}$ \end{tabular}
   &\begin{tabular}{@{}l@{}}GSL-FB: \\ \hspace{1em} sort $L_i$, group into equal sized blocks, \\ \hspace{1em} select blocks using GSL rule \end{tabular} 
   &\begin{tabular}{@{}l@{}}GSL-VB: \\ \hspace{1em} select the $\tau$ maximal coordinates according \\ \hspace{1em} to the GSD rule ($D_{b,i} = D_i$), use SIRT \\ \hspace{1em} to compute the $D_i$ (\ref{subsec:adaptive}) \end{tabular}  
	\\ \cline{2-4}
   &\begin{tabular}{@{}l@{}}Gauss-Southwell-Diagonal (GSD)  \\ $\argmax{b \in \mathcal{B}}\left\{\sum_{i \in b} \frac{ |\nabla_i f(x^k)|^2 }{D_{b,i} }\right\}$ \end{tabular} 
   &\begin{tabular}{@{}l@{}}GSD-FB: \\ \hspace{1em} sort $L_i$, group into equal sized blocks, \\ \hspace{1em} select blocks using GSD rule ($D_{b,i} = L_i$) \end{tabular}
   &\begin{tabular}{@{}l@{}}GSD-VB: \\ \hspace{1em} select the $\tau$ maximal coordinates according \\ \hspace{1em} to the GSD rule ($D_{b,i} = L_i$)% \\ \hspace{1em} ($\equiv$ single-coordinate GSL) 
   \end{tabular}   
   	\\ \cline{2-4}
   &\begin{tabular}{@{}l@{}}Gauss-Southwell-Quadratic (GSQ)  \\ $\displaystyle b_k \in \argmax{b \in \mathcal{B}}\left\{\norm{\nabla_b f(x^k)}_{H_b^{-1}}\right\} $ \end{tabular}
   &\begin{tabular}{@{}l@{}}GSQ-FB: \\ \hspace{1em} sort $L_i$, group into equal sized blocks, \\ \hspace{1em} select blocks using GSQ rule \end{tabular} 
   &\begin{tabular}{@{}l@{}}GSQ-VB: \\ \hspace{1em} use approximation to GSQ rule (\ref{sec:aGSQ}) \\ \hspace{1em} and 10 iterations of IHT \end{tabular} \\ \cline{2-4}
    \end{tabular}
    \caption{\blu{Summary of block selection rules and blocking strategies used in Figures~\ref{fig:exp1} \&~\ref{fig:exp2}.}}
    \label{tab4}
\end{sidewaystable}

\begin{figure}
\includegraphics[width=.32\textwidth]{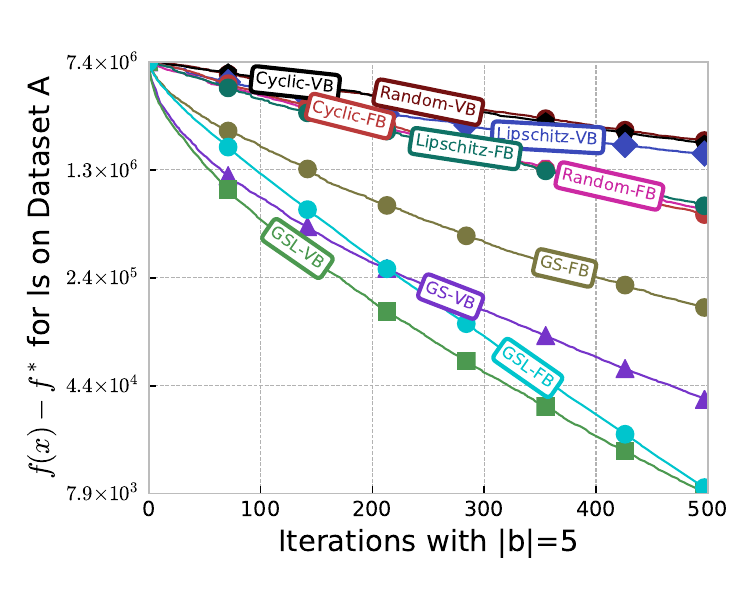}
\includegraphics[width=.32\textwidth]{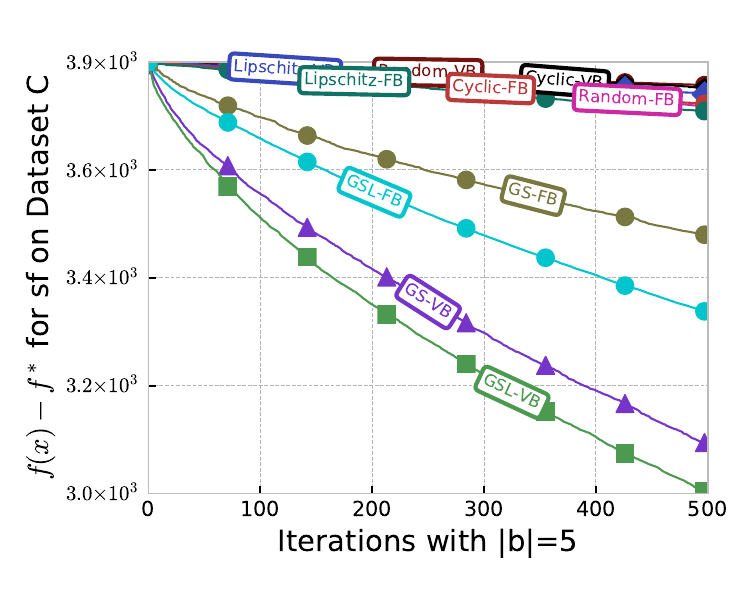}
\includegraphics[width=.32\textwidth]{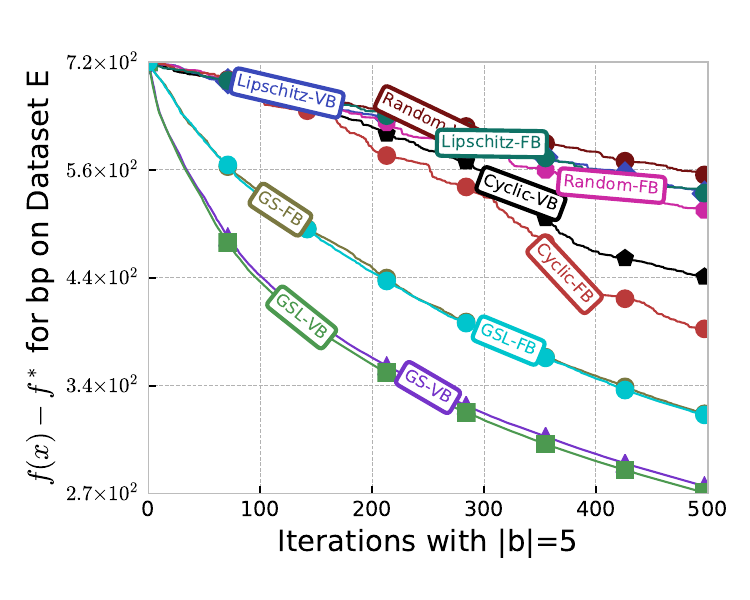}
\caption{\blu{Comparison of cyclic, random and greedy block selection rules using either FB or VB with a gradient update on three different problems. We see that the greedy selection rules tend to make more progress than the cyclic or random selection rules. Further, the greedy VB variations can converge substantially faster than greedy FB, and the GSL rule can converge substantially faster than the classic GS rule. Details on the full set of methods used in this experiment are in Table~\ref{tab4}.}}
\label{fig:exp1}
\end{figure}
% Plot {A,C,E} for smallest block size

We plot selected results \blu{for a block size of 5 in Figure~\ref{fig:exp1}}, while experiments on all datasets and with other block sizes are given in~\ref{append:greedyGradient}. Overall, we observed the following trends \blu{regarding the number of iterations}:
\begin{itemize}
\item \textbf{Greedy rules \blu{tend to make more progress per iteration} than random and cyclic rules}, particularly with small block sizes. This difference is sometimes enormous, and this suggests we should prefer greedy rules when the greedy rules can be implemented with a similar cost to cyclic or random selection.
\item \textbf{The variance of the performance between the rules becomes smaller as the block size increases}. This is because the performance of all methods increases with the block size, and suggests that if we use very-large block sizes with gradient updates that we should prefer simple Cyclic or Random updates.
\item \textbf{VB can \blu{converge substantially faster} than FB when using GS} for certain problems. \blu{This agrees with the theory, because the FBs are a subset of the VBs so the progress bound for VBs always dominates FBs}. Thus, we should prefer GS-VB for problems where this has a similar cost to GS-FB. We found \emph{this trend was reversed for random rules}, where fixed blocks tended to perform better. We suspect this trend is due to the coupon collector problem: with random selection it takes FB fewer iterations than VB to select all variables at least once.
% given that the sample size for FB is smaller.
\item \textbf{GSL \blu{can converge substantially faster} than the classic GS rule} for certain problems (\blu{while the performance was always at least as good}). In some cases the new \blu{block-}GSL rule with FB even outperformed the GS rule with VB. Interestingly, the \blu{prior work evaluating the single-coordinate GSL rule only showed minor practical gains for GSL over GS~\citep{nutini2015}, while in the block case the performance gain was in some cases much larger.}
\end{itemize}
In~\ref{append:greedyGradient} we repeat this experiment for the FB methods but using the approximation to $L_b$ discussed in Section~\ref{sec:LA} (see Figure~\ref{fig:exp3a}). This sought to test whether this procedure, which may underestimate the true $L_b$ and thus use larger step-sizes, would improve performance. This experiment lead to some additional insights:
\begin{itemize}
\item \textbf{Approximating $L_b$ was more effective as the block size increases}. This makes sense, since with large block sizes there are more possible directions and we are unlikely to ever need to use a step-size as small as $1/L_b$ for the global $L_b$ \blu{(which increases with block size)}.
\item \textbf{Approximating $L_b$ is far more effective than using a loose bound}. For small block sizes we have relatively-good bounds for all problems except Problem C. On this problem the Lipschitz approximation procedure was much more effective even for small block sizes.
\item \blu{\textbf{GSL sometimes performed worse than the classic GS rule when approximating $L_b$}. This seems sensible, particularly early in the optimization, since the GS rule does not depend on the $L_b$ while the GSL rule can potentially be negatively affected by poor estimates of $L_b$. This suggests that, for example, we might want to consider methods that start out using the GS rule and gradually switch to using the GSL rule.}
\end{itemize}
This experiment suggests that we should prefer to use an approximation to $L_b$ (or an explicit line-search) when using gradient updates unless  we have a tight approximation to the true $L_b$ \emph{and} we are using a small block size.
\blu{We also performed experiments with different block partitioning strategies when using greedy rules with FB (see Figure~\ref{fig:exp4b}). In those experiments, our sorting approach tended to outperform using random blocks or choosing the blocks to have similar average Lipschitz constants.\footnote{We repeated this experiment comparing different FB partitioning rules with cyclic and random rules, but for these rules we found that the block partitioning strategy did not make a large difference.}}

%Mention block partitioning strategies here, too?

\subsection{Greedy Rules with Matrix Updates}
\label{sec:expNewt}

Our next experiment considers using matrix updates based on the matrices $H_b$ from~\ref{append:blockLipschitz}, and quantifies the effects of the GSQ and GSD rules introduced in Sections~\ref{sec:BGSQ}-\ref{subsec:fixed} as well the approximations to these introduced in Sections~\ref{subsec:adaptive}-\ref{sec:aGSQ}. In particular, for FB we consider the GS rule and the GSL rule (from the previous experiment), the GSD rule (using the diagonal matrices from Section~\ref{subsec:adaptive} with $D_{b,i} = L_i$), and the GSQ rule (which is optimal for the three quadratic objectives). For VB we consider the GS rule from the previous experiment as well as the GSD rule (using $D_{b,i} = L_i$), and the GSQ rule using the approximation from Section~\ref{sec:aGSQ} and 10 iterations of iterative hard thresholding. Other than switching to matrix updates and focusing on these greedy rules, we keep all other experimental factors the same.

\begin{figure}
\includegraphics[width=.32\textwidth]{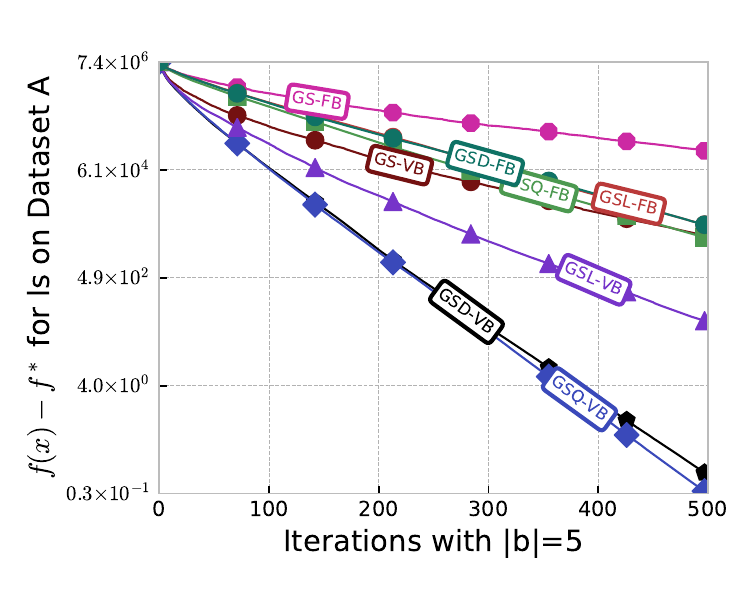}
\includegraphics[width=.32\textwidth]{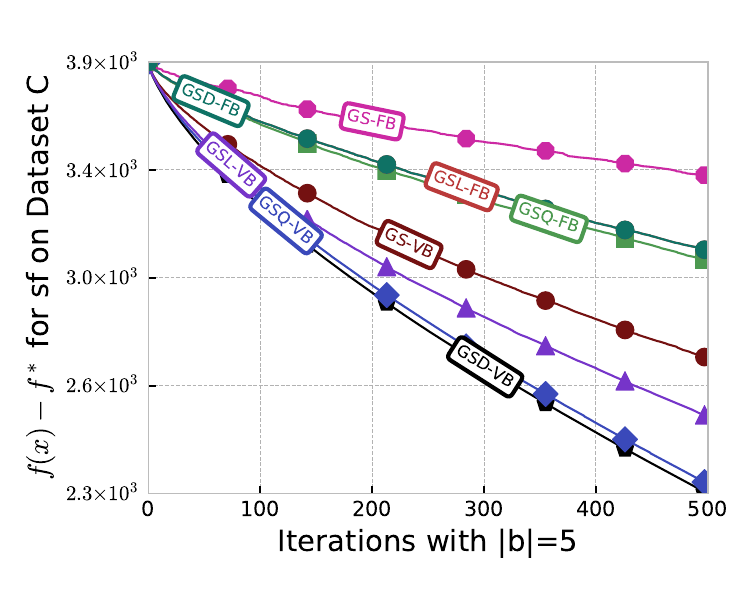}
\includegraphics[width=.32\textwidth]{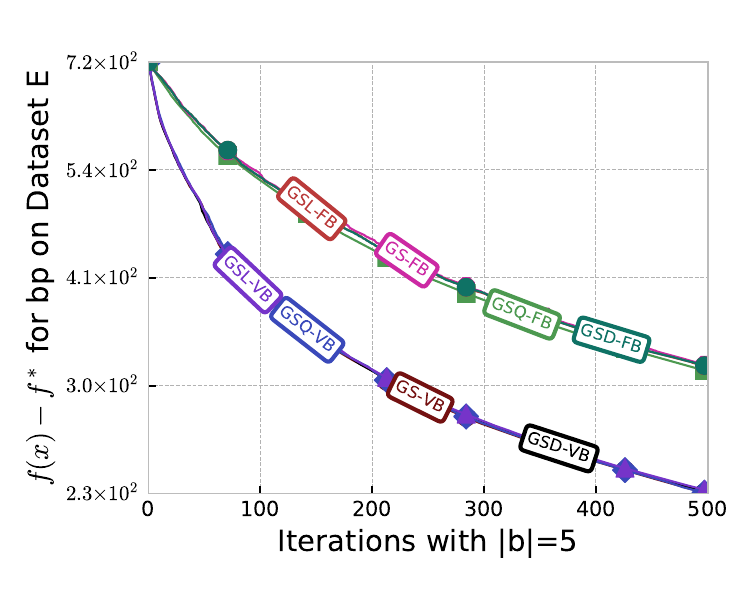}
\caption{\blu{Comparison of different greedy block selection rules using either FB or VB with matrix updates on three different problems. For all methods considered, the VB variations consistently performed similar or better than the FB variations in terms of convergence rate. We see that even using the crude approximation for GSD with VB ($D_{b,i} = L_i$), it often outperforms GS-VB and GSL-VB. Details on the full set of methods used in this experiment are in Table~\ref{tab4}.}}
\label{fig:exp2}
\end{figure}

We plot selected results of doing this experiment in Figure~\ref{fig:exp2}. These experiments showed the following interesting trends:
\begin{itemize}
\item \textbf{There is a larger advantage to VB \blu{in terms of convergence rate} with matrix updates}. When using matrix updates, the basic GS-VB method \blu{converged faster} than even the most effective GSQ-FB rule for smaller block sizes.
\item \textbf{With FB, there is little advantage to GSD/GSQ}. Although the GSL rule consistently improved over the classic GS rule, we did not see any advantage to using the more-advanced GSD or GSQ rules when using FB.
\item \textbf{GSD \blu{can converge substantially faster} than the classic GS rule with VB} for certain problems (\blu{while the performance was always at least as good}). Despite the use of a crude \blu{diagonal approximation}, the GSD rule often outperformed the classic GS rule. The GSD-VB method (which uses $D_{\blu{b},i} = L_i$) \blu{also consistently performed similar or better than} the GSL-VB method (which uses a different diagonal scaling with a tighter bound), despite the GSL-VB having better performance when using gradient updates.
\item \textbf{GSQ \blu{converged slightly faster than} GSD with VB and large blocks}. Although the GSQ-VB rule performed the best across all experiments in terms of number of iterations, the difference was more noticeable for large block sizes. However, this did not offset its high cost in any experiment. We also experimented with OMP instead of IHT, and found it gave a small improvement, but the iterations were substantially more expensive.
\end{itemize}
Putting the above together, with matrix updates our experiments indicate that \blu{among greedy rules} the GSD rule seems to  provide good performance in all settings. We would only recommend using the GSQ rule in settings where we can use VB and where operations involving the objective $f$ are much more expensive than running an IHT or OMP method. We performed experiments with different block partition strategies for FB, but found that when using matrix updates the partitioning strategy did not make a big difference for cyclic, random, or greedy rules.

In~\ref{append:greedyNewtonMatrix} we repeat this experiment for the non-quadratic objectives using the Newton direction and a backtracking line-search to set the step-size (see Figure~\ref{fig:exp2LS}), as discussed in Sections~\ref{sec:line-search} and~\ref{sec:newton}. For both datasets, the Newton updates resulted in \blu{faster convergence} than the matrix updates. This indicates that we should prefer classic Newton updates over the more recent matrix updates for non-quadratic objectives where computing the sub-block of the Hessian is tractable.

\subsection{Message-Passing Updates}
\label{exp:message}

We next seek to quantify the effect of using message-passing to efficiently implement exact updates for quadratic functions, as discussed in Section~\ref{sec:messagepassing}. For this experiment, we focused on the lattice-structured dataset D and the unstructured but sparse dataset E. These are both quadratic objectives with high treewidth, but that allow us to find large forest-structured induced subgraphs. We compared the following strategies to choose the block: greedily choosing the best general unstructured blocks using GS (General), cycling between blocks generated by the greedy graph colouring algorithm of Section~\blu{\ref{sec:redblack}} (Red Black), cycling between blocks generated by the greedy forest-partitioning algorithm of Section~\ref{sec:treePart} (Tree Partitions), greedily choosing a tree using the algorithm of Section~\ref{sec:treeApprox} (Greedy Tree), and growing a tree randomly using the same algorithm (Random Tree). \blu{For the lattice-structured Dataset D, the partitioning algorithms (for Red Black and Tree Partitions) proceed through the variables in order which generates partitions similar to those shown in Figure~\ref{fig:FT}. For the unstructured Dataset E, we apply the partitioning strategies of Section~\ref{sec:treePart} using both a random ordering (Random) and by sorting the Lipschitz constants $L_i$ (Lipschitz).}
Since the cost of the exact update for tree-structured methods is $O(n)$, for the unstructured blocks we chose a block size of $b_k = n^{1/3}$ to make the costs comparable (since the exact solve is cubic in the block size for unstructured blocks). \blu{In Table~\ref{tab:figure6}, we list the full set of methods used in Figure~\ref{fig:exp6}.}

\begin{sidewaystable}
    \footnotesize
    \centering
    \begin{tabular}{@{}ll||l}
    &\multicolumn{1}{c||}{{}}   & \multicolumn{1}{c}{{\bf Blocking Strategy}}\\[2ex]
 \hline \hline
    \multirow{18}*{\rotatebox{90}{{\bf Block Selection Rule}}}  
&\begin{tabular}{@{}l@{}}General \end{tabular}
&\begin{tabular}{@{}l@{}}Select the best general unstructured block \\ using the GS-VB rule from~\ref{tab4}. \end{tabular}
 \\ [2ex]\cline{2-3} 
&\begin{tabular}{@{}l@{}}Random Tree \end{tabular} 
&\begin{tabular}{@{}l@{}}Grow a tree randomly using the algorithm of \\Section~\ref{sec:treeApprox}. \end{tabular} 
	\\ [2ex]\cline{2-3}
&\begin{tabular}{@{}l@{}}Greedy Tree \end{tabular} 
&\begin{tabular}{@{}l@{}}Greedily choose a tree using the algorithm of \\Section~\ref{sec:treeApprox}. \end{tabular}
	\\ [2ex]\cline{2-3}
&\begin{tabular}{@{}l@{}}Red Black \end{tabular}  
&\begin{tabular}{@{}l@{}}Cycle between blocks generated by the \\greedy graph colouring algorithm of Section~\ref{sec:redblack}.\\ This requires choosing an ordering of the variables. \\
For the lattice-structured data we used a colum-major ordering. \\For the unstructured data we considered two approaches:\\ 
\textbf{Random}:  Use a random order. \\
\textbf{Lipschitz}: Proceed through the vertices by descending Lipschitz values.
\end{tabular}
	\\[4ex]\cline{2-3}
&\begin{tabular}{@{}l@{}}Tree Partitions \end{tabular}
&\begin{tabular}{@{}l@{}}Cycle between blocks generated by the \\greedy forest-partitioning algorithm of Section~\ref{sec:treePart}.\\ This requires choosing an ordering of the variables. \\
For the lattice-structured data we used a colum-major ordering. \\For the unstructured data we considered two approaches:\\ 
\textbf{Random}:  Use a random order. \\
\textbf{Lipschitz}: Proceed through the vertices by descending Lipschitz values. \end{tabular}
 \hspace{1em}  \\\cline{2-3}
    \end{tabular}
    \caption{\blu{Summary of blocking strategies used in Figure 6. For the algorithms used for Red Black and Tree Partitions blocking strategies, we consider three variations of ordering strategies in Step 1.}}
    \label{tab:figure6}
\end{sidewaystable}

\begin{figure}
\centering
\includegraphics[width=.45\textwidth]{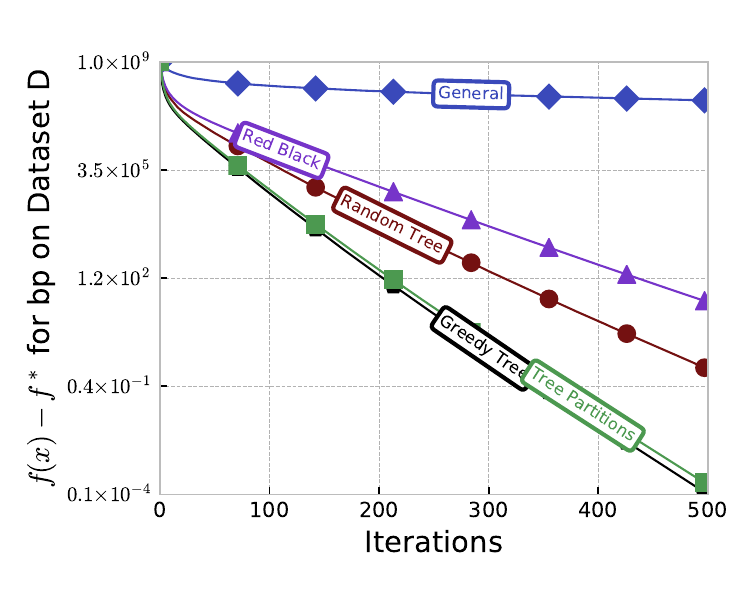}
\includegraphics[width=.45\textwidth]{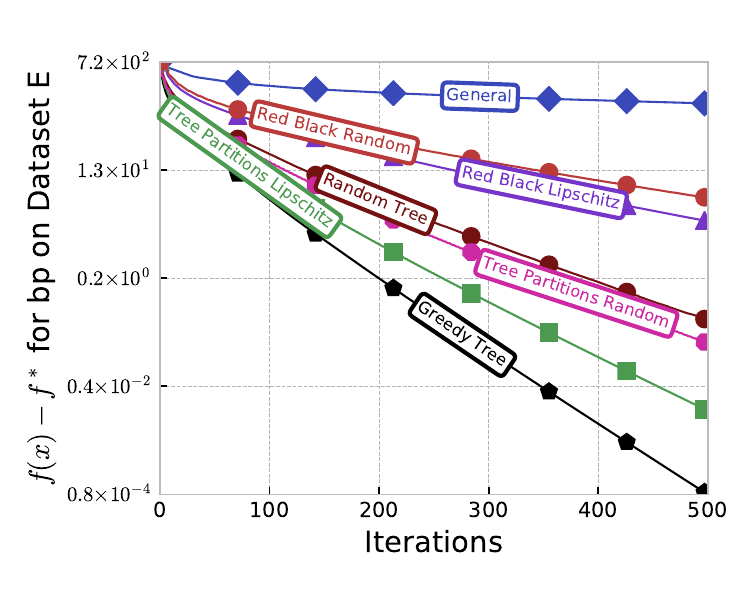}
\caption{\blu{Comparison of different blocking strategies on two quadratic graph-structured problems when using optimal updates. While the Random Tree blocking strategy performed better than the Red Black method, overall the Greedy Tree and Tree Partitions strategies performed best on these structured problems. We note that cycling through the vertices according to descending Lipschitz values for Red Black and Tree Partitions gives us better results than cycling through the vertices in random order. Details on the full set of methods used in this experiment are in Table~\ref{tab:figure6}.}}
\label{fig:exp6}
\end{figure}

We plot selected results of doing this experiment in Figure~\ref{fig:exp6}. Here, we see that even the classic red-black ordering outperforms using general unstructured blocks (since  general blocks must use such small block sizes). The tree-structured blocks perform even better, and in the unstructured setting our greedy approximation of the GS rule under variable blocks \blu{converges faster than} the other strategies. However, our greedy tree partitioning method also performs well. For the lattice-structured data it \blu{converged at a similar rate} to the greedy approximation, while for the unstructured data it \blu{converged faster than} all methods except greedy (and performed better when sorting by the Lipschitz constants than using a random order).

\subsection{Proximal Updates}
\label{sec:experimentsProx}

Our final experiment demonstrates the manifold identification and superlinear/finite convergence properties of the greedy BCD method as discussed in Section~\ref{sec:exactupdates} for a sparse non-negative constrained L1-regularized least-squares problem using Dataset A. In particular, we compare the performance of a projected gradient update with $L_b$ step-size, a projected Newton (PN) solver with line-search as discussed in Section~\ref{sec:superlinear} and the two-metric projection (TMP) update as discussed in Section~\ref{sec:tmp} when using fixed (FB) and variable (VB) blocks of different sizes ($|b|\in\{5,50,100\}$). We use a regularization constant of $\lambda = 50,000$ to encourage a high level of sparsity resulting in an optimal solution $x^*$ with 51 non-zero variables. In Figure~\ref{fig:exp7a} we indicate active-set identification with a star and show that all approaches eventually identify the active-set. We see that TMP \blu{converges at nearly the same speed as} projected Newton for all block sizes \blu{(despite its lower iteration cost)} and both \blu{converge faster} than performing gradient updates. For a block size of 100, we get finite convergence using projected Newton and TMP updates. \blu{In Table~\ref{tab:figure7}, we list details for each method used in Figure~\ref{fig:exp7a}.}

\begin{sidewaystable}
    \centering
    \footnotesize
    \begin{tabular}{@{}ll|l||*{5}{l|l}}
     &\multicolumn{2}{c}{} &\multicolumn{2}{c}{{\bf Blocking Strategy}} \\ [1em]
     & \multicolumn{1}{l}{} & \multicolumn{1}{l||}{Update Formula} &\multicolumn{1}{c}{Fixed (FB)}  &\multicolumn{1}{c}{Variable (VB)}   \\ 
 \hline \hline
    \multirow{8}*{\rotatebox{90}{{\bf Block Update Rule}}}  
   &\begin{tabular}{@{}l@{}}Projected Gradient (PG) \end{tabular}
   &\begin{tabular}{@{}l@{}}$x_b^{k+1} = \left[x_b - \frac{1}{L_b}\nabla_{b} f(x^k)\right]^+$ \end{tabular} 
   &\begin{tabular}{@{}l@{}}PG-FB: \\ \hspace{1em} sort $L_i$, \\ \hspace{1em}group into equal sized blocks, \\ \hspace{1em}select blocks using GSL-q rule \end{tabular}
   &\begin{tabular}{@{}l@{}}PG-VB: \\ \hspace{1em} select the $\tau$ maximal coordinates  \\ \hspace{1em}according to the GSD-q rule \\ \hspace{1em} with $D_{b,i} = L_i$ \end{tabular}  
	\\ \cline{2-5}
   &\begin{tabular}{@{}l@{}}Projected Newton (PN) \end{tabular}
   &\begin{tabular}{@{}l@{}}Solve quadratic program, see \eqref{eq:proxNewton} \\ (using line-search from Section~\ref{sec:superlinear}) \end{tabular}   
   &\begin{tabular}{@{}l@{}}PN-FB: \\ \hspace{1em} sort $L_i$, \\ \hspace{1em}group into equal sized blocks, \\ \hspace{1em}select blocks using GSL-q rule \end{tabular}
   &\begin{tabular}{@{}l@{}}PN-VB: \\ \hspace{1em} select the $\tau$ maximal coordinates \\ \hspace{1em} according  to the GSD-q rule \\ \hspace{1em} with $D_{b,i} = L_i$ \end{tabular}  
       \\ [2em] \cline{2-5} 
   &\begin{tabular}{@{}l@{}}Two-Metric Projection (TMP)  \end{tabular}
   &\begin{tabular}{@{}l@{}}Solve linear system and project, \\ see Section~\ref{sec:tmp} \end{tabular} 
   &\begin{tabular}{@{}l@{}}TMP-FB: \\ \hspace{1em} sort $L_i$, \\ \hspace{1em}group into equal sized blocks, \\ \hspace{1em}select blocks using GSL-q rule \end{tabular}
   &\begin{tabular}{@{}l@{}}TMP-VB: \\ \hspace{1em} select the $\tau$ maximal coordinates \\ \hspace{1em} according  to the GSD-q rule \\ \hspace{1em} with $D_{b,i} = L_i$ \end{tabular}  
    \end{tabular}
    \caption{\blu{Summary of block update rules used in Figure 7.}}
    \label{tab:figure7}
\end{sidewaystable}

\begin{figure}[!h]
\centering
\includegraphics[width=\textwidth]{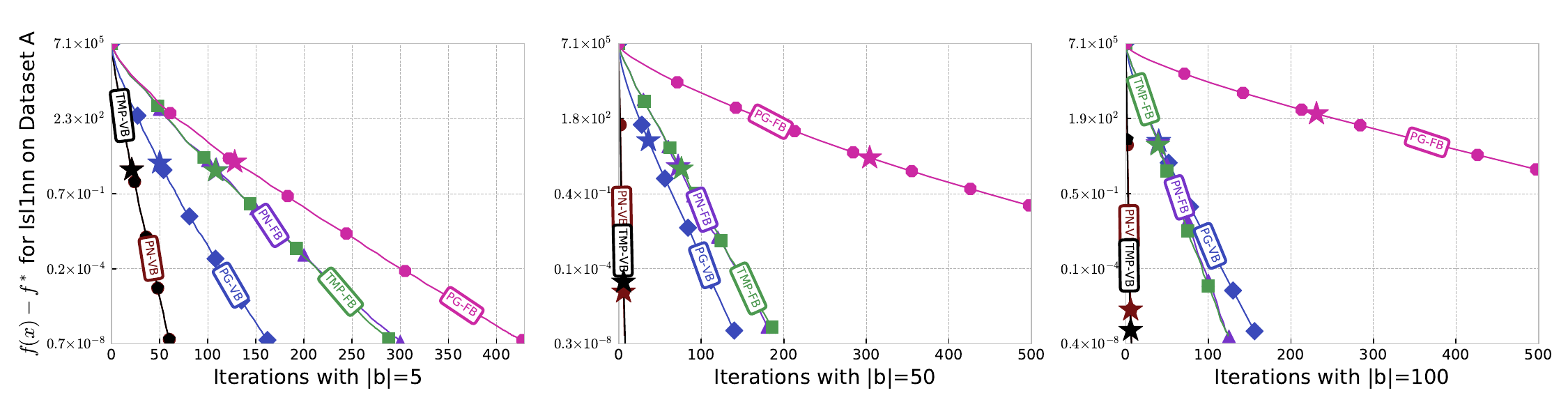}
\caption{\blu{Comparison of different proximal updates when using greedy fixed and variable blocks of different sizes. The star on each curve indicates active-set identification. We see that TMP and projected Newton converge faster than using gradient updates, while TMP achieves a similar convergence speed as projected Newton despite having a lower iteration cost. Further, the VB variations consistently outperformed all the FB variations, except for the largest block size of 100, where TMP-FB and PN-FB beat PG-VB. Details for each method used in this experiment are in Table~\ref{tab:figure7}. }}
\label{fig:exp7a}
\end{figure}

In~\ref{append:activeSetExp} we repeat this experiment for random block selection and show that for such a sparse problem multiple iterations are often required before progress is made due to the repetitive selection of variables that are already zero/active.

\section{Discussion}
\label{sec:discussion}

%An important setting for variational inference in graphical models that is a very important topic in machine learning is the use of BCD methods when there are constraints between blocks~\citep{bishop2006}. We currently have {\em no} convergence rates for variational inference methods, so this would be a significant contribution. 

\blu{An important issue that we have not discussed is how to choose the block size $\tau$. BCD tends to converge faster with larger blocks, but this must be balanced against the higher iteration cost of using larger blocks. In some cases the blocks are dictated by a certain problem structure such as group L1-regularization where the objective is block separable or  multi-class logistic regression where updating a suitable block has the same cost as updating a single coordinate. As a practical recommendation, we suggest choosing the largest block size where we can still view the update as ``inexpensive''. This is because it is typically easy to estimate iteration costs in practice while predicting the convergence rate in practice is more difficult. As an example, if the cost of the update is $O(\tau^3)$ then blocks of size 5 might be preferred over blocks of size 1 while blocks of size 10000 could lead to an unnacceptably slow iteration cost.}

% Accelerated and parallel
In this work we focused on non-accelerated BCD methods. However, we expect that our conclusions are likely to also apply for accelerated BCD methods~\citep{fercoq2013accelerated}. \blu{We thus expect this work to be complimentary to recent work on accelerated greedy coordinate descent methods~\citep{locatello2018matching,lu2018accelerating}.}
Similarly, while we focused on the setting of serial computation, we expect that our conclusions will give insight into developing more efficient parallel and distributed BCD methods~\citep{richtarik2016parallel}.

% Choosing D 
%Our choice of the diagonal matrices $D$ within the GSD rule is clearly sub-optimal, and our experiments showed that this offered little advantage over the simpler GSL rule. 
Although our experiments indicate that our choice of the diagonal matrices $D$ within the GSD rule provides a consistent improvement, this choice is clearly sub-optimal.
A future direction is to find a generic strategy to construct better diagonal matrices, and work on ESO methods could potentially be adapted for doing this~\citep{qu2016coordinate}. This could be in the setting where we are given knowledge of the Lipschitz constants, but a more-interesting idea is to construct these matrices online as the algorithm runs. 

% Second-order rules
The GSQ rule can be viewed as a greedy rule that incorporates more sophisticated second-order information than the simpler GS and GSL rules. In preliminary experiments, we also considered selection rules based on the cubic regularization bound. However, these did not seem to converge more quickly than the existing rules in our experiments, and it is not obvious how one could efficiently implement such second-order rules.

% Non-quadratic and non-cubic rules
We focused on BCD methods that approximate the objective function at each step by globally bounding higher-order terms in a Taylor expansion. However, we would expect more progress if we could bound these locally in a suitably-larger neighbourhood of the current iteration. Alternately, note that bounding the Taylor expansion is not the only way to upper bound a function. For example,~\citet{khan2012variational} discusses a variety of strategies for bounding the binary logistic regression loss and indeed proves that other bounds are tighter than the Taylor expansion (``B{\"o}hning'') bound that we use. It would be interesting to explore the convergence properties of BCD methods whose bounds do not come from a Taylor expansion.

% Generalized BP (non-quadratic and generalized BP)
While we focused on the case of trees, there are message-passing algorithms that allow graphs with cycles~\citep{rose1970triangulated,srinivasan2015}. The efficiency of these methods depends on the ``treewidth'' of the induced subgraph, where if the treewidth is small (as in trees) then the updates are efficient, and if the treewidth is large (as in fully-connected graphs) then these do not provide an advantage. Treewidth is related to the notion of ``chordal'' graphs (trees are special cases of chordal graphs) and chordal embeddings which have been exploited for  matrix problems like covariance estimation~\citep{dahl2008covariance} and semidefinite programming~\citep{sun2014decomposition,vandenberghe2015chordal}. 
Considering ``treewidth 2'' or ``treewidth 3'' blocks would give more progress than our tree-based updates, although it is NP-hard to compute the treewidth of a general graph (but it is easy to upper-bound this quantity by simply choosing a random elimination order).

As opposed to structural constraints like requiring the graph to be a tree, it is now known that message-passing algorithms can solve linear systems with other properties like diagonal dominance or ``attractive'' coefficients~\citep{malioutov2006walk}. There also exist specialized linear-time solvers for Laplacian matrices~\citep{kyng2016approximate}, and these can be used within BCD methods for certain structures.
% it would be interesting to explore BCD methods based on these structures. 
%While message-passing has traditionally been restricted to quadratic objectives, several authors have recently explored message-passing in more general settings including those with linear constraints. These could presumably be applied to use BCD in certain constrained or non-smooth settings.  
It would also be interesting to explore whether approximate message-passing algorithms which allow general graphs~\citep{malioutov2006walk} can be used to improve optimization algorithms.

\acks{We would like to thank Dr.\ Warren Hare for fruitful conversations that helped with Section~\ref{sec:exactupdates}. We would also like to thank Francis Bach, Coralia Cartis, Aleksandar Dogandzic, Chen Greif, Ives Macedo, Anastasia Podosinnikova, Suvrit Sra, and the anonymous reviewers for valuable discussions and feedback. This work was supported by a Discovery Grant from the Natural Sciences and Engineering Research Council of Canada (NSERC). Julie Nutini and Issam Laradji are funded by UBC Four-Year Doctoral Fellowships (4YFs).}

% Manual newpage inserted to improve layout of sample file - not
% needed in general before appendices/bibliography.

\appendix
\gdef\thesection{Appendix \Alph{section}}
%\documentclass{standalone}
%\begin{document}

\section{Cost of  Multi-Class Logistic Regression}
\phantomsection
\label{app:sparsesoftmax}

The typical setting where we expect coordinate descent to outperform gradient descent is when the cost of one gradient descent iteration is similar to the cost of updating all variables via coordinate descent. It is well known that for the binary logistic regression objective, one of the most ubiquitous models in machine learning, coordinate descent with uniform random selection satisfies this property. We previously showed that this property is also satisfied for the GS rule in the case of logistic regression, provided that the data is sufficiently sparse~\citep{nutini2015}. 

In this section we consider \emph{multi-class} logistic regression. We first analyze the cost of gradient descent on this objective and how randomized coordinate descent is efficient for any sparsity level. Then we show that a high sparsity level is not sufficient for the GS rule to be efficient for this problem, but that it is efficient if we use a particular set of fixed blocks.

\subsection{Cost of Gradient Descent}

The likelihood for a single training example $i$ with features $a_i \in \R^d$ and a label $b_i \in \{1,2,\dots,k\}$ is given by
\[
p(b_i \; | \; a_i, X) = \frac{\exp(x_{b_i}^Ta_i)}{\sum_{c=1}^k\exp(x_c^Ta_i)},
\]
where $x_c$ is column $c$ of our matrix of parameters $X \in \R^{d \times k}$ (so the number of parameters $n$ is $dk$). To maximize the likelihood over $m$ independent and identically-distributed training examples we minimize the negative log-likelihood,
\begin{equation}
\phantomsection
\label{eq:MLR}
f(X) = \sum_{i=1}^m\left[-x_{b_i}^Ta_i + \log\left(\sum_{c=1}^k\exp(x_c^Ta_i)\right)\right],
\end{equation}
which is a convex function.
The partial derivative of this objective with respect to a particular $X_{jc}$ is given by
\begin{align}
\phantomsection
\label{eq:partialMLR}
\frac{\partial }{\partial X_{jc}} f(X) % & = -\sum_{i=1}^n x_{ij}\left[I(y_i = c) - p(y_i = c | x_i,W)\right]\\
& =  -\sum_{i=1}^m a_{ij}\left[I(b_i = c) - \frac{\exp(x_c^Ta_i)}{\sum_{c'=1}^k\exp(x_{c'}^Ta_i)}\right],
\end{align}
where $I$ is a $0/1$ indicator variable and $a_{ij}$ is feature $j$ for training example $i$. We use $A$ to denote a matrix where row $i$ is given by $a_i^T$.
To compute the full gradient, the operations which depend on the size of the problem are:
\begin{enumerate}
\item Computing $x_c^Ta_i$ for all values of $i$ and $c$.
\item Computing the sums $\sum_{c=1}^k\exp(x_c^Ta_i)$ for all values of $i$. 
\item Computing the partial derivative sums~\eqref{eq:partialMLR} for all values of $j$ and $c$.
\end{enumerate}
The first step is the result of the matrix multiplication $AX$, so if $A$ has $z$ non-zeroes then this has a cost of $O(zk)$ if we compute it using $k$ matrix-vector multiplications. The second step costs $O(mk)$, which under the reasonable assumption that $m \leq z$ (since each row usually has at least one non-zero) is also in $O(zk)$. The third step is the result of a matrix multiplication of the form $A^TR$ for a (dense) $m$ times $k$ matrix $R$ (whose elements have a constant-time cost to compute given the results of the first two steps), which also costs $O(zk)$ giving a final cost of $O(zk)$.

\subsection{Cost of Randomized Coordinate Descent}

Since there are $n=dk$ variables, we want our coordinate descent iterations to be $dk$-times faster than the gradient descent cost of $O(zk)$. Thus, we want to be able to implement coordinate descent iterations for a cost of $O(z/d)$ (noting that we always expect $z \geq d$ since otherwise we could remove some columns of $A$ that only have zeroes). The key to doing this for randomized coordinate descent is to track two quantities:
\begin{enumerate}
\item The values $x_c^Ta_i$ for all $i$ and $c$.
\item The values $\sum_{c'=1}^k\exp(x_{c'}^Ta_i)$ for all $i$.
\end{enumerate}
Given these values we can compute the partial derivative in $O(z/d)$ in expectation, because this is the expected number of non-zero values of $a_{ij}$ in the partial derivative sum~\eqref{eq:partialMLR} ($A$ has $z$ total non-zeroes and we are randomly choosing one of the $d$ columns).  Further, after updating a particular $X_{jc}$ we can update the above quantities for the same cost:
\begin{enumerate}
\item We need to update $x_c^Ta_i$ for the particular $c$ we chose for the examples $i$ where $a_{ij}$ is non-zero for the chosen value of $j$. This requires an $O(1)$ operation (subtract the old $x_{jc}a_{ij}$ and add the new value) for each non-zero element of column $j$ of $A$. Since $A$ has $z$ non-zeroes and $d$ columns, the expected number of non-zeroes is $z/d$ so this has a cost of $O(z/d)$.
\item We need to update $\sum_{c'=1}^k\exp(x_{c'}^Ta_i)$ for all $i$ where $a_{ij}$ is non-zero for our chosen $j$. Since we expect $z/d$ non-zero values of $a_{ij}$, the cost of this step is also $O(z/d)$.
\end{enumerate}
Note that BCD is also efficient since if we update $\tau$ elements, the cost is $O(z\tau/d)$ by just applying the above logic $\tau$ times.  In fact, the second step and computing the final partial derivative has some redundant computation if we update multiple $X_{jc}$ with the same $c$, so we might have a small performance gain in the block case.

\subsection{Cost of Greedy Coordinate Descent (Arbitrary Blocks)}

The cost of greedy coordinate descent is typically higher than randomized coordinate descent since we need to track \emph{all} partial derivatives. However, consider the  case where each row has at most $z_r$ non-zeroes and each column has at most $z_c$ non-zeroes. In this setting we previously showed that for binary logistic regression it is possible to track all partial derivatives for a cost of $O(z_rz_c)$, and that we can track the maximum gradient value at the cost of an additional logarithmic factor~\citep[Appendix~A]{nutini2015}.\footnote{Note that the purpose of the quantity $z_rz_c$ is to serve as a potentially-crude upper bound on the maximum degree in the dependency graph we describe in Section~\ref{sec:messagepassing}. Any tighter bound on this degree would yield a tighter upper bound on the runtime.} Thus, greedy coordinate selection has a similar cost to uniform selection when the sparsity pattern makes $z_rz_c$ similar to $z/d$ (as in the case of a grid-structured dependency graph like Figure~\ref{fig:lattice}).

Unfortunately, having $z_rz_c$ similar to $z/d$ is not sufficient in the multi-class case. In particular, the cost of tracking all the partial derivatives after updating an $X_{jc}$  in the multi-class case can be broken down as follows:
\begin{enumerate}
\item We need to update $x_c^Ta_i$ for the examples $i$ where $a_{ij}$ is non-zero. Since there are at most $z_c$ non-zero values of $a_{ij}$ over all $i$ the cost of this is $O(z_c)$.
\item We need to update $\sum_{c=1}^k\exp(x_{c}^Ta_i)$ for all $i$ where $a_{ij}$ is non-zero. Since there are at most $z_c$ non-zero values of $a_{ij}$ the cost of this is $O(z_c)$.
\item We need to update the partial derivatives $\partial f/\partial X_{jc}$ for all $j$ and $c$. Observe that each time we have $a_{ij}$ non-zero, we change the partial derivative with respect to all features $j'$ that are non-zero in the example $i$ and we must update all classes $c'$ for these examples. Thus, for the $O(z_c)$ examples with a non-zero feature $j$ we need to update up to $O(z_r)$ other features for that example and for each of these we need to update all $k$ classes. This gives a cost of $O(z_r z_c k)$.
\end{enumerate}
So while in the binary case we needed $O(z_rz_c)$ to be comparable to $O(z/d)$ for greedy coordinate descent to be efficient, in the multi-class case we now need $O(z_rz_ck)$ to be comparable to $O(z/d)$ in the multi-class case. This means that not only do we need a high degree of sparsity but we also need the number of classes $k$ to be small for greedy coordinate descent to be efficient for arbitrary blocks.

\subsection{Cost of Greedy Coordinate Descent (Fixed Blocks)}

Greedy rules are more expensive in the multi-class case because whenever we change an individual variable $X_{jc}$, it changes the partial derivative with respect to $X_{j'c'}$ for a set of $j'$ values and for \emph{all} $c'$. However, we can improve the efficiency of greedy rules by using a special choice of fixed blocks that reduces the number of $j'$ values. In particular, BCD is more efficient for the multi-class case if we put  $X_{jc'}$ for all $c'$ into the same block. In other words, we ensure that each row of $X$ is part of the same block so that we apply BCD to rows rather than in an unstructured way. Below we consider the cost of updating the needed quantities after changing an entire row of $X_{jc}$ values:
\begin{enumerate}
\item Since we are updating $k$ elements, the cost of updating the $x_c^Ta_i$ is $k$-times larger giving $O(z_c k)$ when we update a row. 
\item Similarly, the cost of updating the sums $\sum_{c=1}^k\exp(x_{c}^Ta_i)$ is $k$-times larger also giving $O(z_c k)$.
\item Where we gain in computation is the cost of computing the changed values of the partial derivatives $\partial f/\partial X_{jc}$. As before, each time we have $a_{ij}$ non-zero for our particular row $j$, we change the partial derivative with respect to all other $j'$ for this example and with respect to each class $c'$ for these $j'$. Thus, for the $O(z_c)$ examples with a non-zero feature $j$ we need to update up to $O(z_r)$ other features for that example and for each of these we need to update all $k$ classes. However, since $j$ is the same for each variable we update, we only have to do this once which gives us a cost of $O(z_r z_c k)$.
\end{enumerate}
So the cost to update a row of the matrix $X$ is $O(z_rz_c k)$, which is the same cost as only updating a single element.
Considering the case of updating individual rows, this gives us $d$ blocks so in order for BCD to be efficient it must be $d$-times faster than the gradient descent cost of $O(zk)$. Thus, we need a cost of $O(zk/d)$ per iteration. This is achieved if  $O(z_r z_c)$ to be similar to $O(z/d)$, which is the same condition we needed  in the binary case.

\section{Blockwise Lipschitz Constants}
\phantomsection
\label{append:blockLipschitz}

In this section we show how to derive lower-bounds on the block-Lipschitz constants of the gradient and Hessian for several common problem settings.
% Although it is possible to show the results below without using this equivalence, 
We will use that a twice-differentiable function has an $L$-Lipschitz continuous gradient if and only if the absolute eigenvalues of its Hessian are upper-bounded by $L$,
\[
	\| \nabla f(x) - \nabla f(y) \| \le L \| x - y \| \iff \nabla^2 f(x) \preceq LI.
\]
This implies that when considering blockwise constants we have
\[
	\norm{\nabla_b f(x+U_bd) - \nabla_b f(x)}  \iff \nabla_{bb}^2 f(x) \preceq L_b I,
\]
\blu{for some $L_b \leq L$}.
Thus, bounding the blockwise eigenvalues of the Hessian bounds the blockwise Lipschitz constants of the gradient. We also use that this equivalence extends to the case of general quadratic norms,
\[
\norm{\nabla_b f(x+U_bd) - \nabla_b f(x)}_{H_b^{-1}} \leq \norm{d}_{H_b} \iff \nabla_{bb}^2 f(x) \preceq H_b.
\]

\subsection{Quadratic Functions}

Convex quadratic functions have the form
\[
f(x) = \half x^TAx + c^Tx,
\]
for a positive semi-definite matrix $A$ and vector $c$. For all $x$ the  Hessian with respect to block $b$ is given by the sub-matrix of $A$,
\[
\nabla_{bb}^2 f(x) = A_{bb}.
\]
Thus, we have that $L_b$ is given by the maximum eigenvalue of the submatrix, $L_b = \norm{A_b}$ (the operator norm of the submatrix). In the special case where $b$ only contains a single element $i$, we have that $L_i$ is given by the absolute value of the diagonal element, $L_i = |A_{ii}|$. If we want to use a general quadratic norm we can simply take $H_b = A_{bb}$, which is cheaper to compute than the $L_b$ (since it does not require an eigenvalue calculation).

\subsection{Least Squares}

The least squares objective has the form
\[
f(x) = \half \norm{Ax - c}^2,
\]
for a matrix $A$ and vector $c$. This is a special case of a quadratic function, where the Hessian is given by
\[
\nabla^2 f(x) = A^TA.
\]
This gives us that $L_b = \norm{A_b}^2$ (where $A_b$ is the matrix containing the columns $b$ of $A$). In the special case where the block has a single element $j$, observe that $L_j = \sum_{i=1}^m a_{ij}^2$ (sum of the squared values in column $j$) so we do not need to solve an eigenvalue problem. When using a quadratic norm we can take $H_b = A_b^TA_b$ which similarly does not require solving an eigenvalue problem.

\subsection{Logistic Regression}

The likelihood of a single example in a logistic regression model is given by 
\[
	p(b_i \; | \; a_i,x) = \frac{1}{1 + \exp(-b_i x^T a_i)},
\]
where each $a_i \in \R^d$ and $b_i \in \{ -1 ,1 \}$. To maximize the likelihood over $m$ examples (sampled independently) we minimize the negative log-likelihood,
\[
	f(x) = \sum_{i = 1}^m \log(1 + \exp(-b_i x^T a_i)).
\]
Using $A$ as a matrix where row $i$ is given by $a_i^T$ and defining $h_i(x) = p(b_i \; | \; a_i,x)$, we have that
\begin{align*}
\nabla^2 f(x) & = \sum_{i=1}^m h_i(x)(1-h_i(x))a_ia_i^T\\
& \preceq 0.25\sum_{i=1}^m a_ia_i^T\\
& = 0.25A^TA.
\end{align*}
The generalized inequality above is the binary version of the bound of~\citet{bohning1992multinomial}. This bound can be derived by observing that $h_i(x)$ is in the range $(0,1)$, so the quantity $h_i(x)(1-h_i(x))$ has an upper bound of $0.25$. This result means that we can use $L_b = 0.25\norm{A_b}^2$ for block $b$, $L_j = 0.25\sum_{i=1}^m a_{ij}^2$ for single-coordinate blocks, and $H_b = 0.25A_b^TA_b$ if we are using a general quadratic norm (notice that computing $H_b$ is again cheaper than computing $L_b$).

\subsection{Multi-Class Logistic Regression}

The Hessian of the multi-class logistic regression objective~\eqref{eq:MLR} with respect to parameter vectors $x_c$ and $x_{c'}$ can be written as
\[
\frac{\partial^2}{\partial x_c \partial x_{c'}} f(X) = \sum_{i=1}^m h_{i,c}(X)(I(c = c') - h_{i,c'}(X))a_ia_i^T,
\]
where similar to the binary logistic regression case we have defined $h_{i,c} = p(c \; | \; a_i, X)$. This gives the full Hessian the form
\begin{align*}
\nabla^2 f(X) & = \sum_{i=1}^m H_i(X) \otimes a_ia_i^T,
\end{align*}
where we used $\otimes$ to denote the Kronecker product and where element $(c,c')$ of the $k$ by $k$ matrix $H_i(X)$ is given by $h_{i,c}(X)(I(c = c') - h_{i,c'}(X))$. The bound of~\citet{bohning1992multinomial} on this matrix is that
\[
H_i(X) \preceq \frac{1}{2}\left(I - \frac{1}{k}11^T\right),
\]
where $1$ is a vector of ones while recall that $k$ is the number of classes. Using this we have
\begin{align*}
\nabla^2 f(X) & \preceq \sum_{i=1}^m \frac{1}{2}\left(I - \frac{1}{k}11^T\right) \otimes a_ia_i^T\\
& = \frac{1}{2}\left(I - \frac{1}{k}11^T\right)\otimes \sum_{i=1}^m a_ia_i^T\\
& = \frac{1}{2}\left(I - \frac{1}{k}11^T\right) \otimes A^TA.
\end{align*}
As before we can take submatrices of this expression as our $H_b$, and we can take eigenvalues of the submatrices as our $L_b$. However, due to the $1/k$ factor we can actually obtain tighter bounds for sub-matrices of the Hessian that do not involve at least two of the classes. In particular, consider a sub-Hessian involving the variables only associated with $k'$ classes for $k' < k$. In this case we can replace the $k$ by $k$ matrix $(I - (1/k)11^T)$ with the $k'$ by $k'$ matrix $(I - (1/(k'+1))11^T)$. The ``$+1$'' added to $k'$ in the second term effectively groups all the other classes (whose variables are fixed) into a single class (the ``$+1$'' is included in B{\"o}hning's original paper as they fixe $x_k = 0$ and defines $k$ to be one smaller than the number of classes). This means (for example) that we can take $L_j = 0.25\sum_{i=1}^m a_{ij}^2$ as in the binary case rather than the slightly-larger diagonal element $0.5(1-1/k)\sum_{i=1}^m a_{ij}^2$ in the matrix above.\footnote{The binary logistic regression case can conceptually be viewed as a variation on the softmax loss where we fix $x_c=0$ for one of the classes and thus are always only updating variables from class. This gives the special case of $0.5(I - 1/(k+1)11^T)A^TA = 0.5(1 - 0.5)A^TA = 0.25A^TA$, the binary logistic regression bound from the previous section.}

\section{Derivation of GSD Rule}
\phantomsection
\label{append:GSd} 

In this section we derive a progress bound for twice-differentiable convex functions when we choose and update the block $b_k$ according to the GSD rule with $D_{b,i} = L_i\tau$ (where $\tau$ is the maximum block size). We start by using the Taylor series representation of $f(x^{k+1})$ in terms of $f(x^k)$ and some $z$ between $x^{k+1}$ and $x^k$ (keeping in mind that these only differ along coordinates in $b_k$),
\begin{align*}
f(x^{k+1}) & = f(x^k) + \langle \nabla f(x^k),x^{k+1}-x^k\rangle + \frac{1}{2}(x^{k+1}-x^k)^T\nabla_{b_kb_k}^2 f(z)(x^{k+1}-x^k)\\
& \leq  f(x^k) + \langle \nabla f(x^k),x^{k+1}-x^k\rangle + \frac{|b_k|}{2}\sum_{i \in b_k}\nabla_{ii}^2f(z)(x_i^{k+1}- x_i^k)^2\\
& \leq f(x^k) + \langle \nabla f(x^k),x^{k+1}-x^k\rangle + \frac{\tau}{2}\sum_{i \in b_k}\nabla_{ii}^2f(z)(x_i^{k+1}- x_i^k)^2\\
& \leq  f(x^k) + \langle \nabla f(x^k),x^{k+1}-x^k\rangle + \frac{\tau}{2}\sum_{i \in b_k}L_i(x_i^{k+1}- x_i^k)^2,
\end{align*}
where the first inequality follows from convexity of $f$ which implies that $\nabla_{b_kb_k}^2 f(x^k)$ is positive semi-definite and by Lemma~1 of Nesterov's coordinate descent paper~\citep{nesterov2012efficiency}. The second inequality follows from the definition of $\tau$ and the third follows from the definition of $L_i$. Now using our choice of $D_{b,i} = L_i\tau$ in the update we have for $i \in b_k$ that
\[
x_{i}^{k+1} = x_{i}^k - \frac{1}{L_i\tau}\nabla_i f(x^k),
\]
which yields
\begin{align*}
f(x^{k+1}) & \leq  f(x^k) - \frac{1}{2\tau}\sum_{i \in b_k}\frac{|\nabla_i f(x^k)|^2}{L_i}\\
& = f(x^k) - \frac{1}{2}\max_b \sum_{i \in b}\frac{|\nabla_i f(x^k)|^2}{L_i\tau}\\
& = f(x^k) - \norm{\nabla f(x^k)}_{\mathcal{B}}^2.
\end{align*}
The first equality uses that we are selecting $b_k$ using the GSD rule with $D_{b,i} = {L_i\tau}$ and the second inequality follows from the definition of of the mixed norm $\norm{\cdot}_\mathcal{B}$ from Section~\ref{subsec:convergence} with $H_b = D_b$. This progress bound implies that the convergence rate results in that section also hold.

\section{Efficiently Testing the Forest Property}
\phantomsection
\label{app:forest}

In this section we give a method to test whether adding a node to an existing forest maintains the forest property. In this setting our input is an undirected graph $G$ and a set of nodes $b$ whose induced subgraph $G_b$ forms a forest (has no cycles). Given a node $i$, we want to test whether adding $i$ to $b$ will maintain that the induced subgraph is acyclic. In this section we show how to do this in $O(p)$, where $p$ is the degree of the node $i$.

The method is based on the following  observations:
\begin{itemize}
\item  If the new node $i$ introduces a cycle, then it must be part of the cycle. This follows because $G_b$ is assumed to be acyclic, so no cycles can exist that do not involve $i$.
\item If $i$ introduces a cycle, we can arbitrarily choose $i$ to be the start and end point of the cycle.
\item If the new node $i$ has 1 or fewer neighbours in $b$, then it does not introduce a cycle. With no neighbours it clearly can not be part of a cycle. With one neighbour, we would have to traverse its one edge more than once to have it start and end a path.
\item If the new node $i$ has at least 2 neighbours in $b$ that are part of the same tree, then $i$ introduces a cycle. Specifically, we can construct a cycle as follows: we start at node $i$, go to one of its neighbours, follow a path through the tree to another one of its neighbours in the same tree (such a path exists because trees are connected by definition), and then return to node $i$.
\item If the new node $i$ has at least 2 neighbours in $b$ but they are all in different trees, then $i$ does not introduce a cycle. This is similar to the case where $i$ has only one edge: any path that starts and ends at node $i$ would have to traverse one of its edges more than once (because the disjoint trees are not connected to each other).
\end{itemize}
The above cases suggest that to determine whether adding node $i$ to the forest $b$ maintains the forest property, we only need to test whether node $i$ is connected to two nodes that are part of the same tree in the existing forest. We can do this in $O(p)$ using the following data structures:
\begin{enumerate}
\item For each of the $n$ nodes, a list of the adjacent nodes in $G$. 
\item A set of $n$ labels in $\{0,1,2,\dots,t\}$, where $t$ is the number of trees in the existing forest. This number is set to $0$ for nodes that are not in $b$, is set to $1$ for nodes in the first tree, is set to $2$ for nodes in the second tree, and so on. 
\end{enumerate}
Note that there is no ordering to the labels $\{1,2,\dots,t\}$, each tree is just assigned an arbitrary number that we will use to determine if nodes are in the same tree.
We can find all neighbours of node $i$ in $O(p)$ using the adjacency list, and we can count the number of neighbours in each tree in $O(p)$ using the tree numbers. If this count is at least $2$ for any tree then the node introduces a cycle, and otherwise it does not.

In the algorithms of Section~\ref{sec:treePart} and~\ref{sec:treeApprox}, we also need to update the data structures after adding a node $i$ to $b$ that maintains the forest property. For this update we need to consider three scenarios:
\begin{itemize}
\item If the node $i$ has one neighbour in $b$, we assign it the label of its neighbour.
\item If the node $i$ has no neighbours in $b$, we assign it the label $(t+1)$ since it forms a new tree.
\item If the node $i$ has multiple neighbours in $b$, we need to merge all the trees it is connected to.
\end{itemize}
The first two steps cost $O(1)$, but a naive implementation of the third step would cost $O(n)$ since we could need to re-label almost all of the nodes. Fortunately, we can reduce the cost of this merge step to $O(p)$. This requires a relaxation of the condition that the labels represent disjoint trees. Instead, we only require that nodes with the same label are part of the same tree. This allows multiple labels to be associated with each tree, but using an extra data structure we can still determine if two labels are part of the same tree:
\begin{enumerate}
\setcounter{enumi}{2}
\item A list of $t$ numbers, where element $j$ gives the \emph{minimum node number} in the tree that $j$ is part of.
\end{enumerate}
Thus, given the labels of two nodes we can determine whether they are part of the same tree in $O(1)$ by checking whether their minimum node numbers agree. Given this data structure, the merge step is simple: we arbitrarily assign the new node $i$ to the tree of one of its neighbours, we find the minimum node number among the $p$ trees that need to be merged, and then we use this as the minimum node number for all $p$ trees. This reduces the cost to $O(p)$.

Giving that we can efficiently test the forest property in $O(p)$ for a node with $p$ neighbours, it follows that the total cost of the greedy algorithm from Section~\ref{sec:treeApprox} is $O(n \log n + |E|)$ given the gradient vector and adjacency lists. The $O(n \log n)$ factor comes from sorting the gradient values, and \blu{the sum of the $p$ values is $2|E|$}. If this cost is prohibitive, one could simply restrict the number of nodes that we consider adding to the forest to reduce this time.

\section{Functions Bound Iterates under Strong-Convexity}
\phantomsection
\label{append:strong}

Defining $F(x)$ as $F(x) = f(x) + g(x)$ in the ``smooth plus separable non-smooth'' setting of~\eqref{eq:compositeproblem}, existing works on cyclic~\citep{beck2013convergence} and greedy selection~\citep{nutini2015} of $i_k$ within proximal coordinate descent methods imply that
\begin{equation}
\phantomsection
\label{eq:Frate}
F(x^k) - F(x^*) \leq \rho^k[F(x^0) -  F(x^*)],
\end{equation}
for some $\rho < 1$ when $f$ is strongly-convex. Note that strong-convexity of $f$ implies strong-convexity of $F$, so we have that
\[
F(y) \geq F(x) + \langle s,y-x\rangle + \frac{\mu}{2}\norm{y-x}^2,
\]
where $\mu$ is the strong-convexity constant of $f$ and $s$ is any subgradient of $F$ at $x$. Taking $y=x^k$ and $x = x^*$ we obtain that
\begin{equation}
\phantomsection
\label{eq:Xrate}
F(x^k) \geq F(x^*) + \frac{\mu}{2}\norm{x^k - x^*}^2,
\end{equation}
which uses that $0$ is in the sub-differential of $F$ at $x^*$.
Thus we have that
\[
\norm{x^k - x^*}^2 \leq \frac{2}{\mu}[F(x^k) - F(x^*)] \leq \frac{2}{\mu} \rho^k[F(x^0) - F(x^*)],
\]
which is the type of convergence rate we assume in Section~\ref{sec:maniRate} with $\gamma = \frac{2}{\mu}[F(x^0) - F(x^*)]$. 
%Note that convergence rates of the form~\eqref{eq:Frate} and relationships between functions and iterates of the form~\eqref{eq:Xrate} hold under weaker conditions than strong-convexity of $f$ that includes SVM and LASSO problems. We refer to Karimi et al.~\citep{karimi2016} and the references in that work for a discussion of weaker conditions.

\section{Full Experimental Results}
\phantomsection
\label{append:experiments}

In this section we first provide details on the datasets, and then we present our complete set of experimental results.

\subsection{Datasets}

We considered these five datasets:
\begin{enumerate}[label=\Alph*]
\item A least squares problem with a data matrix $A \in \R^{m \times n}$ and  target $b \in \R^{m}$,
\[
\argmin{x \in \R^n} \frac{1}{2}\norm{Ax-b}^2.
\]
We set $A$ to be an $m$ by $n$ matrix with entries sampled from a $\mathcal{N}(0,1)$ distribution (with $m=1000$ and $n=10000$). We then added 1 to each entry (to induce a dependency between columns), multiplied each column by a sample from $\mathcal{N}(0,1)$ multiplied by ten (to induce different Lipschitz constants across the coordinates), and only kept each entry of $A$ non-zero with probability $10\log(m)/m$.
We set $b = Ax + e$, where the entries of $e$ were drawn from a $\mathcal{N}(0,1)$ distribution while we set 90\% of $x$ to zero and drew the remaining values from a $\mathcal{N}(0,1)$ distribution.
\item A binary logistic regression problem of the form
\[
\argmin{x \in \R^n} \sum_{i=1}^n \log(1+\exp(-b_ix^Ta_i)).
\]
We use the data matrix $A$ from  the previous dataset (setting row $i$ of $A$ to $a_i^T$), and $b_i$ to be the sign of $x^Ta_i$ with the $x$ used to generate the previous dataset. We then flip the sign of each entry in $b$ with probability $0.1$ to make the dataset non-separable.
\item A multi-class logistic regression problem of the form
\[
\argmin{x \in \R^{d \times k}} \sum_{i=1}^m\left[-x_{b_i}^Ta_i + \log\left(\sum_{c=1}^k\exp(x_c^Ta_i)\right)\right],
\]
see~\eqref{eq:MLR}.
We generate a $1000$ by $1000$ matrix $A$ as in the previous two cases. To generate the $b_i \in \{1,2,\dots,k\}$ (with $k = 50$), we compute $AX + E$ where the elements of the matrices $X \in \R^{d \times k}$ and $E \in \R^{m \times k}$ are sampled from a standard normal distribution. We then compute the maximum index in each row of that matrix as the class labels.
\item A label propagation problem of the form
\[
	\min_{x_i \in S' } \frac{1}{2} \sum_{i=1}^n \sum_{j=1}^n w_{ij} (x_i - x_j)^2,
\]
where $x$ is our label vector, $S$ is the set of labels that we do know (these $x_i$ are set to a sample from a normal distribution with a variance of $100$), $S'$ is the set of labels that we do not know, and $w_{ij} \ge 0$ are the weights assigned to each $x_i$ describing how strongly we want the labels $x_i$ and $x_j$ to be similar. We set the non-zero pattern of the $w_{ij}$ so that the graph forms a 50 by 50 lattice-structure (setting the non-zero values to $10000$). We labeled 100 points, leading to a problem with 2400 variables but where each variable has at most 4 neighbours in the graph.
\item Another label propagation problem  for semi-supervised learning in the `two moons' dataset~\citep{zhou2004learning}, which is a binary label propagation problem ($x_i \in [-1,1]$). We generate $2000$ samples from this dataset, randomly label 100 points in the data, and connect each node to its five nearest neighbours (using $w_{ij} = 1$). This results in a very sparse but unstructured graph.
\end{enumerate}

\subsection{Greedy Rules with Gradients Updates}
\phantomsection
\label{append:greedyGradient}

\begin{figure}[ht]
\centering
\includegraphics[width=0.95\textwidth]{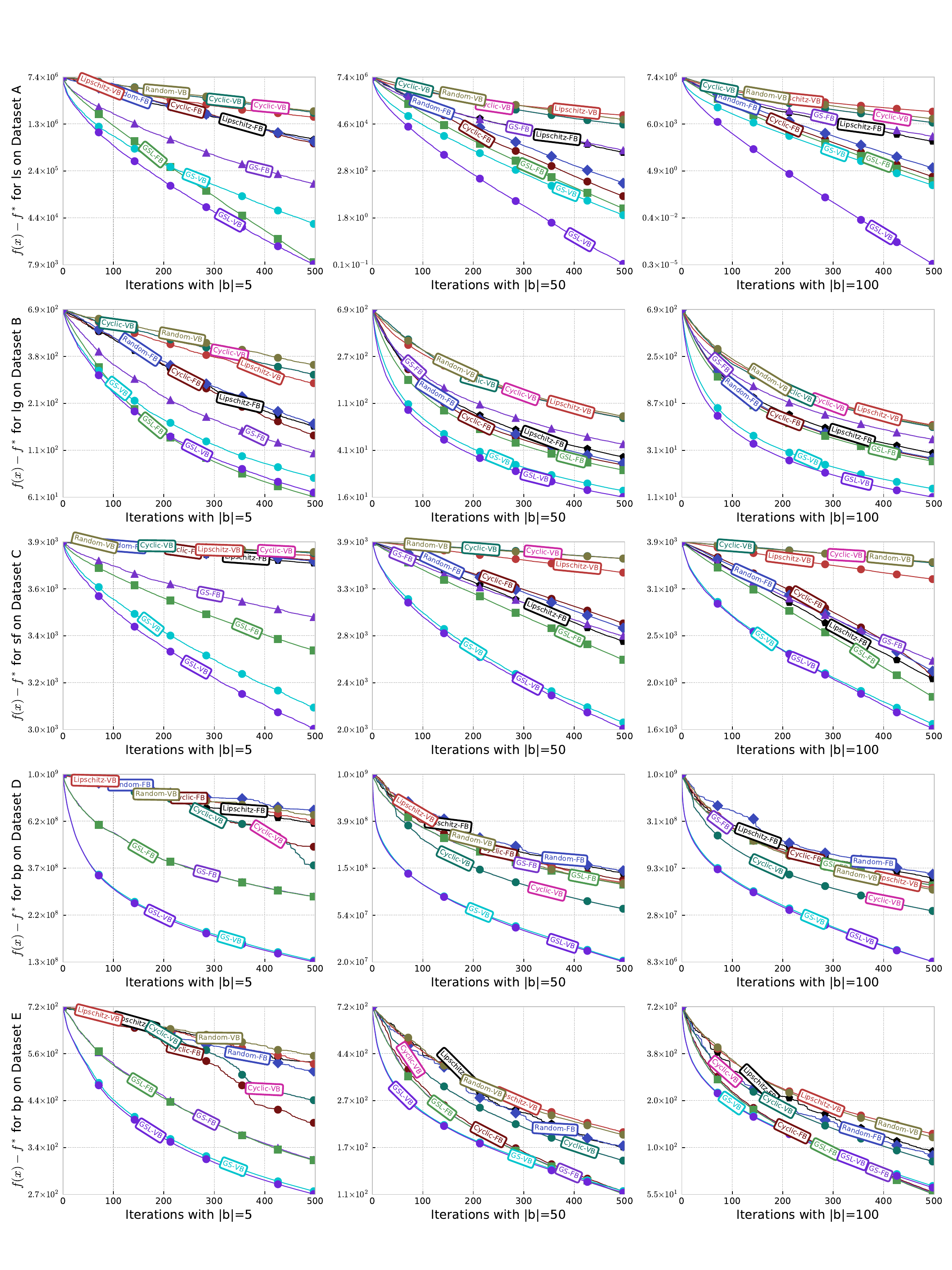}
\caption{Comparison of different random and greedy block selection rules on five different problems (rows) with three different blocks (columns) when using gradient updates.}
\phantomsection
\label{fig:exp1a}
\end{figure}

\begin{figure}[ht]
\centering
\includegraphics[width=0.95\textwidth]{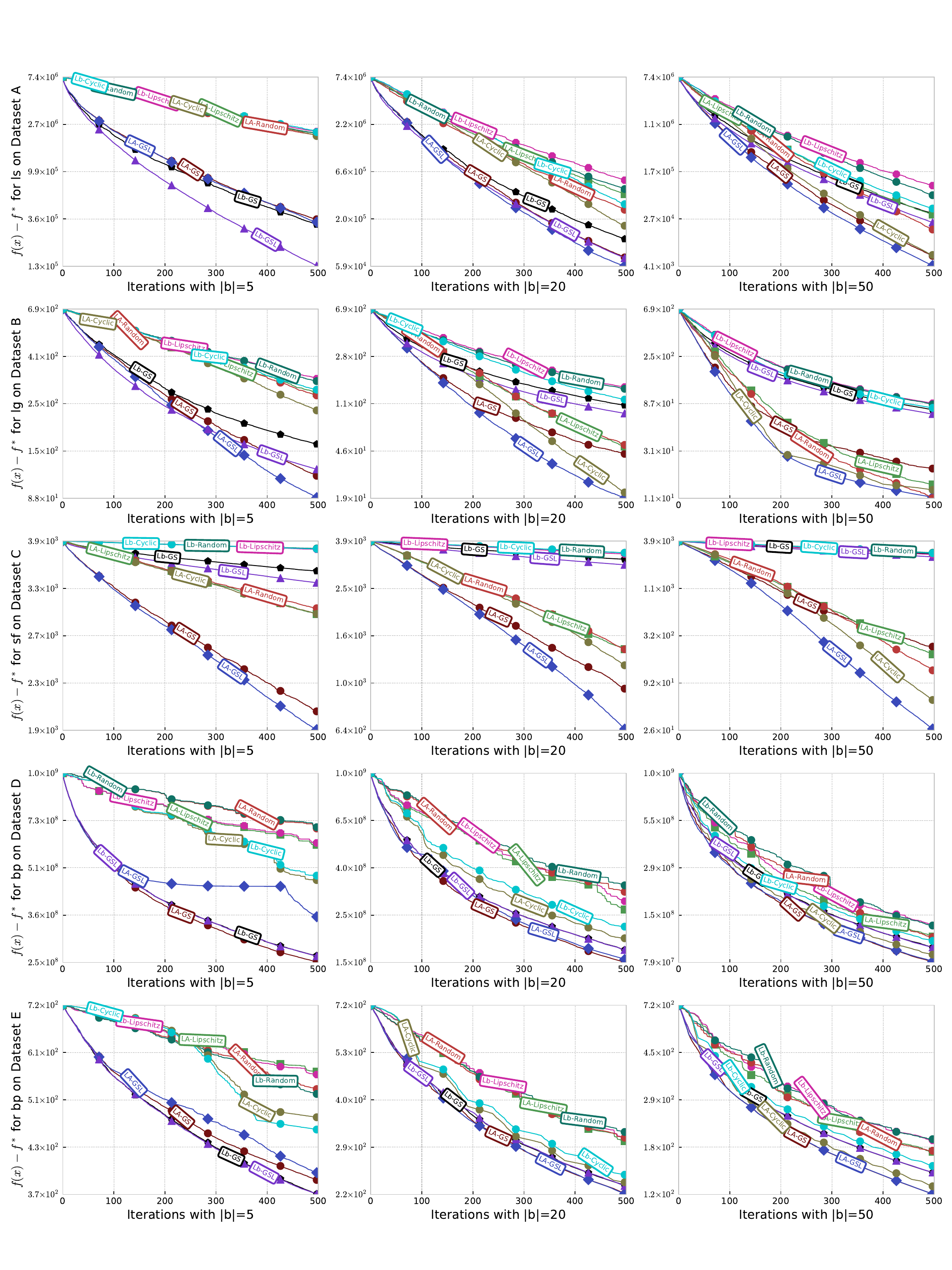}
\caption{Comparison of different random and greedy block selection rules with gradient updates and fixed blocks, using two different strategies to estimate $L_b$.}
\phantomsection
\label{fig:exp3a}
\end{figure}

In Figure~\ref{fig:exp1a} we show the performance of the different methods from Section~\ref{sec:expGrad} on all five datasets with three different block sizes. In Figure~\ref{fig:exp3a} we repeat the experiment but focusing only on the FB methods that do not incorporate the Lipschitz constant. For each FB method, we plot the performance using our upper bounds on $L_b$ as the step-size (Lb) and using the Lipschitz approximation procedure from Section~\ref{sec:LA} (LA). Here we see the LA methods improves performance
 when using large block sizes and in cases where the global $L_b$ bound is not tight. Interestingly, the cyclic and random methods seem to perform as well as greedy when using LA and a large block size (we expect this would only be true for FB).

Our third experiment also focused on the FB methods, but considered different ways to partition the variables into fixed blocks. We considered three approaches:
\begin{enumerate}
\item Order: we partition the variables based on their numerical order (which is similar to using a random order for dataset except Dataset D, where this method groups variables that adjacent in the lattice).
\item Avg: we compute the coordinate-wise Lipschitz constants $L_i$, and place the largest $L_i$ with the smallest $L_i$ values so that the average $L_i$ values are similar across the blocks.
\item Sort: we sort the $L_i$ values and place the largest values together (and the smallest values together).
\end{enumerate}
We compared many variations on cyclic/random/greedy rules with gradient or matrix updates. In the case of greedy rules with gradient updates, we found that the Sort method tended to perform the best while the Order method tended to perform the worst (see Figure~\ref{fig:exp4b}). When using matrix updates or when using cyclic/randomized rules, we found that no partitioning strategy dominated other strategies.

\begin{figure}[ht]
\centering
\includegraphics[width=0.95\textwidth]{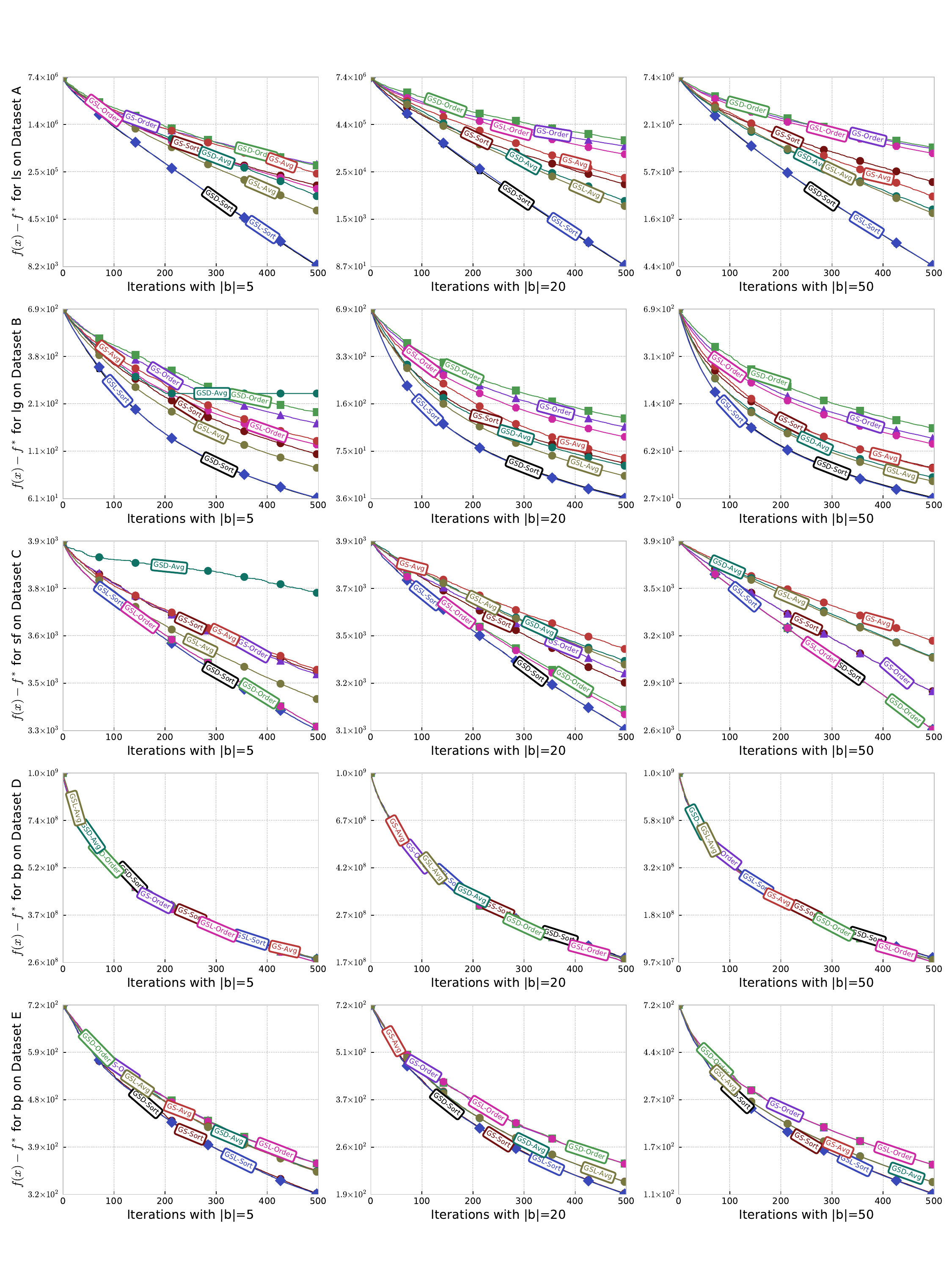}
\caption{Comparison of different random and greedy block selection rules with gradient updates and fixed blocks, using three different ways to partition the variables into blocks.}
\phantomsection
\label{fig:exp4b}
\end{figure}

\subsection{Greedy Rules with Matrix and Newton Updates}
\phantomsection
\label{append:greedyNewtonMatrix}

\begin{figure}[ht]
\centering
\includegraphics[width=0.95\textwidth]{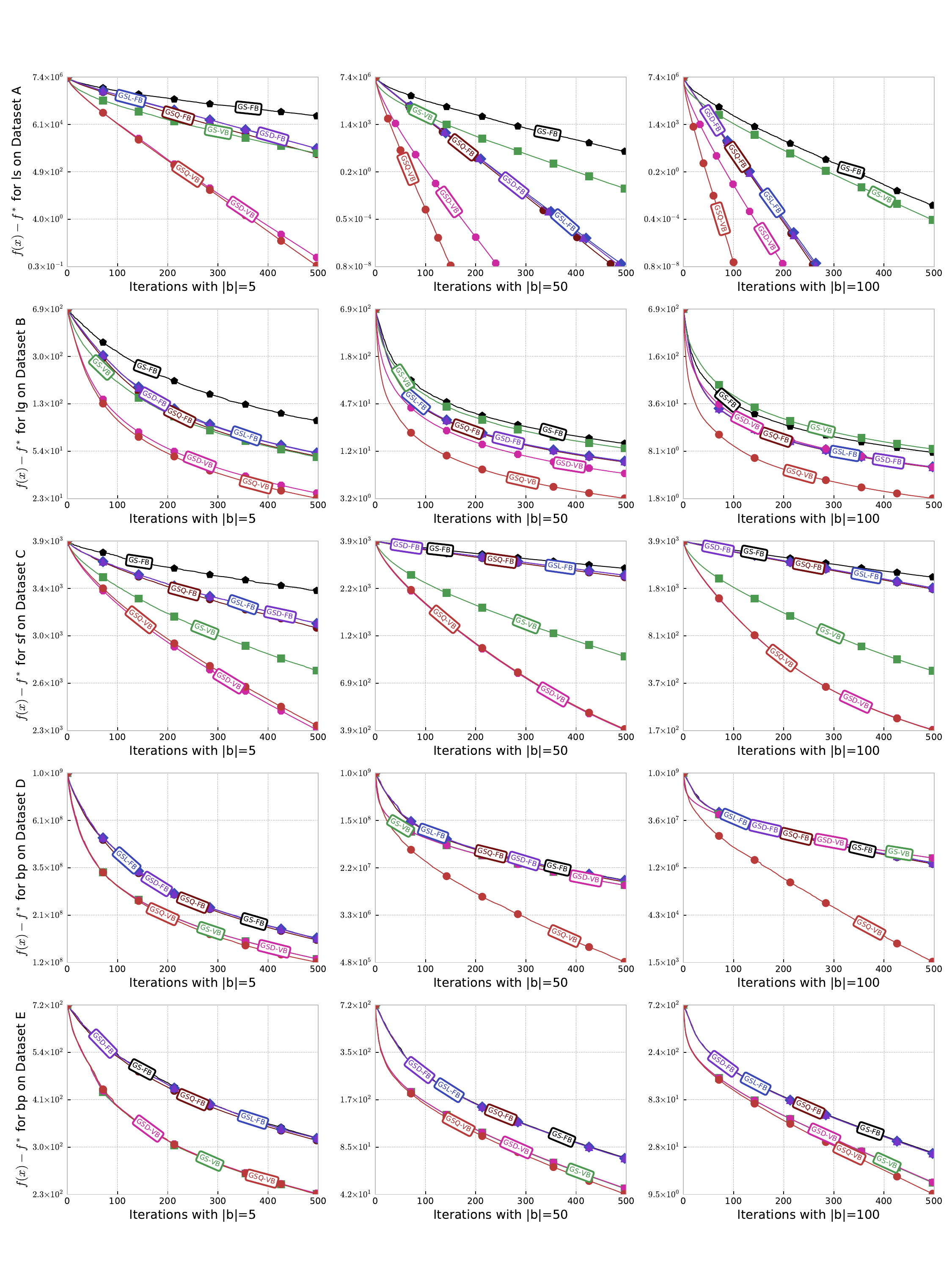}
\caption{Comparison of different  greedy block selection rules when using matrix updates.}
\phantomsection
\label{fig:exp2a}
\end{figure}

\begin{figure}[ht]
\centering
\includegraphics[width=\textwidth]{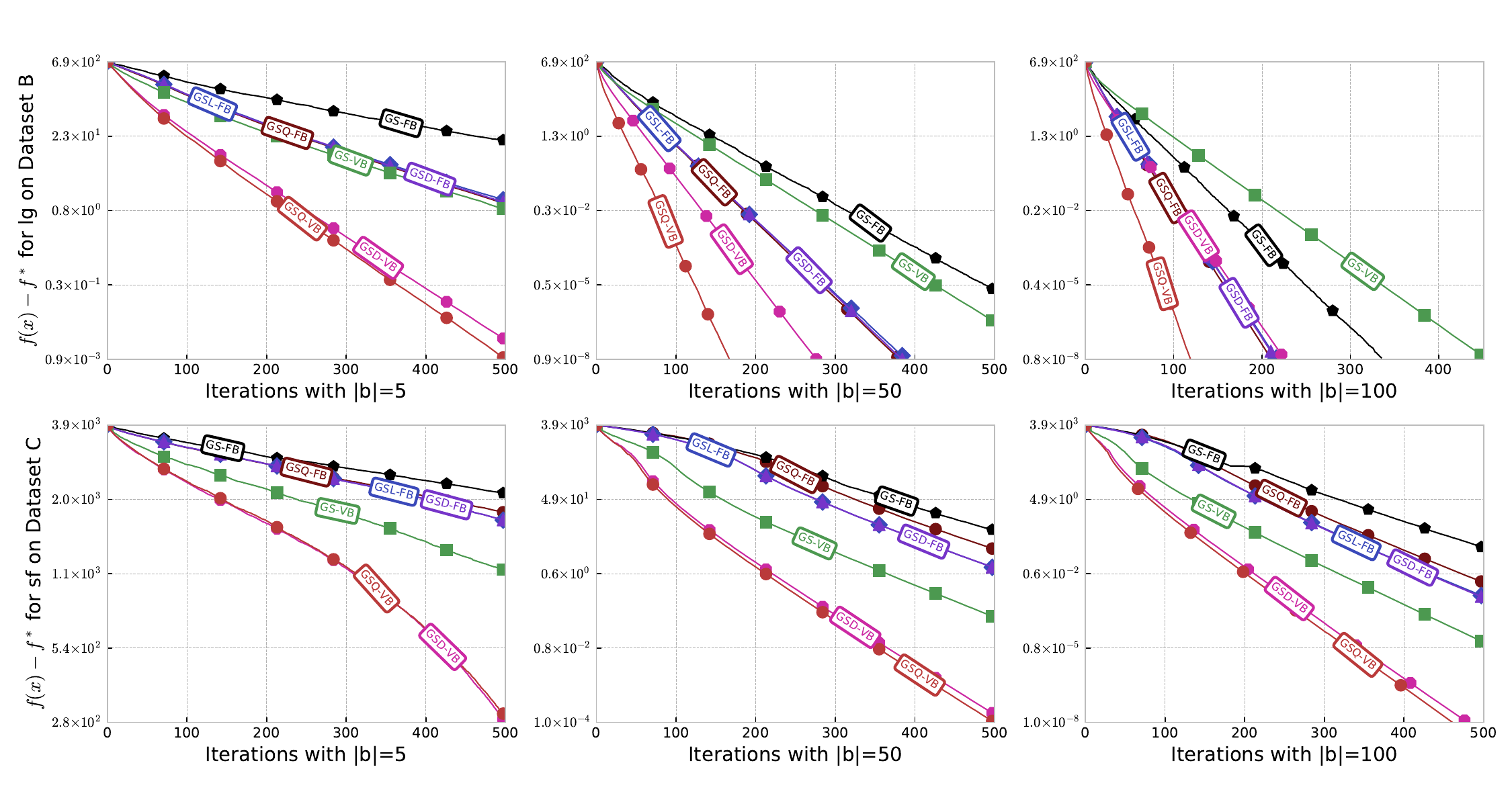}
\caption{Comparison of different  greedy block selection rules when using Newton updates and a line-search.}
\phantomsection
\label{fig:exp2LS}
\end{figure}

In Figure~\ref{fig:exp2a} we show the performance of the different methods from Section~\ref{sec:expNewt} on all five datasets with three different block sizes. In Figure~\ref{fig:exp2LS} we repeat this experiment on the two non-quadratic problems, using the Newton direction and a line-search rather than matrix updates. We see that using Newton's method significantly improves performance over matrix updates.

\subsection{Proximal Updates using Random Selection}
\phantomsection
\label{append:activeSetExp}

\begin{figure}[ht]
\centering
\includegraphics[width=\textwidth]{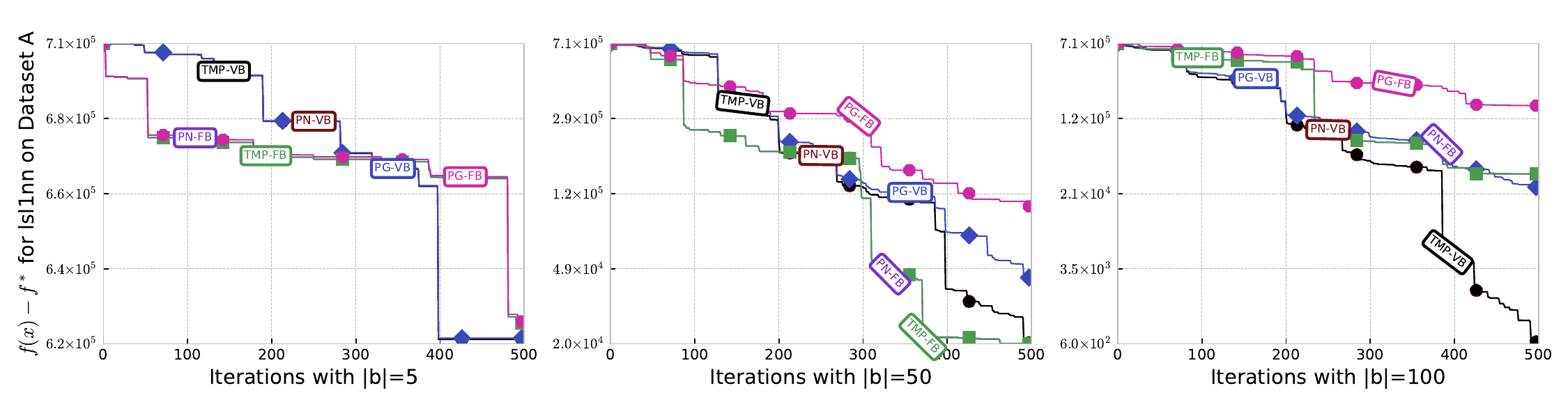}
\caption{Comparison of different updates when using random fixed and variable blocks of different sizes on a non-smooth problem.}
\phantomsection
\label{fig:exp7b}
\end{figure}

In Figure~\ref{fig:exp7a} we compare the performance of a projected gradient update with $L_b$ step-size, a projected Newton (PN) solver with line-search and the two-metric projection (TMP) update when using greedy selected fixed (FB) and variable (VB) blocks of different sizes ($|b|={5,50,100}$). In Figure~\ref{fig:exp7b} we repeat this experiment using randomly selected blocks. We see that for such a sparse problem the random block selection rules suffer from selecting variables that are already zero/active, as seen by the step-like nature of the results.

%\end{document}

\clearpage
\vskip 0.2in
\bibliography{bib}

\end{document}